\documentclass[reqno]{amsart}
\usepackage{preprintStyle}
\pdfoutput=1

\usepackage{pgfplots}
\pgfplotsset{compat=newest}
\usepgfplotslibrary{groupplots}
\usepgfplotslibrary{dateplot}
\usepgfplotslibrary{units}
\usetikzlibrary{spy,backgrounds}
\usepackage{pgfplotstable}
\usepackage{bm}
\usepackage{empheq}
\usepackage{dsfont}
\usepackage{makecell}

\usepackage{graphicx}
\usepackage{color}
\usepackage{circuitikz}
\usepackage{tikz}
\usetikzlibrary{calc,positioning,shapes}
\usetikzlibrary{patterns,decorations.pathmorphing,decorations.markings}
\usepackage{pgfplots}
\pgfplotsset{compat=newest}
\usepackage[margin=10pt,font=small,labelfont=bf,labelsep=endash]{caption}
\usepackage{subcaption}

\usetikzlibrary{intersections, backgrounds}
\usetikzlibrary{circuits}
\usetikzlibrary{circuits.ee.IEC}
\pgfplotsset{compat=newest}

\definecolor{color0}{rgb}{0.12156862745098,0.466666666666667,0.705882352941177}
\definecolor{color1}{rgb}{1,0.498039215686275,0.0549019607843137}
\definecolor{color2}{rgb}{0.172549019607843,0.627450980392157,0.172549019607843}
\definecolor{color3}{rgb}{0.83921568627451,0.152941176470588,0.156862745098039}
\definecolor{color4}{rgb}{0.580392156862745,0.403921568627451,0.741176470588235}
\definecolor{color5}{rgb}{0,0,0}

\definecolor{mycolor1}{rgb}{0.00000,0.44700,0.74100}
\definecolor{mycolor2}{rgb}{0.85000,0.32500,0.09800}
\definecolor{mycolor3}{rgb}{0.92900,0.69400,0.12500}
\definecolor{mycolor4}{rgb}{0.46600,0.67400,0.18800}
\definecolor{mycolor5}{rgb}{0.49400,0.18400,0.55600}

\newcommand{\lineWidth}{1.2pt}
\newcommand{\imageWidth}{2.0in}
\newcommand{\imageHeight}{1.8in}

\definecolor{color2}{rgb}{0.172549019607843,0.627450980392157,0.172549019607843}
\newcommand*\closure[1]{\overline{#1}}

\newcommand{\ReA}[1]{{\color{black}#1}} 
\DeclareMathOperator{\spann}{span}
\newcommand{\Li}{\mathcal{L}}

\newcommand{\X}{{\mathscr X}}
\newcommand{\U}{{\mathscr U}}

\newcommand{\bx}{{\bm x}}
\newcommand{\bs}{{\bm s}}
\newcommand{\R}{\mathbb R}
\newcommand{\N}{\mathbb N}
\newcommand{\Uad}{\U_{\mathsf{ad}}}
\newcommand{\timeInt}{\mathbb{T}}
\newcommand{\state}{\theta}

\newcommand{\adState}{p}
\newcommand{\desired}[1]{#1_{\mathrm{d}}}
\newcommand{\inpVar}{u}
\newcommand{\inpVarDes}{\desired{u}}
\newcommand{\inpVarDim}{\rho}
\newcommand{\outVar}{y}
\newcommand{\outVarDes}{\desired{y}}
\newcommand{\hatoutVarDes}{\desired{\hat y}}
\newcommand{\outVarDim}{p}
\newcommand{\lonereg}{c_{L^1}}
\newcommand{\dt}{\,\mathrm{d}t}
\newcommand{\ddt}{\tfrac{\mathrm{d}}{\mathrm{d}t}}
\newcommand{\ds}{\,\mathrm{d}s}
\newcommand{\finalTime}{\overline{T}}

\newcommand{\coerAV}{\eta_V}
\newcommand{\coerAH}{\eta_H}
\newcommand{\contA}{\gamma_a}
\newcommand{\coerM}{\eta_m}
\newcommand{\coerMi}[1]{\eta_{m_{#1}}}
\newcommand{\contM}{\gamma_m}
\newcommand{\contMi}[1]{\gamma_{m_{#1}}}
\newcommand{\contB}{\gamma_b}
\newcommand{\contC}{\gamma_c}
\newcommand{\mOper}{\calM}
\newcommand{\aOper}{\calA}
\newcommand{\bOper}{\calB}
\newcommand{\cOper}{\calC}

\newcommand{\kOper}{\calK}
\newcommand{\difOper}{\calL}
\newcommand{\solOper}{\calS_{0}}
\newcommand{\solper}{\calS}
\newcommand{\solA}{\mathcal{A}}

\newcommand{\errortol}{\epsilon}
\newcommand{\initBound}{\initBoundn{n}}
\newcommand{\initBoundn}[1]{\Delta_{t_{#1}}}
\newcommand{\intermediateInitState}{\intermediateState_n}
\newcommand{\intermediateInitStateRed}{\stateRed_n}
\newcommand{\optState}{\state^\star}
\newcommand{\optAdState}{\adState^\star}
\newcommand{\optInpVar}[1]{\inpVar_{#1}^\star}
\newcommand{\optInpVarN}{\optInpVar{n}}
\newcommand{\optOutVar}{\outVar^\star}
\newcommand{\stateMPC}{\state_{\mathsf{mpc}}}
\newcommand{\inpVarMPC}{\inpVar_{\mathsf{mpc}}}
\newcommand{\outVarMPC}{\outVar_{\mathsf{mpc}}}

\newcommand{\switch}{\sigma}
\newcommand{\nrModes}{L}
\newcommand{\addSwitching}{\mathbb{S}}
\newcommand{\nrSwitches}{N}

\newcommand{\reduce}[1]{\tilde{#1}}

\newcommand{\Vred}{\reduce{V}}
\newcommand{\projVred}{\Pi_{\Vred}}

\newcommand{\Xred}{\reduce{\X}}

\newcommand{\BoundControlA}{\Delta_{A}}
\newcommand{\BoundControlB}{\Delta_{B}}
\newcommand{\BoundControlBtil}{\tilde\Delta_{B}}
\newcommand{\BoundControlAtil}{\tilde\Delta_{A}}
\newcommand{\err}{e}
\newcommand{\intermediateErr}{\check \err}
\newcommand{\intermediateErrMPC}{\check \err_{\mathsf{mpc}}}
\newcommand{\redErr}{\reduce \err}
\newcommand{\redErrMPC}{\reduce \err_{\mathsf{mpc}}}

\newcommand{\residual}{R}
\newcommand{\addErr}{\epsilon}
\newcommand{\adResidual}{Q}
\newcommand{\stateRed}{\reduce{\state}}
\newcommand{\stateDimRed}{r}
\newcommand{\adStateRed}{\reduce{\adState}}
\newcommand{\inpVarRed}{\reduce{\inpVar}}
\newcommand{\outVarRed}{\reduce{\outVar}}
\newcommand{\Jhat}{\calJ}
\newcommand{\basis}{\psi}
\newcommand{\PODbasis}{\bar{\basis}}
\newcommand{\PODeigVal}{\bar{\lambda}}

\newcommand{\optStateRed}{\stateRed^\star}
\newcommand{\optAdStateRed}{\adStateRed^\star}
\newcommand{\optInpVarRed}[1]{\reduce{\inpVar}_{#1}^\star}
\newcommand{\optInpVarRedN}{\optInpVarRed{n}}
\newcommand{\optOutVarRed}{\outVarRed^\star}
\newcommand{\intermediateState}{\check{\state}}
\newcommand{\intermediateStateMPC}
{\check{\state}_{\mathsf{mpc}}}
\newcommand{\intermediateOutVarMPC}
{\check{\outVar}_{\mathsf{mpc}}}
\newcommand{\intermediateInpVarMPC}
{\check{\inpVar}_{\mathsf{mpc}}}
\newcommand{\intermediateAdStateA}{{\hat \adState}}
\newcommand{\intermediateAdStateB}{\check{\adState}}
\newcommand{\intermediateOutVar}{\check{\outVar}}
\newcommand{\inpVarRedMPC}{\inpVarRed_{\mathsf{mpc}}}
\newcommand{\outVarRedMPC}{\outVarRed_{\mathsf{mpc}}}
\newcommand{\stateRedMPC}{{\stateRed}_{\mathsf{mpc}}}

\newcommand{\advection}{d}

\newcommand{\calA}{\mathcal{A}}
\newcommand{\calB}{\mathcal{B}}
\newcommand{\calC}{\mathcal{C}}

\newcommand{\calE}{\mathcal{E}}

\newcommand{\calI}{\mathcal{I}}
\newcommand{\calJ}{\mathcal{J}}
\newcommand{\calK}{\mathcal{K}}
\newcommand{\calL}{\mathcal{L}}
\newcommand{\calM}{\mathcal{M}}

\newcommand{\calO}{\mathcal{O}}

\newcommand{\calR}{\mathcal{R}}
\newcommand{\calS}{\mathcal{S}}
\newcommand{\calT}{\mathcal{T}}

\newcommand{\snapshots}{\mathsf s}

\newcommand{\T}{\ensuremath\mathsf{T}}

\newcommand{\nrSnapshots}{M}

\newcommand{\abbr}[1]{\textsf{#1}\xspace}
\newcommand{\DAE}{\abbr{DAE}}

\newcommand{\PDE}{\abbr{PDE}}
\newcommand{\PDEs}{\abbr{PDEs}}

\newcommand{\FOM}{\abbr{FOM}}
\newcommand{\ROM}{\abbr{ROM}}
\newcommand{\ROMs}{\abbr{ROMs}}
\newcommand{\MPC}{\abbr{MPC}}
\newcommand{\MOR}{\abbr{MOR}}
\newcommand{\LQR}{\abbr{LQR}}
\newcommand{\OCP}{\abbr{OCP}}
\newcommand{\OCPs}{\abbr{OCPs}}
\newcommand{\BT}{\abbr{BT}} 
\newcommand{\POD}{\abbr{POD}} 
\newcommand{\CIM}{\abbr{CIM}}
\newcommand{\LS}{\abbr{LS}}
\newcommand{\SVD}{\abbr{SVD}}

\newcommand{\cost}{C}
\newcommand{\costFOM}{\cost_{\FOM}}
\newcommand{\costFOMROM}{\cost_{\FOM-\ROM}}
\newcommand{\costROMROM}{\cost_{\ROM-\ROM}}
\newcommand{\costUpdate}{\cost_{\textsf{update}}}
\newcommand{\costEst}{\cost_{\textsf{estimate}}}
\newcommand{\costPOD}{\cost_{\textsf{POD}}}

\title[Certified MPC for switched evolution equations using MOR]{Certified model predictive control for switched evolution equations using Model Order Reduction}
\author[M.~Kartmann \and M.~Manucci \and B.~Unger \and S. Volkwein]{Michael Kartmann${}^{\dagger}$ \and Mattia Manucci${}^\star$ \and Benjamin Unger${}^\ddagger$ \and Stefan Volkwein${}^\dagger$}

\address{${}^{\dagger}$ FB Mathematik und Statistik, Universität Konstanz, Universitätsstr. 10, 78464 Konstanz, Germany}
\email{\{michael.kartmann,stefan.volkwein\}@uni-konstanz.de}
\address{${}^{\star}$ Stuttgart Center for Simulation Science (SC SimTech), University of Stuttgart, Universit\"{a}tsstr.~32, 70569 Stuttgart, Germany}
\email{mattia.manucci@simtech.uni-stuttgart.de}
\address{${}^{\ddagger}$ Institute for Applied and Numerical Mathematics, Karlsruhe Institute of Technology, Englerstr.~2, 76131 Karlsruhe, Germany}
\email{benjamin.unger@kit.edu}

\begin{document}

\begin{abstract}
	We present a model predictive control (\MPC) framework for linear switched evolution equations arising from a parabolic partial differential equation (\PDE). First-order optimality conditions for the resulting finite-horizon optimal control problems are derived. The analysis allows for the incorporation of convex control constraints and sparse regularization. 
    Then, to mitigate the computational burden of the \MPC procedure, we employ Galerkin reduced-order modeling (\ROM) techniques to obtain a low-dimensional surrogate for the state-adjoint systems. We derive recursive a-posteriori estimates for the \ROM feedback law and the \ROM-\MPC closed-loop state and show that the \ROM-\MPC trajectory evolves within a neighborhood of the true \MPC trajectory, whose size can be explicitly computed and is controlled by the quality of the \ROM. Such estimates are then used to formulate two \ROM-\MPC algorithms with closed-loop certification.
\end{abstract}

\maketitle



\section{Introduction}
\label{sec:intro}
We introduce a framework for \emph{model predictive control} (\MPC) for linear switched parabolic \emph{partial differential equations} ({\PDE}s) using Galerkin \emph{model order reduction} (\MOR) techniques. In more detail, we focus on the first-order optimality conditions for the resulting finite-horizon \emph{optimal control problems} (\OCPs) with a proper, lower-semicontinuous, convex regularization function, which can then be used for gradient-based optimization techniques. Since gradients need to be computed many times in the receding horizon problems, we employ \MOR to obtain a low-dimensional surrogate for the optimality systems to mitigate the computational burden conspiring against real-time applicability. Our main goal is to certify the approximation error of the \ROM-\MPC feedback law and closed-loop trajectory. 
\subsection{Main results}
After discussing the well-posedness of the switched system in \Cref{prop:wellPosedness} and the associated finite-time \OCP in \Cref{thm:wellPosednessOCP}, we derive the first-order necessary optimality conditions in \Cref{thm:FONC:MPC}. Interestingly, switching in the bilinear form corresponding to the time derivative directly yields a switched adjoint system with state transition maps. The analysis allows for general convex, lower semicontinuous regularizations, which include control box constraints and $L^1$ regularizations as a special case. Toward a computable approximation, we introduce a Galerkin-reduced switched system and the corresponding \OCP in \Cref{sec:approximation}. To verify the approximation, we present error bounds for the optimal control under perturbed initial values, accounting for the error accumulated from previous \MPC iterations, the state, and the adjoint state in \Cref{thm:apostErrorControl}, \Cref{thm:errorBoundStateEnergy,thm:errorBoundAdjointEnergy}, and \Cref{cor:errBoundOptimalControl} respectively, extending the ideas presented in, e.g., \cite{DBLP:journals/coap/TroltzschV09,bader2016certified,AliH18}. In \Cref{subsec:applicationToMPC}, we apply these results to \MPC and formulate two reduced online-adaptive \MPC algorithms (\Cref{alg:ROM-MPC,alg:ROMROM-MPC}) with recursive a-posteriori certification for the feedback law and the closed-loop state. To demonstrate our framework, we use the \emph{proper orthogonal decomposition}~(\POD) to obtain a \emph{reduced-order model} (\ROM) in \Cref{subsec:POD}. Throughout the manuscript, we illustrate our theoretical findings with a sample application of two adjacent rooms, where the switching is due to an open or closed door connecting the rooms; see \Cref{subsection:example,subsec:exampleContinued,subsec:numerics:openloop,subsec:numerics:MPC}. 
\subsection{Literature-review and state-of-the-art}
The research subject at hand combines different research areas, namely optimal control and \MPC for switched systems, \MPC and \MOR, and \MOR for switched systems. To the best of our knowledge, our paper is the first contribution combining all three areas while emphasizing the infinite-dimensional setting. We thus provide a literature review for the individual topics.

\subsubsection{Optimal control and \texorpdfstring{\MPC}{MPC} for switched systems}
The literature on \OCPs for switched systems is extensive; for this reason, we mention only some references closer to our framework and refer the interested reader to the references therein. The work \cite{RieKUZ99} uses a maximum principle to derive the necessary conditions for a general class of hybrid control systems using a Hamiltonian function. In
\cite{GiuSV01} authors consider an autonomous system, where the switching sequence is given a priori, and the goal is to find the optimal switching times. The possibility of modifying the switching signal to achieve some optimality is certainly an interesting topic, but it is not our focus in this work. Similarly, in \cite{XuA04}, a predefined sequence of the active modes is given, and then the authors study the problem of determining optimal switching times and an optimal control input. Noteworthy, the authors decompose the problem into a two-stage procedure, where, in the first stage, the optimal control problem is solved for given switching times, and then, in the second step, the optimal switching times are determined. We envision that such an extension is also possible for our framework. In \cite{BenD05}, a nonlinear switched system is considered. The authors construct sub-optimal controls by embedding the switched system in an extended non-switched system, where the switching signal is also used for the controller design. It is shown that the trajectories of the switched system are dense in the extended non-switched system. Finally, \cite{WijT23} (see also the thesis \cite{Wij22}) studies the finite-horizon linear quadratic regulator (\LQR) problem for a switched linear differential-algebraic equation (\DAE) with a given switching signal. The resulting optimal control is a feedback controller that can be computed as the solution of a Riccati differential equation. We highlight that all the works mentioned treat the control problem only from a finite-dimensional perspective. For the infinite-dimensional setting using a semigroup approach, we refer to \cite{Han18} and the references therein. Again, the main focus in the context of optimal control is the derivation of optimal switching signals, and the optimal control computation is not considered. We exemplarily mention here \cite{RueH16,RueH17,RueMH18}, where optimal switching is studied for an efficient simulation of gas flow in a pipe network using a model hierarchy. For future work, we want to combine these ideas with the two-stage approach of \cite{XuA04} and our certified \MOR-\MPC framework presented in this paper.

\MPC is an optimization-based control technique with a feedback mechanism for linear and nonlinear systems \cite{grune2017nonlinear}, where a sequence of finite-horizon control problems substitutes the infinite-horizon optimal control problem over receding time intervals. We refer to \emph{stabilizing} \MPC whenever the control objective is to drive the system (possibly unstable) to a predefined desired trajectory by penalizing deviations from the reference using a stage cost function, which is positive definite w.r.t. the reference. In contrast, in \emph{economic} \MPC, one is interested in operating a system at a trajectory with minimal cost. Here, the cost function can model all kinds of quantities and does not need to be positive definite w.r.t to a given target. The stability of the closed-loop trajectory can be ensured by introducing terminal costs and terminal constraints, or by choosing a sufficiently large prediction horizon \cite{grune2017nonlinear,mayne2000constrained,azmi2019hybrid}. In the case of \MPC specifically focusing on switched systems, we mention \cite{10050718,qi2021model,10598629,ZhangB13} and references therein.
\subsubsection{\MPC and \MOR}
Since \MPC represents a multi-query context for finite-horizon optimization problems, employing \ROM techniques to accelerate the process is natural. We focus on Galerkin-\ROMs, i.e., \ROMs obtained by projecting the dynamical system onto low-dimensional linear subspaces. Naturally, this leads to an interest in deriving conditions under which the \ROM-\MPC feedback law and the \ROM-\MPC closed-loop state exhibit correct behavior despite approximation errors.
The work \cite{loehning2014model} treats continuous-time, general stabilizable linear systems coupled with general projection-based \MOR with control and state constraints. Stabilization is achieved by exploiting the feedback of the \MOR error. This feedback, combined with tightened terminal constraints and terminal cost functions, ensures the stability of the \MPC closed-loop.
\cite{DieG23} provides a stabilizing \MPC framework using the reduced basis method for parametrized linear-quadratic control problems with control constraints. The authors propose an algorithm that adaptively searches for the minimal stabilizing prediction horizon and estimates the suboptimality degree of the \ROM-\MPC feedback.
The \ROM feedback stabilizes the original problem, where stability is guaranteed by extending the prediction horizon in the \MPC problem. The theory is based on the \MPC book \cite{grune2017nonlinear} and employs a reduced suboptimality degree for stabilizing \MPC in a discrete-time framework. The reduced suboptimality degree is based on a posteriori error bounds of the value function, optimal control, and optimal state of the finite-horizon linear quadratic subproblem. We emphasize that the bounds for the state and the adjoint we derive here are similar to those in \cite{DieG23} but adapted to our switched setting. In contrast, we are not interested in parametrized systems, but the control is the only parameter. We extend the estimates for linear-quadratic \OCPs derived in \cite{DBLP:journals/coap/TroltzschV09,AliH18} to the case of general convex, lower-semicontinuous regularization functions and perturbed initial values. This enables us to formulate estimates for the closed-loop state and the \ROM feedback law, which account for the error from previous \MPC iterations and are applicable in stabilizing and economic \MPC settings. Another unique aspect of our optimal control error estimators is that they are expressed in terms of output bounds and can be linked to output bounds derived from \emph{balanced truncation} (\BT) model reduction.
In \cite{LorMFP22}, discrete-time linear systems with quadratic cost functional and control-state constraints are treated. The authors provide theoretical performance guarantees, computational efficiency, and output feedback for general projection-based \ROMs. The work \cite{grune2021performance} concerns economic \MPC using \POD. A performance index for economic \MPC is developed and used to evaluate the performance of the \MPC controller online. However, the performance index is directly applied to the \ROM system and does not quantify the error with respect to the \FOM system. Economic \MPC with state constraints on a time-varying parabolic equation is studied in \cite{AndGMMPV22}; the \POD model is updated dynamically using both a-posteriori and a-priori error estimates, but there is no error estimate for the closed-loop state given. The work \cite{AllV15} introduces a stabilizing \MPC framework in continuous time for semilinear parabolic \PDEs without terminal constraints. A sufficient condition for stability of the \POD feedback is given based on a priori analysis, and it is proven that the \ROM feedback stabilizes the \ROM but not the \FOM.
\subsubsection{\MOR for switched systems}
Several works have appeared in recent years addressing the topic of \MOR for switched systems, such that we restrict our literature review to some selected references. In \cite{Wu2009}, a method that uses a set of coupled linear matrix inequalities is proposed, which becomes infeasible in a large-scale context. Nevertheless, the matrix inequalities are important to guarantee quadratic stability of the reduced system and derive an error bound \cite{Petreczky2013}. The works \cite{GosPAF18,Shaker2012} solve a set of coupled Lyapunov equations to compute Gramians, which can then be used for a balancing-based model reduction. Unfortunately, there is no guarantee that the set of coupled Lyapunov equations is solvable \cite{Liberzon2003}, and is not suitable for large-scale settings. In \cite{SchU18}, the switched system is recast as a linear system such that standard methods can be applied, while in \cite{PontesDuff2020}, the system is recast as a bilinear system, and a balanced truncation approach for such systems is employed. See also \cite{ManU24} for large-scale computation using the bilinear formulation. We emphasize that both methods provide an a priori error certification for the approximated input-output map and do not require a priori knowledge of the switching signal. This is crucial to develop a certified \MOR-\MPC framework. In \cite{Bastug2016}, interpolation-based techniques to construct a \ROM for a switched system are discussed. The authors of  \cite{Peitz2019} discretize the control variable to obtain a switched system of autonomous \PDEs that they approximate with \ROMs in an \MPC framework. Finally, we refer to \cite{Hossain2024} for \MOR of switched systems with state jumps and to \cite{Hossain2023}, where a similar technique is applied in the context of switched {\DAE}s. The drawbacks of this approach are two: the switching sequence needs to be known a priori, and the method is not presented to deal with large-scale systems. Recent works \cite{ManU24,ManU24b} overcome these issues, presenting a \MOR framework for large-scale systems that do not require a priori knowledge of the switching sequence.
\subsection{Structure of the paper}
We present the \MPC problem in \Cref{sec: problem formulation} and discuss the well-posedness of the forward problem and the optimal control problem for which we also derive the first-order optimality conditions. In \Cref{sec:approximation}, we show how to obtain suitable error estimates for the control, state, and adjoint trajectories that are then employed to guarantee a desired behavior of the reduced \MPC closed-loop procedure. To demonstrate our theoretical findings, we apply our methodology in \Cref{sec:numerics} to a switched system arising from two adjacent rooms where heat propagation is controlled using a switching law. Finally, we state our conclusions in \Cref{sec:conclusion}.
\subsection{Notation}
The symbol $\calI_H$ denotes the identity operator acting on a Hilbert space $H$. Further, $\Li(H,V)$ denotes the set of all linear and bounded operators mapping between $H$ and a Hilbert space $V$. By $L^2(\timeInt,V)$ we denote the Hilbert space of measurable and square-integrable functions mapping from some (time) interval $\timeInt\subseteq \R$ onto $V$ endowed with the canonical inner product. Whenever it is clear from the context, we write $L^2$ instead of $L^2(\timeInt,V)$. For given $\cOper\in \Li(H,V)$, the dual operator $\cOper'\in \Li(V',H')$ is defined by $ \langle \cOper'\phi,\varphi\rangle_{H',H}=\langle \phi,\cOper\varphi\rangle_{V',V}$ for all $\phi\in V', \varphi\in H$. Here, $\langle \cdot,\cdot\rangle_{H',H}$ denotes the duality pairing in $H$ and $H'$ denotes the topological dual space of $H$. 
We denote the left and right-limits of a function $\switch$ mapping from $\timeInt$ as $\switch(t^+)= \lim_{s\searrow t} \switch(s)$ and $\switch(t^-) = \lim_{s\nearrow t} \switch(s)$, if they exist.
\section{Problem formulation and well-posedness}
\label{sec: problem formulation}
In \Cref{subsec:MPC} we introduce our control problem and the basic \MPC idea. Then, in \Cref{subsec:switchedSystem,subsec:optimalControl}, we provide the precise regularity assumptions and study the well-posedness of the switched system and the \MPC subproblem.
\subsection{\MPC problem}
\label{subsec:MPC}
Let $V$ and $H$ be real separable Hilbert spaces, and let the embedding of $V$ in $H$ be dense and compact. Identifying $H$ with its dual space, we obtain a Gelfand triple $(V,H,V')$ with $V\hookrightarrow H \hookrightarrow V'$, where each embedding is continuous and dense; cf.~\cite[Sec.~23.4]{Zei90}. For given sampling time $\delta>0$ and prediction horizon $T>\delta$, we define the receding prediction interval $\timeInt_n \vcentcolon= (t_n,T_n)$ with $t_n \vcentcolon= n\delta$, $T_n \vcentcolon= t_n + T$, and $n\in\N_0$, and study the finite-time \OCP with the cost function 
\begin{align}\label{eqn:cost:fun}
\begin{aligned}
    J_n(\outVar,\inpVar) &\vcentcolon=  \int_{\timeInt_n} \ell(t,\outVar(t), \inpVar(t)) \dt + \tfrac{\mu}{2} \|\outVar(T_n)-\outVar_{T_n}\|^2_{\R^{\outVarDim}} + g_n(\inpVar)\\
    \ell(t,{\outVar(t)},{\inpVar(t)}) &\vcentcolon= \tfrac{1}{2}\|{\outVar(t)}-\outVarDes(t)\|^2_{\R^{\outVarDim}}+ \tfrac{\lambda}{2} \|{\inpVar(t)}-\inpVarDes(t)\|^2_{\R^{\inpVarDim}}
    \end{aligned}
\end{align}
subject to finding $(\outVar,\state,\inpVar)\colon \timeInt_n \to \R^{\outVarDim}\times V \times \R^{\inpVarDim}$ that solves the switched control system
\begin{equation}
    \label{eqn:weakForm} 
    \left\{\quad
    \begin{aligned}
        m_{\switch}(\ddt \state,\varphi) + a_{\switch}(\state,\varphi) &= b_{\switch}(\inpVar,\varphi) && \text{for all $\varphi\in V$, a.e. (almost everywhere) in }\timeInt_n,\\
       \state(t_n) &= \state_{n} && \text{in } H,\\
       \outVar &= \cOper_{\switch}\state && \text{a.e. in }\timeInt_n, 
    \end{aligned}\right.
\end{equation}
with external switching signal $\switch\colon \timeInt_n \to \{1,\ldots,\nrModes\}$, bilinear forms\footnote{We refer to \Cref{rem:moperdef} for the precise meaning of the time derivative within the bilinear form $m_i$.}
\begin{align*}
    m_i&\colon H\times H\to \R, &
    a_i&\colon V\times V\to \R, &
    b_i&\colon \R^{\inpVarDim}\times H\to\R, 
\end{align*}
linear operators $\cOper_i\colon H\to \R^{\outVarDim}$ for $i=1,\ldots,\nrModes$,
time-dependent target functions $\inpVarDes$, $\outVarDes$, terminal data $y_{T_n}$, initial value $\state_n\in H$, regularization parameters $\mu,\lambda>0$, and functional $g_n$ to model sparse regularization and control constraints (see \Cref{subsec:exampleContinued} later on). The numbers $\inpVarDim,\outVarDim,\nrModes\in\N$, denote the number of control inputs, output quantities, and subsystems, respectively. The precise assumptions and regularities for all the quantities involved are given in the next subsection (see \Cref{ass:switchedSystem}).
In summary, we obtain the finite-time \OCP
\begin{equation}
	\label{eqn:MPCsubproblem}
    \min J_n(\outVar,\inpVar)\quad\text{subject to (s.t.)}\quad(\outVar,\state,\inpVar) \mbox{ solves } \eqref{eqn:weakForm}.
\end{equation}
The idea of \MPC is to solve~\eqref{eqn:MPCsubproblem}, apply the resulting optimal control $\inpVar_n^\star$ to the system in the time interval $\timeInt_n^\delta \vcentcolon= (t_n,t_n+\delta]$, i.e., run the switched system~\eqref{eqn:weakForm} with the control input $\inpVar_n^\star$ to obtain $\state(t_n+\delta)$, and then repeat the process on the interval $\timeInt_{n+1}$ with a new initial value $\state_{n+1} \vcentcolon= \state(t_n+\delta)$. In this way, we iteratively construct the \MPC feedback control $\inpVarMPC$ and the controlled \MPC state~$\stateMPC$. We call the resulting method \FOM-\MPC and provide a pseudocode in \Cref{alg:FOM-MPC}.

\begin{algorithm}[t]
	\caption{(\FOM-\MPC)}\label{alg:FOM-MPC}
	\begin{algorithmic}[1]
		\Require Prediction horizon $T>0$, sampling time $0<\delta<T$, initial condition $\state_0$.
        \State Set $\stateMPC(0) = \state_0$.
		\For{$n=0,1,2,...$}
	        \State Set $t_n = n\delta$ and $\state_n \vcentcolon= \stateMPC(t_n)$.
		    \State Solve the \MPC subproblem \eqref{eqn:MPCsubproblem} for the \MPC feedback control $\optInpVarN$ and the \MPC state $\optState$. 
		    \State Apply the control $\optInpVarN$ to the equation~\eqref{eqn:weakForm} on $\timeInt_n^\delta \vcentcolon= (t_n,t_n+\delta]$, i.e., set
            \begin{align*}
            \inpVarMPC(t)\vcentcolon=\optInpVarN{(t)}, \quad\stateMPC (t)\vcentcolon=\optState(t), \quad\outVarMPC (t)\vcentcolon=\cOper_{\switch{(t)}} \optState{(t)}\quad\text{for }t\in \timeInt_n^\delta.
            \end{align*}
		\EndFor 
	\end{algorithmic}
    \begin{flushleft}
        \textbf{Output:} $\outVarMPC,\ \stateMPC$, $\inpVarMPC$.
    \end{flushleft}
\end{algorithm}
Whenever convenient, we will write the switched system~\eqref{eqn:weakForm} in operator form in the dual space of $V$. For this, let $\mOper_i,\aOper_i$, and $\bOper_i$ denote the operators corresponding to the bilinear forms $m_i,a_i$, and $b_i$, respectively. Then, the switched system~\eqref{eqn:weakForm} is equivalent to the switched abstract system
\begin{equation}
    \label{eq:OPLSS}
    \left\{\quad
    \begin{aligned}
        \mOper_{\switch}\ddt\state + \aOper_{\switch}\state &= \bOper_{\switch}\inpVar && \text{ in }V', \text{ a.e. in } \timeInt_n,\\
       \state(t_n) &= \state_{n} && \text{in } H,\\
       \outVar &= \cOper_{\switch}\state &&\text{a.e. in }\timeInt_n.
    \end{aligned}\right.
\end{equation}
\subsection{Well-posedness of the switched system}
\label{subsec:switchedSystem}
To establish the existence and uniqueness of a solution to the optimal control problem, we first discuss the well-posedness of the switched system~\eqref{eqn:weakForm}. We make the following assumptions.

\begin{assumption}
	\label{ass:switchedSystem}
	Consider the switched system~\eqref{eqn:weakForm}.
	\begin{enumerate}
		\item\label{ass:switchedSystem:switching} The external switching signal $\switch$ belongs to the set
			\begin{align*}
				\addSwitching \vcentcolon= \left\{\switch\colon (0,\infty)\to \{1,\ldots,\nrModes\} \,\left|\,\begin{aligned}&\switch \text{ is piecewise constant with}\\
				 &\text{locally finite number of jumps}\end{aligned}\right.\right\}.
			\end{align*}
            For a given switching signal $\switch\in\addSwitching$, we define
   \begin{equation}\label{eqn:swi:tim}
				\calT_n \vcentcolon= \left\{t\in \timeInt_n \,\left|\,\switch \text{ is discontinuous in t} \right.\right\}.
    \end{equation}
    For simplicity, we assume that $\{t_n,T_n\}\notin \calT_n$ and that there are $|\calT_n| = \nrSwitches-1\in \N$ switches. Note that $N$ depends on $n$, but we suppress this subscript for readability. Finally, we denote the switching times by $t_{n,1}<\ldots<t_{n,i}<\ldots<t_{n,\nrSwitches-1}$ and set $t_{n,0}\coloneqq t_n$ and $t_{n,N}\coloneqq T_n$.
 		\item\label{ass:switchedSystem:A} The bilinear forms $a_i\colon V\times V$ are uniformly continuous and satisfy a uniform G\r{a}rding's inequality, i.e., there exist constants $\coerAV,\contA>0$ and $\coerAH\geq 0$ such that
        	\begin{equation}
        		\label{eqn:properties:a}
        		a_i(\phi,\phi)\geq \coerAV\|\phi\|_V^2 - \coerAH\|\phi\|_H^2,\qquad
        		a_i(\phi,\varphi) \leq \contA \|\phi\|_V\|\varphi\|_V,
       		 \end{equation}
       		 for all $\phi,\varphi\in V$ and all $i=1,\ldots,\nrModes$.
		\item\label{ass:switchedSystem:M} The bilinear forms $m_i\colon H\times H\to \R$ are symmetric and define inner products in the pivot space $H$. In particular, there exist constants $\coerM,\contM>0$ such that 
			\begin{equation}\nonumber
				m_i(\varphi,\varphi)\geq \coerM\,{\|\varphi\|}_H^2,\qquad
				m_i(\phi,\varphi) \leq \contM\,{\|\phi\|}_H {\|\varphi\|}_H
			\end{equation}			        
			for all  $\phi,\varphi\in H$ and all $i=1,\ldots,\nrModes$.
		\item\label{ass:switchedSystem:B} The bilinear forms $b_i\colon \R^{\inpVarDim}\times V$ are uniformly continuous, i.e., there exists a constant $\contB>0$ such that
          	\begin{equation}\nonumber
          		b_i(\bar{\inpVar},\varphi) \leq \contB\, {\|\bar{\inpVar}\|}_{\R^{\inpVarDim}}\|\varphi\|_{V}
          	\end{equation}
          	for all $\bar{\inpVar}\in\R^{\inpVarDim}$, $\varphi\in V$ and all $i=1,\ldots,\nrModes$.
		\item\label{ass:switchedSystem:C}The operators $\cOper_i\in\Li(H,\R^{\outVarDim})$ are uniformly bounded, i.e., there exists a constant $\contC>0$ such that
			\begin{equation}\nonumber
          		\|\cOper_i\phi\|_{\R^{\outVarDim}} \leq \contC \|\phi\|_H
          	\end{equation}
          	for all $\phi\in H$ and all $i=1,\ldots,\nrModes$.
        \item The data satisfy $\inpVarDes\in L^\infty((0,\infty),\R^{\inpVarDim})$, $\outVarDes\in L^\infty((0,\infty),\R^{\outVarDim})$, $y_{T_n}\in\R^{\outVarDim}$ and $\theta_n\in H$.
 	\end{enumerate}
\end{assumption}
\begin{remark}
    \label{rem:moperdef}
    As common in the literature, we identify the operator $\mOper_i\colon H\to H$ with an operator $\widehat{\mOper}_i\colon V'\to V'$ by defining 
    \begin{align}
        \langle \widehat{\mOper}_i \ddt \state,\varphi\rangle_{V',V} \vcentcolon= \langle \ddt(\mOper_i \state),\varphi\rangle_{V',V} =\vcentcolon m_i(\ddt \state,\varphi)
    \end{align}
    for all $\varphi\in H$ and $\state\in L^2(\timeInt_n,H)\cap H^1(\timeInt_n,V')$; see for instance {\cite[Thm.~7]{Muj22}}.
\end{remark}
For the forthcoming analysis, we introduce the switched system with state transitions given by
\begin{subequations}
    \label{eqn:switchedPDE} 
    \begin{empheq}[left=\left\{\quad,right=\right.]{align}
        \label{eqn:switchedPDE:pde} m_{\switch}(\ddt \state,\varphi) + a_{\switch}(\state,\varphi) &= \langle f, \varphi\rangle_{V',V} && \text{for all $\varphi\in V$, a.e.~in $\timeInt_n$},\\
        \label{eqn:switchedPDE:initialCond} \state(t_n) &= \state_{n} && \text{in $H$},\\
        \label{eqn:switchedPDE:stateTransition} \state(t^+) &= \kOper_{\switch(t^-),\switch(t^+)}\state(t^-) && \text{in $H$, for all $t\in\timeInt_n$,}
    \end{empheq}
\end{subequations}
with state transition operators $\kOper_{i,j}\in \Li(H,H)$ for $i,j\in\{1,\ldots,\nrModes\}$. Note that at the switching points, i.e., at the points $t\in\calT_n$, the state transition maps~\eqref{eqn:switchedPDE:stateTransition} give rise to a new initial value for the next time interval, where the switching signal is constant. The original system~\eqref{eqn:weakForm} is recovered by setting $f \vcentcolon= \bOper_{\switch}\inpVar$ and $\kOper_{i,j} = \calI_{H}$ (the identity map in $H$) for $i,j=1,\ldots,\nrModes$. In general, we only allow for a state transition at discontinuities of the switching signal. Hence, we make the following additional assumption.

\begin{assumption}
    \label{ass:transitionOperators}
    The state transition operators $\kOper_{i,j}$ satisfy $\kOper_{i,i} = \calI_{H}$ for $i=1,\ldots,\nrModes$.
\end{assumption}

Similarly to \cite{GosPAF18} and \cite{wloka1987partial}, we introduce for $\switch\in\addSwitching$ the Hilbert spaces
\begin{align*}
    \X_{n,\switch,i} &\vcentcolon= \big\{v \in L^2((t_{n,i},t_{n,i+1}),V) \mid \dot{v} \in L^2((t_{n,i},t_{n,i+1}),V')\big\}\hookrightarrow C([t_{n,i},t_{n,i+1}],H),\\
     \X_n &\vcentcolon= \big\{v \in L^2(\timeInt_n,V) \mid \dot{v} \in L^2(\timeInt_n,V')\big\}\hookrightarrow C(\closure{\timeInt}_n,H),
\end{align*}
and the space
\begin{equation}\label{eqn:swi:sol:space}
    \X_{n,\switch} \vcentcolon= \left\{v\colon \timeInt_n\to H\,\left|\, v|_{(t_{n,i},t_{n,i+1})} \in \X_{n,\switch,i}\right.\right\}\subseteq L^2(\timeInt_n,V),
\end{equation}
where $v|_{(\alpha,\beta)}$ denotes the restriction of $v$ to the time interval $(\alpha,\beta)$.
Similar to the mild solution for switched \PDEs using semigroup theory \cite[Lem.~2]{RueH16}, we obtain the well-posedness of the switched system in the weak setting.

\begin{proposition}[Well-posedness of the switched system]
	\label{prop:wellPosedness}
	Let the switched system~\eqref{eqn:switchedPDE} with external switching signal $\switch\in\addSwitching$ satisfy \Cref{ass:switchedSystem,ass:transitionOperators}. Then, for every $f\in L^2(\timeInt_n,V')$ there exists a unique solution $\state\in\X_{n,\switch}$. If in addition $\kOper_{i,j} = \calI_{H}$ for all $i,j=1,\ldots,\nrModes$, then there is a unique $\state\in\X_n$ satisfying~\eqref{eqn:switchedPDE}. 
\end{proposition}
\begin{proof}
    Note that $\kOper_{\switch(t^-),\switch(t^+)} = \calI_{H}$ for every $t\in(t_{n,i},t_{n,i+1})$, since $\switch$ is constant in $(t_{n,i},t_{n,i+1})$ for $i=0,\ldots,\nrSwitches-1$. Hence, we can construct a solution iteratively by solving the parabolic equation
    \begin{subequations}\nonumber
        \begin{empheq}[left=\left\{\quad,right=\right.]{align}\nonumber
         m_{\switch}(\ddt \state^i,\varphi) + a_{\switch}(\state^i,\varphi) &= \langle f, \varphi\rangle_{V',V} && \text{for all $\varphi\in V$, a.e.~in $(t_{n,i},t_{n,i+1})$},\\
           \nonumber \state^i(t_{n,i}) &= \state_{n,i} && \text{in $H$},
        \end{empheq}
    \end{subequations}
    with $\state_{n,0} \vcentcolon= \state_n$, and $\state_{n,i} \vcentcolon= \kOper_{\switch(t_{n,i}^-),\switch(t^+_{n,i})} \state^{i-1}(t_{n,i}^-)$. We immediately infer a unique solution $\state^i\in\X_{n,\switch,i}$; see for instance \cite[Cha.~3, Thm.~4.1]{LioM72}.
\end{proof}

\begin{corollary}
    \label{cor:wellPosednessForward}
    Assume that the switched system~\eqref{eqn:weakForm} satisfies \Cref{ass:switchedSystem} and define the control space $\U_n \vcentcolon= L^2(\timeInt_n,\R^{\inpVarDim})$.
    Then for every $\inpVar\in \U_n$ there exists a unique $\state\in\X_n$ satisfying~\eqref{eqn:weakForm}.
\end{corollary}
\begin{definition}[Solution operator]
    Let the switched system~\eqref{eqn:weakForm} with an external switching signal $\switch$ satisfy \Cref{ass:switchedSystem}. Then, we define the solution operator $\solper\colon H\times \U_n\to\X_n$ as
    \begin{align}\nonumber
        \solper(\state_0,\inpVar)\vcentcolon= \state,
    \end{align}
    where $\state$ is the unique solution of~\eqref{eqn:weakForm} with initial value $\state_0\in H$ and control input $\inpVar\in\U_n$. To simplify the notation, we define $\solOper\colon\U_n\to\X_n$ as $\solOper\inpVar \vcentcolon= \solper(0,\inpVar)$ for all $\inpVar\in \U_n$. 
\end{definition}

We emphasize that the linearity of~\eqref{eqn:weakForm} implies that the solution operator is affine in the control input for a fixed initial value. If the initial value is zero, then the solution operator is linear and bounded in the control input. 
Its dual operator $\solOper'\colon\X_n' \to \U_n$ is given by
\begin{equation}\nonumber
    \langle \solOper' v , u \rangle_{\U_n} = \langle v, \solOper u\rangle_{\X_n',\X_n} \quad \text{for all }(u,v) \in \U_n \times \X_n',
\end{equation}
where we identified $\U_n$ with its dual space $\U_n'$.
\begin{remark}
We can interpret the switched system~\eqref{eqn:weakForm} as a linear time-varying problem. In particular, for given $\switch\in\addSwitching$, we define the time-dependent bilinear form
\begin{align*}
	m&\colon \timeInt_n\times H \times H\to \R,& (t,\phi,\varphi) &\mapsto \sum\nolimits_{i=1}^{\nrModes} \delta_{i,\switch(t)} m_i(\phi,\varphi),\\
	a&\colon \timeInt_n\times V \times V\to \R,& (t,\phi,\varphi) &\mapsto \sum\nolimits_{i=1}^{\nrModes} \delta_{i,\switch(t)} a_i(\phi,\varphi), \\
	b&\colon \timeInt_n\times \R^{\inpVarDim} \times V\to \R,& (t,\bar{\inpVar},\varphi) &\mapsto \sum\nolimits_{i=1}^{\nrModes} \delta_{i,\switch(t)} b_i(\bar{\inpVar},\varphi),
\end{align*}
where $\delta_{i,j} = 1$ if $i=j$ and zero otherwise, and consider the equivalent time-varying system
\begin{equation}\nonumber
	\left\{\quad\begin{aligned}
		m(t;\ddt \state(t),\varphi) &= a(t;\state(t),\varphi) + b(t;\inpVar(t),\varphi), && \text{ for all }\varphi\in V, \text{ a.e. in }\timeInt_n, \\
		\state(t_n) &= \state_n && \text{in }H.
	\end{aligned}\right.
\end{equation}
\end{remark}

\subsection{Guiding example}
\label{subsection:example}
Let $\Omega\subset \R^{2}$ be the rectangular domain given in \Cref{fig: geometry_tworooms} composed of the two rooms $\Omega_1$ and $\Omega_3$, the door $\Omega_2$ and the walls $\Omega_4$, inspired by the simpler one-dimensional model discussed in \cite{SchU18}. Suppose that $\partial \Omega$ is partitioned into two disjoint sets $\Gamma_c$ and $\Gamma_o$, where the control $u\colon[0,\infty)\to\R^{\inpVarDim}$ acts on $\Gamma_c$ via $\inpVarDim\in\N$ actuators. The switching signal $\switch\colon[0,\infty)\to \{1,2\}$ is used to distinguish the cases in which the door is closed ($\switch = 1$) or open ($\switch=2$). 
The goal is to control the following switched heat equation
\begin{subequations}
    \label{eq: parabolic_problem}
    \begin{align}
       \label{eq: parabolic_problem-a}
	   \zeta(\bx,\switch(t))\partial_t \state(t,\bx) + \difOper(\bx,\switch(t)) \state(t,\bx)&= 0, && (t,\bx) \in (0,\infty)\times\Omega,\\
       \kappa(\bs,\switch(t)) \tfrac{\partial}{\partial n}\state(t,\bs) + \gamma_{\mathrm{o}}\state(t,\bs)&= 0, && (t,\bs) \in (0,\infty) \times \Gamma_{\mathrm{o}},\\
	   \kappa(\bs,\switch(t)) \tfrac{\partial}{\partial n}\state(t,\bs) - \sum\nolimits_{i=1}^{\inpVarDim} \chi_i(\bs) \inpVar_i(t)&=0, && (t,\bs) \in (0,\infty) \times \Gamma_{\mathrm{c}},\\
	   \theta(0,\bx) &= 0, && \bx \in \Omega,
    \end{align}
\end{subequations}
with the differential operator
\begin{align*}
    \difOper(\bx,\switch)\state(\bx)\vcentcolon=-\nabla \cdot\big(\kappa(\bx,\sigma)\nabla \state(\bx)\big) + \advection(\bx)\cdot\nabla\state(\bx) + c(\bx)\state(\bx).
\end{align*}
The coefficients $\zeta$ and $\kappa$ are piecewise constant functions given as
\begin{align*}
    \zeta(\bx, \switch)\vcentcolon=
    \begin{cases}
        \zeta_{{\Omega_2}}(\switch)&\text{if }\bx\in \Omega_2,\\
        \zeta_2&\text{if }\bx\in\Omega_1\cup\Omega_3,\\
        \zeta_3&\text{if }\bx\in\Omega_4,
	\end{cases} \quad
    \kappa(\bx, \switch)\vcentcolon= 
    \begin{cases}
        \kappa_{\Omega_2}(\switch)&\text{if }\bx\in \Omega_2,\\
        \kappa_2&\text{if }\bx\in \Omega_1 \cup \Omega_3,\\
        \kappa_3&\text{if }\bx\in \Omega_4,
	\end{cases}
\end{align*}
\begin{equation*}
    \zeta_{\Omega_2}(\switch)\vcentcolon=
    \begin{cases}
        \zeta_1&\text{if }\sigma=1,\\
        \zeta_2&\text{if }\sigma=2,
	\end{cases}\quad\quad\kappa_{\Omega_2}(\switch)\vcentcolon= 
    \begin{cases}
        \kappa_1&\text{if }\sigma=1,\\
        \kappa_2&\text{if }\sigma=2
	\end{cases}
\end{equation*}
for certain positive constants $\zeta_i, \kappa_i$ $(i=1,\ldots,3)$. Moreover, we suppose $\gamma_{\mathrm{o}}\geq 0$, $c\in L^{\infty}(\Omega)$ with $c$ uniformly positive, $\advection\in L^\infty(\Omega, \R^2)$ with $\nabla \cdot \advection = 0$, $\advection \cdot n \geq 0$. 
\begin{figure}
    \begin{center}
    \begin{tikzpicture}
    \draw[line width=2pt,color=black] (0,-1)--(-3,-1) -- (-3,1)--(0,1);
    \draw[line width=2pt,color=black](0,1)--(3,1)--(3,-1)--(0,-1);
    \draw[line width=2pt,color=blue] (0,-1)--(0,-0.3);
    \draw[line width=2pt,color=blue] (0,0.3)--(0,1);
    \draw[line width=2pt,color=teal] (0,0.3)--(0,-0.3);
    \draw[line width=2pt,color=red] (-3,-1) -- (-3,1);
    \node [below] at (0,-1) {};
    \node  at (-1.5,0) {$\Omega_1$};
    \node[left]  at (0,0.5) {\color{blue}$\Omega_4$};
    \node[left]  at (-3,0) {\color{red}$\Gamma_c$};
    \node[left]  at (-0,0) {\color{teal}$\Omega_2$};
    \node  at (1.5,0) {$\Omega_3$};
    \end{tikzpicture}
    \end{center}
    \caption[a]{The domain $\Omega$ and the two rooms $\Omega_1$, $\Omega_3$. The switching signal $\sigma$ acts on the door $\Omega_2$ between the walls $\Omega_4$ and the control $u$ acts on $\Gamma_c$.}
    \label{fig: geometry_tworooms}
\end{figure}
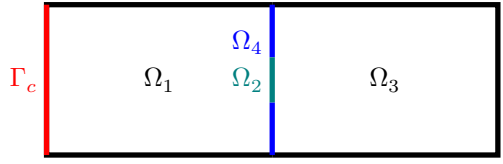
To derive the abstract form~\eqref{eqn:weakForm}, we introduce the spaces $V \vcentcolon= H^1(\Omega)$ and $H \vcentcolon=L^2(\Omega)$ and the bilinear forms $m_i\colon H\times H \to \R$ and $a_i\colon V \times V \to \R$ defined as
\begin{align*}
    m_i(\phi,\varphi)&\vcentcolon= \int_\Omega\zeta(\cdot\,,i)\varphi\phi\,\mathrm d\bx, & 
    a_i(\varphi,\phi)&\vcentcolon=\int_\Omega\kappa(\cdot\,,i)\nabla \varphi\cdot\nabla\phi + \big(\advection\cdot\nabla\varphi\big)\phi+c \varphi \phi\,\mathrm d\bx + \gamma_o\int_{\Gamma_o}\varphi\phi\,\mathrm d\bs
\end{align*}
for $i=1,2$. For the control operator, we define the bilinear form $b\colon \R^{\inpVarDim}\times H \to \R$ given by 
\begin{align*}
	 b(\bar{\inpVar},\phi) \vcentcolon= \sum\nolimits_{i=1}^{\inpVarDim} \bar{\inpVar}_i \langle \chi_i,\phi\rangle_{L^2(\Gamma_{\mathrm{c}})},
\end{align*}
where $\chi_1,\ldots, \chi_\omega$ denote the characteristic functions of a disjoint partition of $\Gamma_c$ into $\inpVarDim$ equally sized subdomains.
Note that~$b$ is independent of the switching signal, such that we simply set $b_i = b$ for $i=1,2$.
For the output functional, we are interested in the average temperature in the two rooms $\Omega_1$ and $\Omega_3$. Thus, we set $\outVarDim=2$ and consider the output $\cOper:H\to\R^{\outVarDim},\ \varphi \mapsto y=(y_1,y_2)=\cOper\varphi$ with
\begin{align*}
    y_1 \vcentcolon= \frac{1}{|\Omega_1|}\int_{\Omega_1}\varphi(\bx)\,\mathrm d\bx,\quad y_2 \vcentcolon=\frac{1}{|\Omega_3|}\int_{\Omega_3}\varphi(\bx)\,\mathrm d\bx.
\end{align*}
It can be shown that \Cref{ass:switchedSystem} is satisfied for this example.
%
\subsection{Well-posedness of the optimal control problem}
\label{subsec:optimalControl}
Before we study the existence, uniqueness, and characterization of minimizers of the \OCP~\eqref{eqn:MPCsubproblem}, we make the following observation. Let $\hat{\state} \vcentcolon= \solper(\state_n,0)$ denote the solution of~\eqref{eqn:weakForm} for the zero input. Then any solution $\state\in\X_n$ of~\eqref{eqn:weakForm} can be written as $\state = \solOper\inpVar + \hat{\state}$.
Using the solution operator, we can formulate the open-loop optimal control problem \eqref{eqn:MPCsubproblem} solely in the input variable. In more detail, define $\hatoutVarDes=\cOper_\switch\hat \state$ and
\begin{align*}
    \Jhat_n\colon\U_n\to \R,\qquad \inpVar \mapsto J_n(\cOper_{\switch}\solOper\inpVar+\hatoutVarDes,\inpVar)
\end{align*}
with $J_n$ as in~\eqref{eqn:cost:fun}. The resulting \OCP then reads
\begin{equation}
    \label{eqn:MPCsubproblemControlreduced}
    \min \Jhat_n(\inpVar)\quad\text{s.t.}\quad\inpVar\in \U_n,
\end{equation}
which is equivalent to~\eqref{eqn:MPCsubproblem} due to \Cref{prop:wellPosedness}.
For the well-posedness of the finite-horizon \OCP~\eqref{eqn:MPCsubproblem} or equivalently \eqref{eqn:MPCsubproblemControlreduced}, we make the following standard assumption.
\begin{assumption}
	\label{ass:propertiesG}
	The function $g_n\colon\U_n\to \R\cup \{\pm \infty \}$ in~\eqref{eqn:cost:fun} is proper, convex, and lower semicontinuous.
\end{assumption}
\begin{theorem}[Unique solution of the \OCP]
    \label{thm:wellPosednessOCP}
    Let \Cref{ass:switchedSystem,ass:propertiesG} be satisfied. Then there exists a unique minimizer $\optInpVarN \in \U_n$ of \eqref{eqn:MPCsubproblemControlreduced}. Moreover, there exists a unique $(\optOutVar,\optInpVarN)$ solving the \OCP~\eqref{eqn:MPCsubproblem}.
\end{theorem}
\begin{proof}
    Due to $\lambda>0$ and $\solOper\in\Li(\U_n,\X_n)$, the function $\Jhat_n$ is strictly convex, coercive, and weakly lower semicontinuous in $\U_n$. Thus, the existence of a unique minimizer $\optInpVarN$ follows from standard arguments (see, e.g., \cite[Thm.~2.14]{troltzsch2024optimal}).
    Due to the equivalence of the \OCPs~\eqref{eqn:MPCsubproblemControlreduced} and~\eqref{eqn:MPCsubproblem}, the second claim directly follows.
\end{proof}
Next, we aim to characterize the unique solution $\optInpVarN$ by examining the first-order optimality conditions. We will see that the switched bilinear form $m_{\switch}$ interacting with the time derivative generates a discontinuous adjoint dynamic (in the $H$ inner product).
Note that the function~$g_n$ is generally not differentiable in the classical sense, which motivates introducing the convex subdifferential of $g_n$ at $\optInpVarN$ (cf.~\cite[Def.~16.1]{bauschke2011convex})
\begin{align}\label{eqn: convex subdifferential}
		\partial g_n(\optInpVarN)\vcentcolon=\left\{w\in \U_n \ \big| \ {\langle w, v-\optInpVarN \rangle}_{\U_n} \leq g_n(v) -g_n(\optInpVarN)\ \text{ for all }v \in \U_n \right\},
\end{align}
which is non-empty due to \Cref{ass:propertiesG}. Using this characterization, the optimality condition can be formulated as follows.
\begin{theorem}[First-order optimality condition]
    \label{thm:FONC:MPC}
    Let \OCP~\eqref{eqn:MPCsubproblem} satisfy \Cref{ass:switchedSystem,ass:propertiesG}, consider the set of switching time points $t_{n,i}\in \calT_n$ \eqref{eqn:swi:tim} for $i=1,\ldots,\nrSwitches-1$. Then, $\optInpVarN \in \U_n$ is optimal for \eqref{eqn:MPCsubproblemControlreduced} if and only if there exists unique
    \begin{equation*}
        \optState\in\X_n\quad\text{and}\quad\optAdState \in \X_{n,\switch},
    \end{equation*}
    with $\X_{n,\switch}$ defined in \eqref{eqn:swi:sol:space}, that solve the coupled system 
    \begin{subequations}
        \label{eqn:MPCsubproblem:FONC}
        \begin{empheq}[left=\left\{,right=\right.]{align}
            \label{eq: state mpc subproblem 1}
            m_{\switch}(\ddt \optState,\varphi) + a_{\switch}(\optState,\varphi) &= b_{\switch}(\optInpVarN, \varphi)&& \text{a.e.~in $\timeInt_n$,} \\
            \label{eq: state mpc subproblem 2}
            \optState(t_n) &= \state_n && \text{in $H$,}\\
            \label{eq: state mpc subproblem 3}
            \optOutVar 	&= \cOper_{\switch}\optState && \text{a.e. in $\timeInt_n$,}\\ 
            \label{eq: adjoint mpc subproblem 1}
            - m_{\switch}(\ddt \optAdState,\varphi) + a_{\switch}(\varphi,\optAdState) &= \langle \optOutVar - \outVarDes, \cOper_{\switch}\varphi\rangle_{\R^{\outVarDim}}&& \text{a.e.~in $\timeInt_n$,}\\
            \label{eq: adjoint mpc subproblem 2}
            m_{\switch(T_n)}(\varphi,\optAdState(T_n))&={\mu}\langle \optOutVar(T_n)- \outVar_{T_n},\cOper_{\switch(T_n)} \varphi\rangle_{\R^{\outVarDim}}&& \\ 
            \label{eq: adjoint mpc subproblem 3}
            m_{\switch(t_{n,i}^-)}(\varphi, \optAdState(t_{n,i}^-))&=m_{\switch(t^+_{n,i})}(\varphi,\optAdState(t_{n,i}^+)) && \text{for all $t_{n,i}\in \calT_n$,}\\  
            \label{eq: state mpc subproblem gradient}
          \langle   \lambda (\optInpVarN-\inpVarDes)  +\bOper_{\switch}' \optAdState, v - \optInpVarN\rangle_{\U_n}&\geq g_n(\optInpVarN) -g_n(v)
        \end{empheq}
    \end{subequations}
    for all $\varphi\in V$, all $v\in \U_n$.
    If $g_n \equiv 0$, \eqref{eq: state mpc subproblem gradient} reduces to
    \begin{equation}\label{eq: unconstrained opt condition}
      \lambda (\optInpVarN-\inpVarDes)  +\bOper_{\switch}' \optAdState = 0\quad \text{in }\U_n.
    \end{equation}
\end{theorem}
Before proceeding with the proof of \Cref{thm:FONC:MPC}, let us show that problem \eqref{eq: adjoint mpc subproblem 1}-\eqref{eq: adjoint mpc subproblem 3}, often called the \textit{adjoint system}, is well-posed. We emphasize that \eqref{eq: adjoint mpc subproblem 3} constitutes a state transition at the switching times.
\begin{corollary}[Well-posedness of the adjoint equation]
    \label{cor:wellPosedAdjoint}
    Let $\switch\in\addSwitching$ and system~\eqref{eqn:weakForm} satisfy \Cref{ass:switchedSystem}. Then for every $\outVar\in C(\closure{\timeInt}_n, \R^\outVarDim)$, there exists a unique $\adState \in \X_{n,\switch}$ solving
    \begin{subequations}
        \label{eqn:AdjointSystem}
        \begin{empheq}[left=\left\{\quad,right=\right.]{align}
            \label{eq: adjoint mpc subproblem 1 introduced}
            - m_{\switch}(\ddt \adState,\varphi) + a_{\switch}(\varphi,\adState) &= \langle \outVar - \outVarDes, \cOper_{\switch}\varphi\rangle_{\R^{\outVarDim}}&& \text{a.e. in $\timeInt_n$},\\
            \label{eq: adjoint mpc subproblem 2 introduced}
            m_{\switch(T_n)}(\varphi,\adState(T_n))&={\mu}\langle \outVar(T_n)- \outVar_{T_n},\cOper_{\switch(T_n)} \varphi\rangle_{\R^{\outVarDim}}&& \\ 
            \label{eq: adjoint mpc subproblem 3 introduced}
            m_{\switch(t_{n,i}^-)}(\varphi, \adState(t_{n,i}^-))&=m_{\switch(t^+_{n,i})}(\varphi,\adState(t_{n,i}^+)), && \text{for all $t_{n,i}\in \calT_n$}
        \end{empheq}
    \end{subequations}
    for all $\varphi\in V$. The corresponding solution operator is denoted as
    \begin{equation}\nonumber
        \solA\colon C(\closure{\timeInt}_n, \R^\outVarDim)\to \X_{n,\switch},\quad
        \outVar\mapsto \solA(\outVar)=\adState.
    \end{equation}
\end{corollary}
\begin{proof}   
    To apply \Cref{prop:wellPosedness}, we introduce the variable $\gamma(t) = \adState(T_n-t)$ and the switching signal $\tilde{\switch}(t) = \switch(T_n-t)$ for almost all (f.a.a.) $t\in\timeInt_n$ and obtain the switched system
    \begin{equation}\nonumber
        \left\{\quad\begin{aligned}
            m_{\tilde{\switch}(t)}(\ddt \gamma(t),\varphi) + a_{\tilde{\switch}(t)}(\varphi,\gamma(t)) &= \langle f(t),\varphi\rangle_{\R^{\outVarDim}}&& \text{f.a.a. } t\in \timeInt_n,\\
            \gamma(0) &= \gamma_0 && \text{in $H$},\\ 
            \gamma(t) &= \calK_{\tilde{\switch}(t^-),\tilde{\switch}(t^+)}\gamma(t^-) && \text{in $H$, for all $t\in\calT_n$}, 
        \end{aligned}\right.
    \end{equation}
    with $f\vcentcolon=\cOper_{\tilde{\switch}}'(\outVar(T_n-\cdot) - \outVarDes(T_n-\cdot))\in L^2(\timeInt_n,V')$, $\gamma_0$ being the unique solution of
    \begin{align*}
        m_{\switch(T_n)}(\gamma_0,\varphi) &= \mu\langle \outVar(T_n)- \outVar_{T_n},\cOper_{\switch(T_n)} \varphi\rangle_{\R^{\outVarDim}}\qquad\text{for all $\varphi\in H$},
    \end{align*}
    and $\kOper_{ij} \vcentcolon= \calM_j^{-1}\calM_i$. The result follows now immediately from \Cref{prop:wellPosedness}.
\end{proof}

\begin{proof}[Proof of \Cref{thm:FONC:MPC}]
    First, we introduce the evaluation operator $\calE_{T_n}\in \calL(C(\closure{\timeInt}_n,\R^p),\R^p)$ and $F_n\colon\U_n\to\R$ by
    \begin{align*}
        \calE_{T_n}\outVar &\vcentcolon= \outVar(T_n), &
        F_n(\inpVar) &\vcentcolon= \tfrac{1}{2}\int_{\timeInt_n} \ell(t,(\cOper_{\switch}\calS\inpVar(t)+\hatoutVarDes(t), \inpVar(t)) \dt + \tfrac{\mu}{2} \|\calE_{T_n}\cOper_{\switch}\calS\inpVar-\hat y_{T_n}\|^2_{\R^{\outVarDim}},
    \end{align*}
    with $\hatoutVarDes$ introduced in the beginning of \Cref{subsec:optimalControl} and $\hat{y}_{T_n}=\hatoutVarDes(T_n)$. Consequently, we obtain $\Jhat_n(\inpVar) = F_n(\inpVar)+ g_n(\inpVar)$ and the necessary optimality condition for \eqref{eqn:MPCsubproblemControlreduced}, given by Fermat's principle \cite[Thm.~16.3]{bauschke2011convex}, as
    \begin{equation} \nonumber
        -\nabla F_n(\optInpVarN) \in \partial g_n(\optInpVarN).
    \end{equation}
    By \eqref{eqn: convex subdifferential} this is equivalent to
    \begin{equation} \nonumber
        \langle \nabla F_n(\optInpVarN), v-\optInpVarN\rangle_{\U_n} \geq g_n(\optInpVarN)-g_n(v) \quad \text{for all }v \in \U_n.
    \end{equation}
    Due to the strict convexity of $\Jhat_n$ the condition is also sufficient. Since $F_n$ is quadratic, we obtain
    \begin{equation}\label{eqn: derivation of gradient of Fn}\nonumber
        \nabla F_n(\optInpVarN) = \lambda (\optInpVarN-\inpVarDes)+\solOper'\cOper_{\switch}'(\cOper_{\switch}\solOper \optInpVarN-{\hatoutVarDes})+\mu \solOper'\cOper_{\switch}'\calE_{T_n}'(\calE_{T_n}\cOper_{\switch}\solOper \optInpVarN-\hat \outVar_{T_n}).
    \end{equation}
    Let $\optAdState \in \X_{n,\sigma}$ be the unique solution of \eqref{eq: adjoint mpc subproblem 1}-\eqref{eq: adjoint mpc subproblem 3}. Thus, the claim follows if we can show
    \begin{equation*}
         \langle\calB_{\sigma}'\optAdState, v\rangle_{\U_n}= \langle\solOper'\cOper_{\switch}'(\cOper_{\switch}\solOper \optInpVarN-{\hatoutVarDes})+\mu \solOper'\cOper_{\switch}'\calE_{T_n}'(\calE_{T_n}\cOper_{\switch}\solOper \optInpVarN-\hat \outVar_{T_n}),v\rangle_{\U_n}
    \end{equation*}
    for all $v\in \U_n$. Note that here we understand $\bOper'_\switch \in \calL(L^2(\timeInt_n,V), \U_n)$ by identifying $\U_n$ with its dual. Using the definition of the solution operator $\solOper$ and \Cref{lemma:S-IBP} we obtain
    \begin{align*}
        \langle \calB_{\switch}'\optAdState , v \rangle_{\U_n} 
        &= \int_{t_n}^{T_n} \langle \calB_{\sigma(t)}v(t),\optAdState(t)\rangle_{V',V} \dt =\int_{t_n}^{T_n} b_{\switch(t)}\big(v(t),\optAdState(t)\big)\dt\\
        &= \int_{t_n}^{T_n} m_{\switch(t)}\big((\ddt\solOper v)(t),\optAdState(t)\big) + a_{\switch(t)}\big((\solOper v)(t),\optAdState(t)\big) \dt\\
        &= \int_{t_n}^{T_n}-m_{\switch_i}\big(\ddt \optAdState(t),(\solOper v)(t)\big) + a_{\switch(t)}\big((\solOper v)(t),\optAdState(t)\big) \dt \\
        &\qquad +   \sum_{i=1}^{\nrSwitches-1} m_{\switch(t^-_{n,i})}\big((\solOper v)(t_{n,i}^{-}),\optAdState(t_{n,i}^-)\big) - m_{\switch(t^{+}_{n,i})}\big((\solOper v)(t_{n,i}^+),\optAdState(t_{n,i}^+))\big) \\ 
        &\qquad +m_{\switch(T_n)}((\solOper v)(T_n),\optAdState(T_n))-m_{\switch(t_n)}((\solOper v)(t_n),\optAdState(t_n))\\
        &= \int_{t_n}^{T_n} -m_{\switch_i}\big(\ddt\left( \optAdState(t)\right),(\solOper v)(t)\big) + a_{\switch(t)}\big((\solOper v)(t),\optAdState(t)\big) \dt \\
        &\qquad +m_{\switch(T_n)}((\solOper v)(T_n),\optAdState(T_n)),
    \end{align*}
    where the components of the sum vanish due to~\eqref{eq: adjoint mpc subproblem 3}. Note that we also used $(\solOper v)(t_n) = 0$ and that $t_n, T_n$ are no switching points, i.e., $\switch$ is continuous at $t_n,T_n$. Using now the adjoint equation \eqref{eq: adjoint mpc subproblem 1}-\eqref{eq: adjoint mpc subproblem 3}, $\solOper(\optInpVarN) = \optState-\hat \state$, and the definition of $\hatoutVarDes$ and $\hat \outVar_T$, we obtain
    \begin{align*}
        \langle \bOper_{\switch}'\optAdState , v \rangle_{\U_n} 
        &= \int_{t_n}^{T_n} \langle \optOutVar(t) - \outVarDes(t), (\cOper_{\switch}\solOper v)(t)\rangle_{\R^{\outVarDim}} \dt  + {\mu}\langle \optOutVar(T_n)- \outVar_{T_n},\calE_{T_n}\cOper_{\switch}\solOper v\rangle_{\R^{\outVarDim}} \\
     &=  \langle \cOper_{\switch}\solOper \optInpVarN - \hatoutVarDes, \cOper_{\switch}\solOper v\rangle_{L^2(\timeInt_n, \R^p)}  + \mu \langle \calE_{T_n}\cOper_{\switch}\solOper\optInpVarN- \hat{\outVar}_{T_n},\calE_{T_n}\cOper_{\switch}\solOper v\rangle_{\R^{\outVarDim}}\\
     &=  \langle \solOper'\cOper_{\switch}'(\cOper_{\switch}\solOper\optInpVarN - \hatoutVarDes), v \rangle_{\U_n}  + {\mu}\langle\solOper'\cOper_{\switch}' \calE_{T_n}'(\calE_{T_n}\cOper_{\switch}\solOper\optInpVarN- \hat{\outVar}_{T_n}),  v\rangle_{\U_n}.
     \end{align*}
     If $g_n\equiv 0$, the right-hand side of $\eqref{eq: state mpc subproblem gradient}$ vanishes and choosing $v= \optInpVarN \pm \big(\lambda (\optInpVarN-\inpVarDes)  +\bOper_{\switch}' \optAdState\big) $ leads to \eqref{eq: unconstrained opt condition}.
\end{proof}
\subsection{Guiding example (continued)}
\label{subsec:exampleContinued}
We continue with the example from \Cref{subsection:example}. To that end, we assume that the feasible control set is given as a non-empty set $\Uad \subseteq \U_n$ defined as
\begin{equation}\nonumber
    \Uad \coloneqq \big\{u\in \U_n \mid \inpVar_a(t) \leq u(t)\leq \inpVar_b(t) \in \R^{\inpVarDim} \text{ f.a.a }t\in \timeInt_n\big\}
\end{equation}
for certain bounds $\inpVar_a, \inpVar_b\in \U_n$. We are interested in tracking a target $\outVarDes\in L^\infty(\timeInt_n,\R^{\outVarDim})$ with the minimal possible control cost. Moreover, we want to enhance the sparsity of the control signal. Therefore, we consider the $L^1$-regularized optimization problem
\begin{align}
\label{eqn:optimization_example}
\begin{aligned}
    \min\limits_{\outVar\in L^2(\timeInt_n,\R^\outVarDim),\inpVar\in\U_n} \tfrac{1}{2} \int_{\timeInt_n} {\|{\outVar}(t)-\outVarDes(t)\|}^2_{\R^{\outVarDim}}+ \lambda\,{\|\inpVar(t)\|}_{\R^\inpVarDim}^2 \dt + \lonereg\,{\|\inpVar \|}_{L^1(\timeInt_n, \R^\inpVarDim)}
    \end{aligned}
\end{align}
subject to the box constraints $\inpVar\in \Uad$ and to $\outVar=\cOper\state$ and $(\state, \inpVar)$ solving \eqref{eq: parabolic_problem}. Here, $ \lambda, \lonereg>0$ are regularization parameters. To cast \eqref{eqn:optimization_example} in the general form of $J_n$ in \eqref{eqn:cost:fun}, we define~$g_n$ as
\begin{equation}\nonumber
    g_n(\inpVar) \coloneqq \lonereg \|\inpVar \|_{L^1(\timeInt_n,\R^\inpVarDim)} + \mathbbm{1}_{\Uad}(u),
\end{equation}
where the indicator function $\mathbbm{1}_{\Uad}$ is given as
\begin{equation}\nonumber
    \mathbbm{1}_{\Uad}(u) \coloneqq
        \begin{cases}
        0&\text{if }u\in \Uad,\\
        \infty&\text{otherwise}.
	\end{cases}
\end{equation}
It has been shown in \cite{azmi2023nonmonotone} that $g_n$ satisfies \Cref{ass:propertiesG} and therefore the optimality condition from \Cref{thm:FONC:MPC} is applicable.
%
\section{Certified approximations for model predictive control}
\label{sec:approximation}
For practical computations, it is necessary to approximate the switched abstract control system~\eqref{eqn:weakForm}. This work focuses on Galerkin approximations of~\eqref{eqn:weakForm} using a linear subspace $\Vred \subseteq V$. As a consequence, a crucial aspect becomes the estimation of the error between the \MPC feedback law and the resulting \MPC trajectory of~\eqref{eqn:MPCsubproblem} computed for the space $V$ and for the subspace $\Vred$. This section is devoted to providing such error estimates. 
The reduced switched system obtained via Galerkin projection is given by
\begin{equation}
	\label{eqn:weakForm:reduced}
    \left\{\quad
    \begin{aligned}
        m_{\switch}(\dot{\stateRed},\varphi) + a_{\switch}(\stateRed,\varphi) &= b_{\switch}(\inpVar,\varphi) && \text{for all $\varphi\in \Vred$, a.e. in $\timeInt_n$},\\
       \stateRed(t_n) &= \stateRed_n && \text{in }H, \\
       \outVarRed &= \cOper_{\switch}\stateRed && \text{a.e. in $\timeInt_n$}, 
    \end{aligned}\right.
\end{equation}
for $\intermediateInitStateRed\in \Vred$. We refer to~\eqref{eqn:weakForm:reduced} as the \ROM for~\eqref{eqn:weakForm}. Denoting the corresponding solution spaces as $\Xred_{n,\switch,i} \subseteq\X_{n,\switch,i}$, $\Xred_n \subseteq\X_n$, and $\Xred_{n,\switch} \subseteq\X_{n,\switch}$, and assuming that the original system~\eqref{eqn:weakForm} satisfies \Cref{ass:switchedSystem}, then the reduced system~\eqref{eqn:weakForm:reduced} also satisfies \Cref{ass:switchedSystem} by construction. In particular, we can apply \Cref{cor:wellPosednessForward} to introduce the reduced solution operator $\reduce{\calS}\colon H\times \U\to \Xred_n$ and the reduced cost function $\reduce{\Jhat}_n(\inpVar)\coloneqq J_n(\cOper_\switch\reduce{\calS}(\intermediateInitStateRed,\inpVar),\inpVar)$. The resulting finite-horizon \ROM-\MPC subproblem reads
\begin{equation}
	\label{eqn:MPCsubproblem:red}
	\min \reduce{\Jhat}_n(\inpVar)\quad\text{s.t.}\quad u\in \U_n.
\end{equation}
Using \Cref{thm:FONC:MPC}, we immediately conclude the existence of a unique optimal control $\optInpVarRedN \in \U_n$, satisfying the reduced optimality system.
\begin{corollary}[First-order optimality condition for \eqref{eqn:MPCsubproblem:red}]
    \label{cor:reducedMPCoptimalitySystem}
    Under \Cref{ass:switchedSystem,ass:propertiesG} the \ROM-\MPC subproblem \eqref{eqn:MPCsubproblem:red} admits a unique solution $\optInpVarRedN \in \U_n$ satisfying the following sufficient first-order optimality condition: $\optInpVarRedN$ is optimal for \eqref{eqn:MPCsubproblem:red} if and only if there exist the state $\optStateRed\in \Xred_n$ and adjoint state $\optAdStateRed \in \Xred_{n,\switch}$ solving the coupled system
 \begin{subequations}
       \label{eqn:MPCsubproblem:red:FONC}
        \begin{empheq}[left=\left\{,right=\right.]{align}
            \label{eq: reduced state mpc subproblem 1}
            m_{\switch}(\ddt\optStateRed,\varphi) + a_{\switch}(\optStateRed,\varphi) &= b_{\switch}(\optInpVarRedN,\varphi)&& \text{a.e. in } \timeInt_n,\\
            \label{eq: reduced state mpc subproblem 2}
            \optStateRed(t_n) &= \intermediateInitStateRed,&& \text{in $H$}\\ 
            \label{eq: reduced state mpc subproblem 3}
            \optOutVarRed &= \cOper_{\switch}\optStateRed&&\text{{a.e} in $\timeInt_n$,}\\
            \label{eq: reduced adjoint mpc subproblem 1}
            -m_{\switch}(\ddt\optAdStateRed, \varphi) + a_{\switch}(\varphi,\optAdStateRed) &= \langle \optOutVarRed - \outVarDes, \cOper_{\switch}\varphi\rangle_{\R^{\outVarDim}}&& \text{a.e.~in $\timeInt_n$,} \\
            \label{eq: reduced adjoint mpc subproblem 2}
            m_{\switch(T_n)}(\optAdStateRed(T_n),\varphi) &= \mu\langle \optOutVarRed(T_n) - \outVar_{T_n},\cOper_{\switch(T_n)} \varphi\rangle_{\R^{\outVarDim}}&&\\  
            \label{eq: reduced adjoint mpc subproblem 3}
             m_{\switch(t_{i-1})}(\varphi, \optAdStateRed(t_{i}^-))&=m_{\switch(t_{i})}(\varphi,\optAdStateRed(t_{i}^+))&& \text{for all $t_{n,i}\in\calT_n$},\\  
            \label{eq: reduced state mpc subproblem gradient}
            {\langle \lambda (\optInpVarRedN - \inpVarDes) + \bOper_{\switch}'\optAdStateRed, v - \optInpVarRedN\rangle}_{\U_n} &\geq g_n(\optInpVarRedN) - g_n(v), &&
        \end{empheq}
    \end{subequations}
    for all $\varphi\in \Vred$ and all $v\in \U_n$. The reduced adjoint solution operator is denoted as 
    \begin{equation}\nonumber
        \reduce{\solA}\colon C(\closure{\timeInt}_n, \R^\outVarDim)\to \reduce{\X}_{n,\switch}, \quad 
        \outVarRed \mapsto \reduce{\solA}(\outVarRed) = \adStateRed
    \end{equation} such that \eqref{eq: reduced adjoint mpc subproblem 1}-\eqref{eq: reduced adjoint mpc subproblem 3} can be written as  $\optAdStateRed=\reduce\solA(\optOutVarRed)$.
\end{corollary}

First, in \Cref{subsec:preliminaries:error_est}, we define the energy norms and the corresponding spaces used for the error estimation. A-posteriori estimates for the optimal control under perturbed initial values are introduced in \Cref{subsec:optimalcontrolapost}. In \Cref{subsec:apostStateAdjoint}, we investigate residual-based a-posteriori error estimates for the state and the adjoint variable, which can be used to obtain a cheaply computable control error estimator. Further, we prove bounds for the optimal state. Eventually, we apply the results to the \MPC setting in \Cref{subsec:applicationToMPC}: we develop two certified \ROM-\MPC algorithms with a-posteriori bounds for the \ROM-\MPC feedback law and the \ROM-\MPC closed-loop trajectories. The main result is that the trajectory of both \ROM-\MPC algorithms stays in a neighborhood of the \FOM-\MPC trajectory, and the exactness of the \ROM can control the size of the neighborhood; see \Cref{cor:aposterioriBoundMPC}.

\begin{remark}
    As standard in the \MOR literature, we will introduce an intermediate space $\check{V}\subseteq V$ stemming from a high-dimensional finite element approximation. The corresponding model is called \emph{truth model} with the assumption that the error introduced by restricting to $\check{V}$ is negligible. In contrast, the space $\Vred\subseteq \check V$ will be derived using \MOR techniques as we discuss in \Cref{subsec:POD}. We do not distinguish the spaces $\check{V}$ and $V$ for notational convenience.
\end{remark}

\subsection{Preliminaries}\label{subsec:preliminaries:error_est}
For the error estimation, we additionally assume that the bilinear forms $a_i$ induce inner products in $V$.
\begin{assumption}
    \label{ass:switchedSystem:A:coerciv}
    The bilinear forms $a_i\colon V\times V$ are uniformly coercive, i.e., $\coerAH = 0$ in \eqref{eqn:properties:a}.
\end{assumption}
Note that this is done without loss of generality, by the standard transformation given in \cite[Cha.~III, Sec.~1.2]{lions1971optimal} or \cite[Sec.~2.3]{DieG23}.
Next, let us introduce suitable norms to measure the error.
\begin{definition}[Energy norms]\label{def:energy:norm}
   For $\mathfrak{a}\in\{a_i,m_i \mid i = 1, \ldots, \nrModes\}$ and $\phi,\varphi\in W\in\{ V,H\}$, let us introduce:
   \begin{enumerate}
       \item the energy product $\langle\varphi,\phi\rangle_{\mathfrak{a}}\vcentcolon= \tfrac{1}{2}\mathfrak{a}(\varphi,\phi)+\tfrac{1}{2}\mathfrak{a}(\phi,\varphi)$ with $W=V$ if $\mathfrak{a} = a_i$ or $W=H$ if $\mathfrak{a}= m_i$, for any $i\in \{1,\ldots,\nrModes\}$;
       \item the energy norm $\|\cdot\|^2_{\mathfrak{a}}\vcentcolon= \langle\cdot\,,\cdot\rangle_{\mathfrak{a}}=\mathfrak{a}(\cdot,\cdot)$.
   \end{enumerate}
   We write $W_{\mathfrak{a}}$ if the space $W$ is  equipped with the norm $\|\cdot\|_{\mathfrak{a}}$. Similarly, for positive constants $\omega_0,\ldots,\omega_{\nrSwitches-1}$, $t\in\timeInt_n$, and $\mathfrak{a}_{\switch}\in\{a_{\switch},m_{\switch}\}$ for some $\switch \in \addSwitching$, we define the time-dependent energy norm $\|v\|_{\mathfrak{a}_{\switch}, \omega,t}$
    for $v\in L^2((t_n,t), W)$ as
    \begin{equation}\nonumber
        \|v\|_{\mathfrak{a}_{\switch}, \omega,t}^2 = 
        \sum\limits_{i=0}^{N_t-1}\omega_i\int_{t_{n,i}}^{\min\{t,t_{n,i+1}\}}\|v(s)\|_{\mathfrak{a}_{\switch(s)}}^2 \ds,
    \end{equation}
    where $N_t\in\{0,\ldots,N-1\}$ is the smallest integer with $t\in [t_{n,N_{t}},t_{n,{N_t}+1})$. The corresponding space is denoted with $L^2((t_n,t), W_{\mathfrak{a}_{\switch}})$. The corresponding dual norm and space are defined for $v\in L^2((t_n,t), W')$ as
    \begin{equation}\label{eqn:dualsiwtchnorm}
        {\|v\|}_{\mathfrak{a}_\switch', \omega,t}^2\vcentcolon= 
        \sum\limits_{i=0}^{N_t-1}\omega_i\int_{t_{n,i}}^{\min\{t,t_{n,i+1}\}}\|v(s)\|^2_{\mathfrak{a}_{\switch(s)}'} \ds \qquad \text{and} \qquad L^2((t_n,t), W'_{\mathfrak{a}_{\switch}}),
    \end{equation}
    respectively.
    If $t=T_n$ we simply write $ \|v\|_{\mathfrak{a}_\switch,\omega}^2$ and  $\|v\|_{\mathfrak{a}_\switch',\omega}^2$. If also $\omega_i = 1$ for all $i$, then we write $ \|v\|_{\mathfrak{a}_\switch}^2$ and  $\|v\|_{\mathfrak{a}_\switch'}^2$. 
\end{definition}
By construction, we have $\|\varphi\|_{m_i}^2 \leq \contMi{i}\,\|\varphi\|_H^2 \leq\smash{\tfrac{\contMi{i}}{\coerMi{j}}}\|\varphi\|_{m_j}^2$ 
and thus we can define the equivalence constants for all $i,j\in \{1,\ldots,\nrModes\}$ as
\begin{align}
    \label{eqn:def:cij}
    c_{i,j} \vcentcolon= \begin{cases}
            \tfrac{\contMi{i}}{\coerMi{j}} &\quad\text{if }i\neq j,\\
            1&\quad\text{if }i=j.
        \end{cases}
\end{align}
Due to the equivalence of norms we use, the linear operators $\bOper_{\switch}':\U_n\to L^2(\timeInt_n, V_{a_\switch}) $ and $\cOper_{\switch}:{L^2(\timeInt_n,\R^{\outVarDim})} \to {L^2(\timeInt_n, H_{m_\switch}})$ are uniformly bounded, i.e., there exist constants $\gamma_{\bOper_{\switch},{a}_{\switch}}>0$, $\gamma_{\cOper_{\switch},{m}_{\switch}}>0$ such that
    \begin{equation}
        \label{eqn:properties:c:en:norm2}
        {\|\bOper_{\switch}'\varphi\|}_{\U_n}\; \leq\; \gamma_{\bOper_{\switch},{a}_{\switch}}{\|\varphi\|}_{{a}_{\switch}},\quad{\|\cOper_{\switch}\phi\|}_{L^2(\timeInt_n,\R^{\outVarDim})}\; \leq\; \gamma_{\cOper_{\switch},{m}_{\switch}} {\|\phi\|}_{{m}_{\switch}},
    \end{equation}
for all $\varphi\in L^2(\timeInt_n,V)$ and $\phi\in L^2(\timeInt_n,H)$. Moreover, for $\mathfrak{a}\in\{m_1,\ldots,m_{\nrSwitches}\}$ the operators $\cOper_i$ are also uniformly bounded, i.e.,
    \begin{equation}
            \label{eqn:properties:c:enrgy:norm}
            {\|\cOper_i\phi\|}_{\R^{\outVarDim}} \;\leq\; \gamma_{\cOper_i,\mathfrak{a}}\,{\|\phi\|}_{\mathfrak{a}},
    \end{equation}
for all $\phi\in H$ and all $i\in \{1,\ldots,\nrModes\}$.
\subsection{A-posteriori error estimates for the optimal control under perturbed initial values}
\label{subsec:optimalcontrolapost}
In this section, we want to obtain error bounds for the optimal controls $\optInpVarRedN, \optInpVarN$. In the analysis, we take into account that for the computation of the control $\optInpVarRedN$ only an approximation $\intermediateInitState\in H$ of the true state $\state_n$ might be available, which we assume to satisfy the error bound
\begin{equation}\label{eqn:stateInitEst}
    \|\state_n-\intermediateInitState\|_{m_{\switch(t_n)}}\leq \initBound. 
\end{equation}
In our case, $\initBound$ is determined by the error of earlier \MPC iterations. In other applications, $\initBound$ may arise from differences between measured and simulated initial values at $t_n$. Now we can specify the \ROM initial value $\stateRed_n= \projVred\intermediateInitState\in\Vred$ with orthogonal projection $\projVred\colon H \to \Vred\subseteq H$. 
\begin{remark}
    We emphasize that we do not require the reduced initial value to be the projection of the full initial value, i.e., we allow $\stateRed_n \neq \projVred \state_n$, since in general only $\intermediateState_n\neq \state_n$ will be known during the reduced \MPC iterations.        
\end{remark}
To derive a-posteriori error estimates for the optimal control, we introduce the intermediate state $\intermediateState=\solper(\intermediateInitState,\optInpVarRedN) \in \X_n$  and the intermediate output $\intermediateOutVar=\cOper_\switch\intermediateState\in C(\closure{\timeInt}_n,\R^p)$, which solves the original switched system \eqref{eqn:weakForm} with the reduced optimal control $\optInpVarRedN$ as input and perturbed initial value $\intermediateInitState\in H$.
Further, we introduce the intermediate adjoint states $\intermediateAdStateA=\solA(\intermediateOutVar)\in \X_{n,\sigma}$ and $\intermediateAdStateB=\calA(\optOutVarRed)\in \X_{n,\sigma}$ being the solutions of the adjoint equation \eqref{eqn:AdjointSystem} with input $\intermediateOutVar$ and $\optOutVarRed$ respectively.
An overview of all relevant variables and constants and their definition is given in \Cref{tab:overviewVars}. Let us start with quantifying the distance of the optimal adjoint $\optAdState$ from $\intermediateAdStateA$ and $\intermediateAdStateB$. 
\begin{table}[t]
    \centering
    \begin{tabular}{ll}
        \toprule
        \textbf{name} & \textbf{description}\\\midrule
        $\state_n$, $\intermediateInitState$, $\stateRed_n$ & \FOM initial value, intermediate initial value satisfying \eqref{eqn:stateInitEst} and \ROM initial\\
        & value satisfying $\stateRed_n= \projVred\intermediateInitState$\\
        $\optInpVarN,\optState,\optAdState,\optOutVar$ & \FOM optimal control, state, adjoint state, and output solving~\eqref{eqn:MPCsubproblem:FONC}, i.e.,\\ & $\optState=\solper(\state_n,\optInpVarN),\,\optOutVar=\cOper_\switch\optState,\, \optAdState=\solA(\optOutVar)$ \\
        $\optInpVarRedN,\optStateRed,\optAdStateRed,\optOutVarRed$ & \ROM optimal control, state, adjoint state, and output solving~\eqref{eqn:MPCsubproblem:red:FONC}, i.e.,  \\
        & $\optStateRed=\reduce \solper(\stateRed_n,\optInpVarRedN),\, \optOutVarRed=\cOper_\switch\optStateRed,\,{\optAdStateRed}=\reduce \solA(\optOutVarRed)$ \\
        $\intermediateState,\intermediateOutVar$ & \FOM state and output solving \eqref{eqn:weakForm} with input $\optInpVarRedN$ and initial value $\intermediateInitState$, i.e.,\\
        & $\intermediateState=\solper(\intermediateInitState,\optInpVarRedN)$, $\intermediateOutVar=\cOper_\switch\intermediateState$\\
        $\intermediateAdStateA$, $\intermediateAdStateB$ & \FOM adjoint states solving \eqref{eqn:AdjointSystem} with input $\intermediateOutVar$ or $\optOutVarRed$, i.e.,
        $\intermediateAdStateA=\solA(\intermediateOutVar)$, $\intermediateAdStateB=\solA(\optOutVarRed)$\\\bottomrule
    \end{tabular}
    \caption{Overview of the variables used in the certification framework}
    \label{tab:overviewVars}
\end{table}
\begin{lemma}\label{lem:relationIntermeadiateAdjoints}
Let \Cref{ass:switchedSystem,ass:propertiesG,ass:switchedSystem:A:coerciv} be valid and define
 \begin{equation}\label{eqn:err:adj:A:B}
     e_{\intermediateAdStateA} \vcentcolon= \intermediateAdStateA -\optAdState \in \X_{n,\switch},\quad e_{\intermediateAdStateB} \vcentcolon= \intermediateAdStateB -\optAdState\in\X_{n,\switch}.
 \end{equation}
 Then
\begin{subequations}
    \begin{align}
        \|e_{\intermediateAdStateA} (t_n)\|^2_{m_{\switch(t_n)}} &
        \leq \tfrac{\gamma^2_{\cOper_{\switch},{a}_{\switch}}}{2} \|\intermediateOutVar - \optOutVar \|^2_{L^2} + \mu^2 \gamma^2_{\cOper_{\sigma(T_n)},m_{\switch(T_n)}}\| \intermediateOutVar(T_n) - \optOutVar(T_n)\|^2_{\R^{\outVarDim}},
        \label{eqn:opt:adj:adj:hat}\\
        \|e_{\intermediateAdStateB}(t_n)\|^2_{m_{\switch(t_n)}} &
        \leq \tfrac{\gamma^2_{\cOper_{\switch},{a}_{\switch}}}{2} \|\optOutVarRed - \optOutVar\|^2_{L^2}
        +\mu^2  \gamma^2_{\cOper_{\sigma(T_n)},m_{\switch(T_n)}}\| \optOutVarRed(T_n) - \optOutVar(T_n)\|^2_{\R^{\outVarDim}},
        \label{eqn:opt:adj:adj:che}
  \end{align}
  \ReA{for $\gamma_{\cOper_{\switch},{a}_{\switch}}$, $\gamma_{\cOper_{\sigma(T_n)},m_{\switch(T_n)}}$ defined in \eqref{eqn:properties:c:en:norm2} and \eqref{eqn:properties:c:enrgy:norm}, respectively}.
\end{subequations}
\end{lemma}
\begin{proof}
    We show the proof for \eqref{eqn:opt:adj:adj:hat} since \eqref{eqn:opt:adj:adj:che} follows by repeating the same arguments. Due to linearity, $e_{\intermediateAdStateA}$ satisfies 
    \begin{subequations}
	\begin{empheq}[left=\left\{\quad,right=\right.]{align}
		- m_{\switch(t)}(\ddt e_{\intermediateAdStateA}(t),\varphi) + a_{\switch(t)}(\varphi,e_{\intermediateAdStateA}(t)) &= \langle \intermediateOutVar(t) - \optOutVar(t),\cOper_{\switch(t)}\varphi\rangle_{\R^{\outVarDim}}\label{eqn:err:adj:hat:a}\\
		m_{\switch(T_n)}(e_{\intermediateAdStateA}(T_n),\varphi) &= \mu\,\langle \intermediateOutVar(T_n) - \optOutVar(T_n),\cOper_{\switch(T_n)}\varphi\rangle_{\R^{\outVarDim}}\label{eqn:err:adj:hat:b}\\
        m_{\switch(t^{-}_{n,i})}(\varphi,e_{\intermediateAdStateA}(t_{n,i}^-))&=m_{\switch(t^+_{n,i})}(\varphi,e_{\intermediateAdStateA}(t_{n,i}^+))\label{eqn:err:adj:hat:c}
	\end{empheq}
    \end{subequations}
    for all $\varphi \in V$. Due to \eqref{eqn:err:adj:hat:c}, the map {$(\timeInt_n\ni t\mapsto m_{\switch(t)}(e_{\intermediateAdStateA}(t),\cdot)\in V')$} is continuous on $\closure \timeInt_n$ and $e_{\intermediateAdStateA}$ admits a weak derivative w.r.t. the time-dependent inner product $m_\switch$. Therefore we have
    \begin{align*}
       \int_{t_n}^{T_n} m_{\switch(t)}(\ddt e_{\intermediateAdStateA}(t),e_{\intermediateAdStateA}(t)) {\dt}&=\tfrac{1}{2} \int_{t_n}^{T_n}\ddt m_{\switch(t)}(e_{\intermediateAdStateA}(t),e_{\intermediateAdStateA}(t)) \dt\\
       &={\tfrac{1}{2}\|e_{\intermediateAdStateA}(t_n)\|^2_{m_{\switch(t_n)}}-\tfrac{1}{2}\|e_{\intermediateAdStateA}(T_n)\|_{m_{\switch(T_n)}}^2}.
    \end{align*}
    Now, {choosing $\varphi=e_{\intermediateAdStateA}(t)$ in} \eqref{eqn:err:adj:hat:a}, and integrating over time, we get
 \begin{align}\label{eqn:adj:hat:nor}
    \begin{aligned}
		\tfrac{1}{2}\|e_{\intermediateAdStateA}(t_n)\|^2_{m_{\switch(t_n)}}+\|e_{\intermediateAdStateA}\|^2_{a_{\switch}} 
        = \int_{t_n}^{T_n} \langle \cOper_{\switch(t)}'(\intermediateOutVar(t) - \optOutVar(t)), e_{\intermediateAdStateA}(t)\rangle_{V',V} \dt
        + \tfrac{1}{2} \|e_{\intermediateAdStateA}(T_n)\|_{m_{\switch(T_n)}}^2.
  \end{aligned}
	\end{align}
Using Cauchy-Schwarz, the weighted Young inequality with $\varepsilon_1=\nicefrac{\gamma^2_{\cOper_{\switch},\mathfrak{a}_{\switch}}}{2}$, and \eqref{eqn:properties:c:en:norm2}
 \begin{align}\label{eqn:adj:hat:t}
    \begin{aligned}
		\int_{t_n}^{T_n}& \langle \cOper_{\switch(t)}'(\intermediateOutVar(t) - \optOutVar(t)), e_{\intermediateAdStateA}(t)\rangle_{V',V} \dt\;\le\;\int_{t_n}^{T_n}\|\intermediateOutVar(t) - \optOutVar(t)\|_{\R^{\outVarDim}}\|\cOper_{\switch(t)} e_{\intermediateAdStateA(t)}\|_{\R^{\outVarDim}}\dt\\
  \le&\;\tfrac{1}{2}\left(\varepsilon_1\|\intermediateOutVar - \optOutVar\|^2_{L^2}+\frac{1}{\varepsilon_1}\|\cOper_{\switch}e_{\intermediateAdStateA}\|^2_{L^2}\right)
  \leq \;\tfrac{1}{2}\left(\tfrac{\gamma^2_{\cOper_{\switch},{a}_{\switch}}}{2}\|\intermediateOutVar - \optOutVar\|^2_{L^2}+2\|e_{\intermediateAdStateA}\|^2_{a_{\switch}}\right).
  \end{aligned}
	\end{align}
 Similarly, the terminal condition \eqref{eqn:err:adj:hat:b} gives with Young's inequality for $\varepsilon_2=\gamma^2_{\cOper_{\sigma(T_n)},m_{\switch(T_n)}}$
    \begin{align}\label{eqn:adj:fin:t}
      \begin{aligned}
        \|e_{\intermediateAdStateA}(T_n)\|^2_{m_{\switch(T_n)}} \;\leq&\;  \mu \| \intermediateOutVar(T_n) - \optOutVar(T_n)\|_{\R^{\outVarDim}}\|\cOper_{\switch(T_n)}e_{\intermediateAdStateA}(T_n)\|_{\R^{\outVarDim}}\\
        \leq&\;\tfrac{1}{2}\left(\mu^2 \varepsilon_2\| \intermediateOutVar(T_n) - \optOutVar(T_n)\|^2_{\R^{\outVarDim}}+\frac{1}{\varepsilon_2}\|\cOper_{\switch(T_n)}e_{\intermediateAdStateA}(T_n)\|_{\R^{\outVarDim}}\right)\\
        \leq&\;\tfrac{1}{2}\left( \mu^2\gamma^2_{\cOper_{\sigma(T_n)},m_{\switch(T_n)}}\| \intermediateOutVar(T_n) - \optOutVar(T_n)\|^2_{\R^{\outVarDim}}+\|e_{\intermediateAdStateA}(T_n)\|^2_{m_{\switch(T_n)}}\right),
     \end{aligned}
    \end{align}
    where we used \eqref{eqn:properties:c:enrgy:norm}. Plugging \eqref{eqn:adj:hat:t} and \eqref{eqn:adj:fin:t} inside \eqref{eqn:adj:hat:nor} gives \eqref{eqn:opt:adj:adj:hat}.
\end{proof}
With these preparations, we can state a bound for the error in the optimal control.
\begin{theorem}[A-posteriori estimate for the optimal control]
	\label{thm:apostErrorControl}
   Let \Cref{ass:switchedSystem,ass:propertiesG,ass:switchedSystem:A:coerciv} be satisfied and let $\optInpVarN \in \U_n$ and $\optInpVarRedN \in \U_n$ be the unique solutions of the \FOM-\MPC problem \eqref{eqn:MPCsubproblem} and the \ROM-\MPC problem \eqref{eqn:MPCsubproblem:red}, respectively. {Further, let $\initBound$ be defined in \eqref{eqn:stateInitEst}}. Then
   \begin{equation}
        \label{eq: a post estimate general all adjoint}
        \|\optInpVarRedN - \optInpVarN\|^2_{\U_n}
    	\leq \min\{ \BoundControlA^2(\optInpVarRedN, \initBound), \BoundControlB^2(\optInpVarRedN, \initBound)\}
    \end{equation}
     with
     \begin{align}
        \BoundControlA^2(\optInpVarRedN, \initBound) &\vcentcolon= \tfrac{1}{\lambda^2}\|\mathcal B_{\sigma}'(\optAdStateRed- \intermediateAdStateA)\|^2_{\U_n} + \tfrac{C_1}{2\lambda}\initBound^2,\label{control:bound:A}\\
        \BoundControlB^2(\optInpVarRedN, \initBound) &\vcentcolon= \tfrac{1}{\lambda^2} \|\bOper_{\switch}' (\optAdStateRed-\intermediateAdStateB)\|_{\U_n}^2  + \tfrac{1}{\lambda}\|\optOutVarRed - \intermediateOutVar\|_{L^2}^2+ \tfrac{\mu}{\lambda}\|\optOutVarRed(T_n)-\intermediateOutVar(T_n)\|_{\R^{\outVarDim}}^2+\tfrac{C_1}{\lambda}\initBound^2,\label{control:bound:B}
     \end{align}
     for $C_1\vcentcolon=\max\{\nicefrac{\gamma^2_{\cOper_{\switch},{a}_{\switch}}}{2},\mu \gamma^2_{\cOper_{\sigma(T_n)},m_{\switch(T_n)}}\}$ and $\intermediateOutVar$, $\intermediateAdStateA$, $\intermediateAdStateB$ as in \Cref{tab:overviewVars}.
\end{theorem} 
\begin{proof}
	Setting $v = \optInpVarN\in\U_n$ in \eqref{eq: reduced state mpc subproblem gradient} and $v = \optInpVarRedN\in\U_n$ in \eqref{eq: state mpc subproblem gradient} and adding the equations yields for $p\in \{\intermediateAdStateA, \intermediateAdStateB \}$
	\begin{align}
		\label{eqn:gradientEstimate}
		\begin{aligned}
		\lambda \|\optInpVarRedN - \optInpVarN\|^2_{\U_n} &\leq \langle \bOper_{\switch}'(\optAdStateRed-\optAdState), \optInpVarN-\optInpVarRedN \rangle_{\U_n}\\
		&= \langle \bOper_{\switch}'(\optAdStateRed-p), \optInpVarN-\optInpVarRedN \rangle_{\U_n} + \langle \bOper_{\switch}'(p-\optAdState), \optInpVarN-\optInpVarRedN \rangle_{\U_n}.
		\end{aligned}
	\end{align}
    Choosing $p=\intermediateAdStateA$ leads to the bound $\BoundControlA$, while the choice $p=\intermediateAdStateB$ results in the bound $\BoundControlB$.
    We first focus on the bound {$\BoundControlB$}.
	For the second term in \eqref{eqn:gradientEstimate} we test the state equation for $\optState$ and $\intermediateState$ against 
 $\varphi = e_{\intermediateAdStateB}(t)$ defined in \eqref{eqn:err:adj:A:B}, this lead to
	\begin{align}\label{eqn:controlBoundPorrf:DualityArgument}
 \begin{aligned}
		\langle \bOper_{\switch}'(\intermediateAdStateB-&\optAdState), \optInpVarN-\optInpVarRedN \rangle_{\U_n}= \int_{t_n}^{T_n} b_{\switch(t)}(\optInpVarN(t) - \optInpVarRedN(t),e_{\intermediateAdStateB}(t)) \dt\\
    =& \int_{t_n}^{T_n} m_{\switch(t)}\big(\ddt \big(\optState(t) - \intermediateState \big), e_{\intermediateAdStateB}(t)\big) + a_{\switch(t)}(\optState(t) - \intermediateState(t), e_{\intermediateAdStateB}(t)) \dt\\
    =& \int_{t_n}^{T_n} -m_{\switch(t)}\left(\optState(t)-\check{\state}(t), \ddt e_{\intermediateAdStateB}(t)\right) + a_{\switch(t)}(\optState(t) - \check{\state}(t), e_{\intermediateAdStateB}(t)) \dt\\
		&+ m_{\switch(T_n)}(\optState(T_n) - \intermediateState(T_n), e_{\intermediateAdStateB}(T_n))-m_{\switch(t_n)}(\state_{n} - \intermediateInitState, e_{\intermediateAdStateB}(t_n))\\
     &+\sum_{t_{n,i}\in\calT_n}m_{\switch(t^{-}_{n,i})}(\optState(t_{n,i}) - \intermediateState(t_{n,i}),e_{\intermediateAdStateB}(t_{n,i}^-))-m_{\switch(t^{+}_{n,i})}(\optState(t_{n,i}) - \intermediateState(t_{n,i}),e_{\intermediateAdStateB}(t_{n,i}^+)),
     \end{aligned}
	\end{align}
	where the last equality follows from \Cref{lemma:S-IBP}. Using the adjoint system for $\optAdState=\solA(\optOutVar)$ and $\intermediateAdStateB=\solA(\optOutVarRed)$ with $\varphi=\optState(t)-\intermediateState(t)$ together with the symmetry of $m_i$ (see~\Cref{ass:switchedSystem}\,\ref{ass:switchedSystem:M}), we obtain
	\begin{equation}\label{eqn:tmpest:0}
	\begin{aligned}
		\langle \bOper_{\switch}^*(\intermediateAdStateB-\optAdState), \optInpVarN-\optInpVarRedN \rangle_{\U_n} &= \int_{t_n}^{T_n} \langle \optOutVarRed(t) - \optOutVar(t), \cOper_{\switch(t)}(\optState(t)-\intermediateState(t))\rangle_{\R^{\outVarDim}} \dt\\
		&\qquad + \mu \langle \optOutVarRed(T_n) - \optOutVar(T_n), \cOper_{\switch(T_n)}\big(\optState(T_n) - \intermediateState(T_n)\big)\rangle_{\R^{\outVarDim}}\\
  &\qquad -m_{\switch(t_n)}(\state_{n} - \intermediateInitState, e_{\intermediateAdStateB}(t_n)).
	\end{aligned}
	\end{equation}
    Combining \eqref{eqn:gradientEstimate} for $p=\intermediateAdStateB$ with \eqref{eqn:tmpest:0} yields
    \begin{align} \label{eqn:tmpest:5}
    \begin{aligned}
		\lambda \|\optInpVarRedN - \optInpVarN\|^2_{\U_n} 
		 \;\leq & \; \langle \bOper_{\switch}'(\optAdStateRed-\intermediateAdStateB), \optInpVarN-\optInpVarRedN \rangle_{\U_n} \; + \;\langle\optOutVarRed- \optOutVar, \optOutVar-\intermediateOutVar\rangle_{L^2}\\ 
        +&\;\mu\langle  \optOutVarRed(T_n)- \optOutVar(T_n), \optOutVar(T_n)-\intermediateOutVar(T_n)\rangle_{\R^{\outVarDim}}
        +\big| m_{\switch(t_n)}(\state_{n} - \intermediateInitState, e_{\intermediateAdStateB}(t_n)) \big|.
        \end{aligned}
	\end{align}
    Further, using Cauchy-Schwarz and Young's weighted inequality with $\varepsilon_1>0$, \eqref{eqn:stateInitEst}, and \Cref{lem:relationIntermeadiateAdjoints} we find
    \begin{align}  
    \label{eqn:tempest:4}
    \begin{aligned}
        \big|m_{\switch(t_n)}(\state_{n} -& \intermediateInitState, e_{\intermediateAdStateB}(t_n))\big|\; \leq \; \| \state_{n}(t_n) - \intermediateInitState(t_n)\|_{m_{\switch(t_n)}} \|e_{\intermediateAdStateB}(t_n) \|_{m_{\switch(t_n)}}\\
        \leq& \; \tfrac{\varepsilon_1}{2}\initBound^2 + \tfrac{1}{2\varepsilon_1}\|e_{\intermediateAdStateB}(t_n) \|_{m_{\switch(t_n)}}^2\\
        \leq &\; \tfrac{\varepsilon_1}{2}\initBound^2 + \tfrac{1}{2\varepsilon_1}\left({\tfrac{\gamma^2_{\cOper_{\switch},{a}_{\switch}}}{2}}\|\optOutVarRed - \optOutVar\|^2_{L^2}
    +{\mu^2 \gamma^2_{\cOper_{\sigma(T_n)},m_{\switch(T_n)}}} \| \optOutVarRed(T_n) - \optOutVar(T_n)\|^2_{\R^{\outVarDim}}\right) .
    \end{aligned}
    \end{align}
	Adding and subtracting $\optOutVarRed$, using Cauchy-Schwartz and the weighted Young's inequality for $\varepsilon_2>0$, we get for the second term in \eqref{eqn:tmpest:5}
	\begin{align}
		\label{eqn:tmpest:2}
		\begin{aligned}
		 \langle \optOutVarRed - \optOutVar, \optOutVar-\intermediateOutVar\rangle_{L^2}
		\;=&\;-\|\optOutVarRed - \optOutVar\|_{L^2}^2 + \langle \optOutVarRed-\optOutVar, \optOutVarRed - \intermediateOutVar\rangle_{L^2}\\
		\leq&\; \left(\tfrac{\varepsilon_2}{2}-1\right)\|\optOutVarRed - \optOutVar\|_{L^2}^2 + \tfrac{1}{2\varepsilon_2}\|\optOutVarRed - \intermediateOutVar\|_{L^2}^2.
		\end{aligned}
	\end{align}
	Similarly, we can estimate the third term in \eqref{eqn:tmpest:5} for $\varepsilon_3>0$ as
	\begin{align}
		\label{eqn:tmpest:3}
  \begin{aligned}
		\mu \langle \outVarRed(T_n) - \optOutVar(T_n),& \optOutVar(T_n) -\intermediateOutVar(T_n)\rangle_{\R^{\outVarDim}} \\
     \leq&\; \mu\left(\tfrac{\varepsilon_3}{2}-1\right)\|\optOutVarRed(T_n) - \optOutVar(T_n)\|_{\R^{\outVarDim}}^2 + \mu\tfrac{1}{2\varepsilon_3}\|\optOutVarRed(T_n) - \intermediateOutVar(T_n)\|_{\R^{\outVarDim}}^2.
  \end{aligned}
	\end{align}
    Choosing 
    $ \varepsilon_1=\max\{\nicefrac{\gamma^2_{\cOper_{\switch},{a}_{\switch}}}{2},\mu \gamma^2_{\cOper_{\sigma(T_n)},m_{\switch(T_n)}}\}\eqqcolon C_1$, $\varepsilon_2=\varepsilon_3=1$ and inserting \eqref{eqn:tempest:4}, \eqref{eqn:tmpest:2}, \eqref{eqn:tmpest:3}  into \eqref{eqn:tmpest:5} leads to
	\begin{align}\label{eq:optimalOutputBound}
 \begin{aligned}
		\lambda \|\optInpVarRedN -& \optInpVarN\|^2_{\U_n} 
		 \\
   \leq &\;  \langle \bOper_{\switch}'(\optAdStateRed-\intermediateAdStateB), \optInpVarN-\optInpVarRedN \rangle_{\U_n}  +\tfrac{1}{2} \|\optOutVarRed - \intermediateOutVar\|_{L^2}^2  + \tfrac{\mu}{2}\|\optOutVarRed(T_n) - \intermediateOutVar(T_n)\|_{\R^{\outVarDim}}^2
        +\tfrac{C_1}{2}\initBound^2, 
        \end{aligned}
	\end{align}
    Finally, using Cauchy-Schwartz together with the weighted Young's inequality
    \begin{equation*}
		\nonumber 
		\langle \bOper_{\switch}'(\optAdStateRed-\intermediateAdStateB), \optInpVarN-\optInpVarRedN \rangle_{\U_n} \leq \tfrac{1}{2\lambda} \|\bOper_{\switch}'(\optAdStateRed-\intermediateAdStateB)\|_{\U_n}^2 + \tfrac{\lambda}{2} \|\optInpVarRedN - \optInpVarN\|^2_{\U_n},
	\end{equation*}
    that, plugged into \eqref{eq:optimalOutputBound}, implies the bound $\BoundControlB$.

    Now we focus on $\BoundControlA$ and thus choose $\adState=\intermediateAdStateA$ in \eqref{eqn:gradientEstimate}. Similar to in \eqref{eqn:controlBoundPorrf:DualityArgument}, we obtain for the second term in \eqref{eqn:gradientEstimate} using the definition of $\intermediateState$ and $\intermediateAdStateA$ and the corresponding equations
    \begin{equation*}
	\begin{aligned}
		\langle \bOper_{\switch}^*(\intermediateAdStateA-\optAdState), \optInpVarN-\optInpVarRedN \rangle_{\U_n} &=  \langle \optOutVar- \intermediateOutVar, \intermediateOutVar-\optOutVar\rangle_{L^2}+ \mu \langle \optOutVar(T_n) - \intermediateOutVar(T_n),\intermediateOutVar(T_n) - \optOutVar(T_n)\big)\rangle_{\R^{\outVarDim}} \\
   &\qquad- m_{\switch(t_n)}(\state_{n} - \intermediateInitState, e_{\intermediateAdStateA}(t_n)).
	\end{aligned}
	\end{equation*}
   Therefore \eqref{eqn:gradientEstimate} for $\adState=\intermediateAdStateA$ reduces to
    \begin{align} \label{eqn:optimalOutputBound2}
		\begin{aligned}
		\lambda \|\optInpVarRedN - \optInpVarN\|^2_{\U_n}\;+&\; \|\intermediateOutVar - \optOutVar\|_{L^2}^2 \;+\; \mu\|\intermediateOutVar(T_n) - \optOutVar(T_n)\|_{\R^{\outVarDim}}^2\\
  \leq &\; \langle \bOper_{\switch}'(\optAdStateRed - \intermediateAdStateA), \optInpVarN-\optInpVarRedN \rangle_{\U_n}
  +\big|m_{\switch(t_n)}(\state_{n} - \intermediateInitState, \intermediateAdStateA(t_n)-\optAdState(t_n))\big|.
		\end{aligned}
	\end{align}
    The initial term can be treated similarly as in \eqref{eqn:tempest:4} using \eqref{eqn:opt:adj:adj:hat} in \Cref{lem:relationIntermeadiateAdjoints}, i.e.,
    \begin{align}        \label{eqn:tempest:6}
   \begin{aligned}
        \big| m_{\switch(t_n)}(\state_{n} -& \intermediateInitState, e_{\intermediateAdStateA}(t_n))\big|\\
        \leq &\; \tfrac{\varepsilon_4}{2}\initBound^2 + \tfrac{1}{2\varepsilon_4}\left(\tfrac{\gamma^2_{\cOper_{\switch},{a}_{\switch}} }{2}\|\intermediateOutVar - \optOutVar\|^2_{L^2}+\mu^2 \gamma^2_{\cOper_{\sigma(T_n)},m_{\switch(T_n)}} \| \intermediateOutVar(T_n) - \optOutVar(T_n)\|^2_{\R^{\outVarDim}}\right).
        \end{aligned}
    \end{align}
 Using Cauchy-Schwartz together with the weighted Young's inequality, we find
    \begin{equation*}
		\nonumber 
		\langle \bOper_{\switch}'(\optAdStateRed-\intermediateAdStateA), \optInpVarN-\optInpVarRedN \rangle_{\U_n} \leq \tfrac{1}{2\lambda} \|\bOper_{\switch}'(\optAdStateRed-\intermediateAdStateA)\|_{\U_n}^2 + \tfrac{\lambda}{2} \|\optInpVarRedN - \optInpVarN\|^2_{\U_n},
	\end{equation*}
 that, together with \eqref{eqn:optimalOutputBound2} and \eqref{eqn:tempest:6} for $\varepsilon_4=\tfrac{C_1}{2}$ directly implies  the bound $\BoundControlA$.
\end{proof}
Note that both error bounds $\BoundControlA$ and $\BoundControlB$ require bounds on the state and adjoint output maps and the error in the initial value.
\begin{remark}\label{rem:unconstrainedAPost}
    If there is no perturbation in the initial value, i.e., $\initBound=0$, then the bound $\BoundControlA$ simplifies to
        \begin{align}\label{eq:remarkControlBound}
             \|\optInpVarN-\optInpVarRedN\|_{\U_n} &\leq \BoundControlA(\optInpVarRedN, 0) = \tfrac{1}{\lambda}\|\bOper_{\switch}'(\optAdStateRed- \intermediateAdStateA)\|_{\U_n}
        \end{align}
        and can thus be interpreted as a perturbation bound (see, e.g., {\cite{DBLP:journals/coap/TroltzschV09}}) with the perturbation $\xi = \bOper_{\sigma}'(\optAdStateRed- \intermediateAdStateA)$. However, our bounds hold for arbitrary $g_n$ satisfying \Cref{ass:propertiesG}, while in the literature only the special case $g_n=\mathds{1}_{\Uad}$ has been considered so far. In the case $g_n\equiv 0$, the bound in \eqref{eq:remarkControlBound} and the bound from {\cite{DBLP:journals/coap/TroltzschV09}} coincide. To see this, we add $\pm \lambda (\optInpVarRedN-\inpVarDes)$ in \eqref{eq: a post estimate general all adjoint} and use the unconstrained optimality condition for the \ROM-\MPC subproblem \eqref{eqn:MPCsubproblem:red:FONC}, i.e.,
        \[\nabla \reduce{\Jhat}_n(\optInpVarRedN) = \bOper_{\switch}'\optAdStateRed + \lambda (\optInpVarRedN-\inpVarDes)=0,\]
        to obtain
         \begin{align}
             \|\optInpVarN-\optInpVarRedN\|_{\U_n} &\leq \BoundControlA(\optInpVarRedN,0)=\tfrac{1}{\lambda}\|\bOper_{\switch}'(\optAdStateRed- \intermediateAdStateA)\|_{\U_n}
             = \tfrac{1}{\lambda}\|\nabla \Jhat_n (\optInpVarRedN)\|_{\U_n}\nonumber.
        \end{align}
         Note that this bound is a standard residual-based coercivity bound (cf.~{\cite{patera2007reduced}}) for the unconstrained first-order optimality condition \eqref{eq: unconstrained opt condition}. The bilinear form represents the Hessian of ${\Jhat}_n$, which possesses the coercivity constant $\lambda>0$, and the gradient $\nabla {\Jhat}_n$ acts as the residual operator, incorporating the reduced control $\optInpVarRedN$. In that sense, \eqref{eq:remarkControlBound} extends the standard coercivity residual bound.
\end{remark}
\subsection{A-posteriori error estimates for the state and adjoint}
\label{subsec:apostStateAdjoint}
After the reduced control $\optInpVarRedN$ is computed, the error bound $\BoundControlA$ in  \eqref{eq: a post estimate general all adjoint} (or $\BoundControlB$) can be evaluated by solving the \FOM equations for $\intermediateAdStateA, \intermediateState$ (or for $\intermediateAdStateB, \intermediateState$). Since the direct evaluation of these bounds depends on \FOM calculations, we will call them \emph{expensive-to-evaluate bounds}. These bounds may be suitable for certain scenarios, such as a single optimization of the open-loop problem. However, it can become costly in a multi-query context for open-loop problems, like \MPC (see \Cref{fig:mpcExample1table}).

Therefore, in this subsection, we aim to derive an upper bound whose computation is based solely on the \ROM level (see \Cref{cor:errBoundOptimalControl}). We will refer to these estimates as \emph{cheap-to-evaluate bounds}. Specifically, we focus on determining an upper bound for $\BoundControlB$, noting that a similar approach can be applied to $\BoundControlA$.
One way to achieve this is by incorporating output bounds for the primal and dual outputs, such as those derived from balanced truncation model reduction. Alternatively, one can estimate using the continuity constants of the output operators and introduce residual-based error estimates for the state and adjoint equations, which will be discussed in the next two results.
In the specific case of the switched system \eqref{eqn:weakForm}, we must account for the contributions arising from switching within the bilinear form $m_\switch$. The proof involves iteratively estimating across the switching intervals, progressing from left to right (or vice versa for the adjoint), as the residual generated by the initial (or terminal) condition is known. Along the way, it is necessary to consider the jumps in the bilinear form $m_\switch$, characterized by the constants $c_{i,j}$ defined in \eqref{eqn:def:cij}.
\begin{lemma}[A-posteriori estimate for the intermediate state]
\label{thm:errorBoundStateEnergy}
    Let \Cref{ass:switchedSystem,ass:switchedSystem:A:coerciv} be satisfied. Let $\intermediateState=\solper(\intermediateInitState,\optInpVarRedN)\in \X_n$ and $\optStateRed=\reduce\solper(\stateRed_n,\optInpVarRedN) \in \Xred_n$ be given. 
    Define the intermediate reduction error as $\err \vcentcolon= \intermediateState - \optStateRed$ and the residual $\residual\colon \timeInt_n\to V'$ as
\begin{align*}
		\residual(t) \vcentcolon= \bOper_{\switch(t)}\optInpVarRedN(t) - \mOper_{\switch(t)}\ddt{\stateRed}(t) - \aOper_{\switch(t)}\stateRed(t).
\end{align*}
Then, for any $\tilde{\nrSwitches}\in\{1,\ldots,\nrSwitches\}$, we have 
\begin{align} \label{eqn:stateBound:InductionHypothesis2}
    \begin{aligned}
    \|e(t_{n,\tilde \nrSwitches})\|_{m_{\switch(t^{-}_{\tilde \nrSwitches})}}^2 &+ \sum\limits_{i=0}^{\tilde \nrSwitches-1}\omega_{\tilde \nrSwitches,i}\int_{t_{n,i}}^{t_{n,i+1}}\|e(t)\|_{a_{\switch(t)}}^2 \dt \\
    &\leq   
    \omega_{\tilde\nrSwitches,0 } \|e(t_n)\|_{m_{\switch(t_n)}}^2 
    + \sum\limits_{i=0}^{\tilde \nrSwitches-1}\omega_{\tilde \nrSwitches,i}\int_{t_{n,i}}^{t_{n,i+1}}\|\residual(t)\|_{{a'_{\switch(t)}}}^2 \dt
    \end{aligned}
\end{align}
with $\omega_{\tilde \nrSwitches,i} = \prod_{k=i+1}^{\tilde \nrSwitches-1}c_{\switch(t^{+}_{n,k}),\switch(t^-_{n,k})} \quad \text{for }i=0,\ldots,\tilde \nrSwitches-1$,
where $c_{\switch(t^{+}_{n,k}),\switch(t^-_{n,k})}$ is defined in \eqref{eqn:def:cij}. Note that $\omega_{\tilde \nrSwitches}=(\omega_{\tilde \nrSwitches,i})_{i=0}^{\tilde \nrSwitches-1}$ depends on $n$, but we have omitted the subscript for notational convenience. In particular, for $\tilde \nrSwitches=\nrSwitches$ we define $\omega\coloneqq \omega_N$ and \eqref{eqn:stateBound:InductionHypothesis2} simplifies to
\begin{align}
    \label{eqn:stateBound:Venergy}
    \|e(T_n)\|_{m_{\switch(T_n)}}^2 +  \|\err\|_{a_{\switch},\omega}^2 
    \leq \Delta^2_{\state}(\optInpVarRedN) \vcentcolon=\omega_0 \|\err(t_n)\|_{m_{\switch(t_n)}}^2 + \|\residual\|_{a_{\switch}',\omega}^2,
\end{align}
\ReA{with the dual norm $\|\cdot\|_{a_{\switch}',\omega}$ introduced in \eqref{eqn:dualsiwtchnorm}.}
\end{lemma}
\begin{proof}
	Due to linearity, the error satisfies for $t\in\timeInt_n$
	\begin{align*}
		\left\{\quad
    	\begin{aligned}
        	m_{\switch(t)}(\dot{e}(t),\varphi) + a_{\switch(t)}(e(t),\varphi) &= \langle R(t), \varphi\rangle_{V',V},&& \text{a.e.~in $\timeInt_n$},\\ 
       		e(t_n) &= \intermediateInitState - \stateRed_n,
    \end{aligned}\right.
	\end{align*}
	for all $\varphi\in V$. Note that due to the switching in $m_\switch$, the variables $\state$, $\stateRed$, and $e$ do not admit a weak derivative in $L^2(\timeInt_n,V')$. Therefore, we have to estimate the error on each switching interval separately and concatenate the results later on. For almost all $t\in (t_{n,i},t_{n,i+1})$ and for any $i\in \{0,\ldots,N-1\}$, testing against $e(t)\in V$ results in 
    \begin{align}\nonumber
  \begin{aligned}
		\tfrac{1}{2} \ddt m_{\switch(t^{+}_{n,i})}(e(t),e(t)) + a_{\switch(t^{+}_{n,i})}(e(t),e(t)) \;=&\; m_{\switch(t^{+}_{n,i})}(\dot{e}(t),e(t)) + a_{\switch(t^{+}_{n,i})}(e(t),e(t))\\
  =&\; \langle R(t), e(t)\rangle_{V',V}.
  \end{aligned}
	\end{align}
	 Integrating from $t_{n,i}$ to $t_{n,i+1}$ leads to 
    \begin{equation}\nonumber
       \tfrac{1}{2}\|e(t_{n,i+1})\|_{m_{\switch(t^{-}_{n,i+1})}}^2 + \int_{t_{n,i} }^{t_{n,i+1}}a_{\switch(t)}(e(t),e(t))\dt  \leq \int_{t_{n,i}}^{t_{n,i+1}} \langle R(t), e(t)\rangle_{V',V}\dt + \tfrac{1}{2} \|e(t_{n,i})\|_{m_{\switch(t^{+}_{n,i})}}^2.
    \end{equation}
    Next we apply Cauchy-Schwarz and the weighted Young's inequalities to obtain
    \begin{align}\label{eqn:stateBound:tmp2}
    \begin{aligned}
       \|e(t_{n,i+1})\|_{m_{\switch(t^{-}_{n,i+1})}}^2 +& \int_{t_{n,i} }^{t_{n,i+1}}a_{\switch(t)}(e(t),e(t))\dt \\ 
       \leq& \int_{t_{n,i}}^{t_{n,i+1}} \|R(t)\|^2_{{a_{\switch(t)}}}\dt + \|e(t_{n,i})\|_{m_{\switch(t^{+}_{n,i})}}^2.
    \end{aligned}
    \end{align}
    Now we prove \eqref{eqn:stateBound:InductionHypothesis2} using the induction principle. For $\tilde \nrSwitches=1$, we have from \eqref{eqn:stateBound:tmp2} and $i=0$
    \begin{align}
		\nonumber 
		\|e(t_{n,1})\|_{m_{\switch(t^{-}_{n,1})}}^2 + \int_{t_{n} }^{t_{n,1}}a_{\switch(t)}(e(t),e(t))\dt  \leq \int_{t_{n}}^{t_{n,1}} \|R(t)\|^2_{{{a_{\switch(t)}}}}\dt + \|e(t_{n,0})\|_{m_{\switch(t^{+}_{n,0})}}^2,
	\end{align}
    which is \eqref{eqn:stateBound:InductionHypothesis2} for $\tilde \nrSwitches =1$ since $\omega_{1,0}=1$.
    Now let \eqref{eqn:stateBound:InductionHypothesis2} be valid for $\tilde \nrSwitches=j\in\{1,\ldots,\nrSwitches\}$, we need to show that it holds for $\tilde \nrSwitches=j+1$. From \eqref{eqn:stateBound:tmp2} for $i=j$ we have
    \begin{align}
		\label{eqn:stateBound:tmp3}
  \begin{aligned}
		\|e(t^{-}_{n,j+1})\|_{m_{\switch(t^{-}_{n,j+1})}}^2 +& \int_{t_{n,j} }^{t_{n,j+1}}a_{\switch(t)}(e(t),e(t))\dt \\ 
  \leq &\int_{t_{n,j}}^{t_{n,j+1}} \|\residual(t)\|_{{{a'_{\switch(t)}}}}^2 \dt +
  \|e(t_{n,j})\|_{m_{\switch(t^{+}_{n,j})}}^2.
    \end{aligned}
	\end{align}
    From \eqref{eqn:def:cij} we have $\|e(t_{n,j})\|_{m_{\switch(t^{+}_{n,j})}}^2\leq c_{\switch(t^{+}_{n,j}),\switch(t^-_{n,j})}\|e(t_{n,j})\|_{m_{\switch(t^{-}_{n,j})}}^2$. Therefore, using this last relation, we insert \eqref{eqn:stateBound:InductionHypothesis2} into \eqref{eqn:stateBound:tmp3} for $\tilde \nrSwitches=j$ so that we get
    \begin{align}\label{eqn:ind:pro}
    \begin{aligned}
    &\|e(t^{-}_{n,j+1})\|_{m_{\switch(t^{-}_{n,j+1})}}^2 + \int_{t_{n,j} }^{t_{n,j+1}}a_{\switch(t)}(e(t),e(t))\dt\\
    \leq& \int_{t_{n,j}}^{t_{n,j+1}} \|R(t)\|^2_{a_{\switch(t)}'}\dt -c_{\switch(t^{+}_{n,j}),\switch(t^-_{n,j})} \sum\limits_{i=0}^{ j-1}\omega_{j,i}\int_{t_{n,i}}^{t_{n,i+1}}a_{\switch(t)}(e(t),e(t)) \dt\\
    +  &c_{\switch(t^{+}_{n,j}),\switch(t^-_{n,j})}\left( \prod\limits_{i=1}^{j-1}c_{\switch(t^{+}_{n,i}),\switch(t^{-}_{n,i})} \|e(t_n)\|_{m_{\switch(t_n)}}^2 
    + \sum\limits_{i=0}^{j-1}\omega_{j,i}\int_{t_{n,i}}^{t_{n,i+1}}\|\residual(t)\|_{{{a'_{\switch(t)}}}}^2 \dt\right).
    \end{aligned}	
     \end{align}
    Then, reordering the terms in \eqref{eqn:ind:pro} and using $\omega_{j+1,j}=1$, $c_{\switch(t^{+}_{n,i}),\switch(t^{-}_{n,i})} \omega_{j,i}=\omega_{j+1,i}$, we obtain \eqref{eqn:stateBound:InductionHypothesis2} for $\tilde \nrSwitches=j+1$. Finally, \eqref{eqn:stateBound:Venergy} directly follows from \eqref{eqn:stateBound:InductionHypothesis2} for $\tilde \nrSwitches=\nrSwitches$ using the norms introduced in \Cref{def:energy:norm}. \qedhere
\end{proof}
\begin{remark}
    Note that if there is no switching in the bilinear form $m_\switch$, then all the constants $c_{\switch_i, \switch_{i+1}}$ in \eqref{eqn:stateBound:Venergy} become equal to $1$, simplifying the estimate to the residual-based bound in the energy norm.
\end{remark}
Next, we will focus on the adjoint state. In this case, we must also account for the residual of the continuity condition \eqref{eq: reduced adjoint mpc subproblem 3} at the switching times.
\begin{lemma}[A-posteriori error estimate for the intermediate adjoint state]
	\label{thm:errorBoundAdjointEnergy}
	Let \Cref{ass:switchedSystem,ass:switchedSystem:A:coerciv} be satisfied. Let $\intermediateAdStateB=\solA(\optOutVarRed) \in \X_{n,\switch}$ and $\optAdStateRed=\reduce \solA(\optOutVarRed) \in \Xred_{n,\switch}$ 
 and let us define the error $\addErr \vcentcolon= \intermediateAdStateB - \optAdStateRed$ and the residuals:
	\begin{align*}
 \begin{aligned}
		\adResidual_{T_n} \;\vcentcolon=&\; -\mOper_{\switch(T_n)} \optAdStateRed(T_n) + \mu \cOper_{\switch(T_n)}' (\optOutVarRed(T_n) - \outVar_T), \\
        \adResidual_{t_{n,i}} \;\vcentcolon=&\; -\mOper_{\switch(t_{n,i}^-)} \optAdStateRed(t_{n,i}^-) +\mOper_{\switch(t_{n,i}^+)} \optAdStateRed(t_{n,i}^+) \quad \text{for all } t_{n,i}\in\calT_n,\\
		\adResidual \colon \timeInt_n\to V',\quad t\mapsto \adResidual(t) \;\vcentcolon=&\; \cOper_{\switch(t)}' (\outVarDes(t) - \optOutVarRed(t)) + \mOper_{\switch(t)}' \ddt{\optAdStateRed}(t) - \aOper_{\switch(t)}' \optAdStateRed(t).
  \end{aligned}
	\end{align*}
    Then, for any $\tilde{\nrSwitches}\in\{0,\ldots,\nrSwitches-1\}$, we have
    \begin{align}\label{eqn:adj:est:rel}
    \begin{aligned}             \|\addErr(t_{n,\tilde{\nrSwitches}}^+)\|_{m_{\switch(t^+_{n,\tilde{\nrSwitches}})}}^2 +& \sum_{i=\tilde \nrSwitches }^{\nrSwitches-1}\tilde \omega_{\tilde \nrSwitches,i}\int_{t_{n,i}}^{t_{n,i+1}}  \|\addErr(t)\|^2_{a_{\switch(t)}}\dt\\
        \le&\; 2\sum_{i=\tilde \nrSwitches }^{\nrSwitches-1}\tilde \omega_{\tilde \nrSwitches,i} \|\adResidual_{t_{n,i+1}}\|^{2}_{m_{\switch(t^{+}_{n,i})}}
    +\sum_{i=\tilde \nrSwitches }^{\nrSwitches-1}\tilde \omega_{\tilde \nrSwitches,i}\int_{t_{n,i}}^{t_{n,i+1}}\|\adResidual(t)\|_{a'_{\switch(t)}}^2 \dt,
    \end{aligned}
    \end{align}
    for $\tilde \omega_{\tilde \nrSwitches,i} \vcentcolon= \prod\limits_{k=\tilde \nrSwitches+1}^{i} 2c_{\switch(t^{+}_{n,k}),\switch(t^{-}_{n,k})} $ for $i=\tilde \nrSwitches,\ldots,N-1$.
%
In particular, for $\tilde \nrSwitches=0$ we define $\tilde \omega\coloneqq \tilde \omega_0 = (\tilde \omega_{0,i})_{i=\tilde \nrSwitches}^{N-1}$ and \eqref{eqn:adj:est:rel} reads as
    \begin{align} \label{eqn:errorBoundAdjointEnergy}
    \|\addErr(t_{n,0}^+)\|_{m_{\switch(t_{n,0})}}^2 +  &\|\addErr\|_{a_{\switch},\tilde\omega}^2
    \leq   \Delta^2_{\adState}(\optOutVarRed) \;\vcentcolon=\;2\sum\limits_{i=0}^{N-1}\tilde \omega_i \|\adResidual_{t_{n,i+1}}\|_{m_{\switch(t^-_{n,i+1})}}^2 + \|\adResidual\|_{a_\switch',\tilde \omega}^2.
    \end{align}
\end{lemma}
\begin{proof}
	We proceed similarly as in the proof of \Cref{thm:errorBoundStateEnergy}. The error equation for $\epsilon$ is given for all $\varphi\in V$ and almost all $t\in \timeInt_n$ as
	\begin{align*}
		\left\{\quad
		\begin{aligned}
			-m_{\switch(t)}(\varphi,\ddt \epsilon(t)) + a_{\switch(t)}(\varphi,\epsilon(t)) &=  \langle \adResidual(t),\varphi\rangle_{V',V}, & \\
			m_{\switch(T_n)} (\epsilon(T_n), \varphi) &= \langle \adResidual_{T_n},\varphi \rangle_{V',V}, &  \\
            m_{\switch(t^{-}_{n,i})} (\epsilon(t_{n,i}^-), \varphi) &= \langle \adResidual_{t_{n,i}},\varphi \rangle_{V',V} + m_{\switch(t_{n,i}^+)}(\addErr(t_{n,i}^+), \varphi), & \forall t_{n,i}\in\calT_n.
		\end{aligned}\right.
	\end{align*}
	Using the Cauchy-Schwarz inequality, we obtain for the terminal condition
	\begin{align}\nonumber
		 \|\epsilon(T_n)\|_{m_\switch(T_n)}^2\;=\;m_{\switch(T_n)}(\epsilon(T_n),\epsilon(T_n)) \leq \|\adResidual_{T_n}\|_{m_\switch(T_n)} \|\epsilon(T_n)\|_{m_\switch(T_n)}.
	\end{align}
    For the switching points $t_{n,i}\in\calT_n$, we get by testing with $\epsilon(t_{n,i}^-)$
    \begin{align*}
		 \|\epsilon(t_{n,i}^-) \|_{m_{\switch(t^{-}_{n,i})}}\;\leq\; \|\adResidual_{t_{n,i}}\|_{m_{\switch(t^{-}_{n,i})}}+\sqrt{c_{\switch(t^+_{n,i}),\switch(t^-_{n,i})}}\,\| \epsilon({t^+_{n,i}})\|_{m_{\switch(t^+_{n,i})}},
	\end{align*}
    which implies, by squaring and then using Young's inequality
    \begin{align}\label{eqn:adjointBoundSwitchEnergy}
		 \|\epsilon(t_{n,i}^-) \|^2_{m_{\switch(t^{-}_{n,i})}}\;\leq\; 2\,\|\adResidual_{t_{n,i}}\|^2_{m_{\switch(t^{-}_{n,i})}}+2{c_{\switch(t^+_{n,i}),\switch(t^{-}_{n,i})}}\| \epsilon({t^+_{n,i}})\|^2_{m_{\switch(t^{+}_{n,i})}}.
	\end{align}
    We now test the error system with $\epsilon(t) \in V$ for almost all $t\in (t_{n,i},t_{n,i+1})$ and $i=0,\ldots,N-1$, to obtain
	\begin{equation*}
		-\tfrac{1}{2}\ddt m_{\switch(t)}(\epsilon(t),\epsilon(t)) + a_{\switch(t)}(\epsilon(t),\epsilon(t))
		\;=\;\langle \adResidual(t), \epsilon(t)\rangle_{V',V}.
	\end{equation*}
	Integration from $t_{n,i}$ to $t_{n,i+1}$ leads to 
	\begin{align}\nonumber
		 \tfrac{1}{2}\|\epsilon(t_{n,i}^+)\|^2_{m_{\switch(t_{n,i}^+)}}+ \int_{t_{n,i}}^{t_{n,i+1}}a_{\switch(t)}(\epsilon(t),\epsilon(t))\dt
		\;=\; \int_{t_{n,i}}^{t_{n,i+1}}\langle \adResidual(t), \epsilon(t)\rangle_{V',V}\dt+\tfrac{1}{2}\|\epsilon(t_{n,i+1}^-)\|^2_{m_{\switch(t_{n,i+1}^-)}}
	\end{align}
    and by Cauchy-Schwarz in combination with Young's inequality, we derive
    \begin{align}
		\label{eqn:adjointBoundEnergy:tmpEnergy}
    \begin{aligned}
        \|\epsilon(t_{n,i}^+)\|^2_{m_{\switch(t_{n,i}^+)}}\;+&\; \int_{t_{n,i}}^{t_{n,i+1}}a_{\switch(t)}(\epsilon(t),\epsilon(t))\dt\\
		\leq&\; \int_{t_{n,i}}^{t_{n,i+1}} \|\adResidual(t)\|^2_{{a'_{\switch(t)}}}\dt+\|\epsilon(t_{n,i+1}^-)\|^2_{m_{\switch(t_{n,i+1}^-)}}.
        \end{aligned}
	\end{align}
    Next, we prove \eqref{eqn:adj:est:rel} by induction and start with the case $\tilde N = N-1$. We obtain from \eqref{eqn:adjointBoundEnergy:tmpEnergy} for $i=N-1$
    \begin{align*}
		 \|\epsilon(t_{n,N-1}^+)\|^2_{m_{\switch(t_{n,N-1}^+)}} \;+&\; \int_{t_{n,N-1}}^{t_{n,N}}a_{\switch(t)}(\epsilon(t),\epsilon(t))\dt
		\leq\; \int_{t_{n,N-1}}^{t_{n,N}}\|\adResidual(t) \|^2_{{a'_{\switch(t)}}}\dt+\|\epsilon(t_{n,N}^-)\|^2_{m_{\switch(t^{-}_{n,N})}}.
	\end{align*}
    By definition $\|\epsilon(t_{n,N}^-)\|^2_{m_{\switch(t^{-}_{n,N})}} =\|\epsilon(T_n)\|^2_{m_{\switch({T_n})}}$ and therefore
    \begin{align}\nonumber
		 \|\epsilon(t_{n,N-1}^+)\|^2_{m_{\switch(t_{n,N-1}^+)}}+ \int_{t_{n,N-1}}^{t_{n,N}}a_{\switch(t)}(\epsilon(t),\epsilon(t))\dt
		\leq \int_{t_{n,N-1}}^{t_{n,N}}\|\adResidual(t) \|^2_{{{a'_{{\switch(t)}}}}}\dt+2\|\adResidual_{T_n} \|^2_{m_{\switch({T_n})}},
	\end{align}
    which is \eqref{eqn:adj:est:rel} for $\tilde \nrSwitches= N-1$.
    Assume \eqref{eqn:adj:est:rel} to be valid for $\tilde \nrSwitches=j$ with $1\le j \le\nrSwitches-1$. We need to show that it holds for $\tilde \nrSwitches=j-1$. Choosing $i=j-1$ in \eqref{eqn:adjointBoundEnergy:tmpEnergy} gives
    \begin{align*}
    \begin{aligned}
        \|\epsilon(t_{n,j-1}^+)\|^2_{m_{\switch(t_{n,j-1}^+)}}+ \int_{t_{n,j-1}}^{t_{n,j}}a_{\switch(t)}(\epsilon(t),\epsilon(t))\dt
		\;\leq\; \int_{t_{n,j-1}}^{t_{n,j}} \|\adResidual(t) \|^2_{{a'_{\switch(t)}}}\dt+\|\epsilon(t_{n,j}^-)\|^2_{m_{\switch(t^-_{n,j})}}.
  \end{aligned}
    \end{align*}
    Combining this with \eqref{eqn:adjointBoundSwitchEnergy} for $i=j$ we get
    \begin{align}\nonumber
    \begin{aligned}
       & \|\epsilon(t_{n,j-1}^+)\|^2_{m_{\switch(t_{n,j-1}^+)}}\;+\; \int_{t_{n,j-1}}^{t_{n,j}}a_{\switch(t)}(\epsilon(t),\epsilon(t))\dt\\ \nonumber
        \leq&\; \int_{t_{n,j-1}}^{t_{n,j}} \|\adResidual(t) \|^2_{{{a'_{{\switch(t)}}}}}\dt+2\|\adResidual_{t_{n,j}}\|^2_{m_{\switch(t^{-}_{n,j})}} +2{c_{\switch(t^+_{n,j}),\switch(t^-_{n,j})}}   \|\epsilon(t_{n,j}^+)\|^2_{m_{\switch(t^+_{n,j})}}.
    \end{aligned}
    \end{align}
    Finally, using the induction hypothesis to bound $\| \epsilon({t^+_{n,j}})\|^2_{m_{\switch(t^+_{n,j})}}$ results in
    \begin{align}\label{eqn:ind:j-1}
    \begin{aligned}
        \|\epsilon(t_{n,j-1}^+)\|^2_{m_{\switch(t_{n,j-1}^+)}}\;+&\; \int_{t_{n,j-1}}^{t_{n,j}}a_{\switch(t)}(\epsilon(t),\epsilon(t))\dt\;
		\leq\;  \int_{t_{n,j-1}}^{t_{n,j}} \|\adResidual(t) \|^2_{{{a'_{{\switch(t)}}}}}\dt\\
  +&\;2\,\|\adResidual_{t_{n,j}}\|^2_{m_{\switch(t^{-}_{n,j})}} 
  +\;2{c_{\switch(t^+_{n,j}),\switch(t^-_{n,j})}}\bigg(
    2\sum\limits_{i=j}^{N-1}\tilde \omega_{j,i} \|\adResidual_{t_{n,i+1}}\|_{m_{\switch(t^-_{n,i+1})}}^2 \\
    +&\; \sum\limits_{i=j}^{N-1}\tilde \omega_{j,i}  \int_{t_{n,i}}^{t_{n,i+1}}\|\adResidual(t)\|_{{{a'_{{\switch(t)}}}}}^2 \dt
    - \sum\limits_{i=j}^{N-1}\tilde  \omega_{j,i} \int_{t_{n,i}}^{t_{n,i+1}}\|\addErr(t)\|_{a_{\switch(t)}}^2 \dt\bigg).
    \end{aligned}
    \end{align}
    Reordering \eqref{eqn:ind:j-1} gives \eqref{eqn:adj:est:rel} for $\tilde \nrSwitches=j-1$. 
\end{proof}
Finally, we can formulate an upper bound for $\BoundControlB$, which can be efficiently evaluated using an offline-online decomposition of the state and adjoint residual norms (cf. \cite{patera2007reduced}).
\begin{corollary}
    \label{cor:errBoundOptimalControl} 
	Let \Cref{ass:switchedSystem,ass:propertiesG,ass:switchedSystem:A:coerciv} be satisfied. Suppose that $\optInpVarN \in \U_n$ and $\optInpVarRedN \in \U_n$ are the unique solutions of the \FOM-\MPC problem \eqref{eqn:MPCsubproblem} and the \ROM-\MPC problem \eqref{eqn:MPCsubproblem:red}, respectively. Then
    \begin{align}
    	\label{eq: a post estimate RB realization}
     \begin{aligned}
    \BoundControlB^2(\optInpVarRedN,\initBound) \leq\ \BoundControlBtil^2(\optInpVarRedN,\initBound)\vcentcolon=C_1\initBound^2+ C_2 \Delta^2_{\adState}(\optOutVarRed) 
          +C_3 \Delta^2_{\state}(\optInpVarRedN)
          \end{aligned}
    \end{align}
    with $C_2,C_3 >0$ given as
    %
    \begin{align*}
     C_2 &\vcentcolon= \max_{i\in\{0,\ldots,N-1\}}\frac{\gamma^2_{\bOper_{\switch},{a}_{\switch}}}{\tilde \omega_i\lambda^2} , & 
     C_3 &\vcentcolon= \max\bigg\{\max_{i\in\{0,\ldots,N-1\}}\frac{\gamma^2_{\cOper_{\switch},{a}_{\switch}}}{ \omega_i\lambda }, \frac{\mu \gamma^2_{\cOper_{\sigma(T_n)},m_{\switch(T_n)}}}{\lambda}\bigg\}, 
    \end{align*}
    and $\omega$, $\tilde\omega$ defined as in  \Cref{thm:errorBoundStateEnergy,thm:errorBoundAdjointEnergy} and $C_1$ as in \Cref{thm:apostErrorControl}.
\end{corollary}
\begin{proof}
Observe that, by \Cref{def:energy:norm}, we have the following equivalence of norms
\begin{equation}\label{eqn:norm:equiv}
    \min_{i\in\{0,\ldots,N-1\}}\omega_i\|\varphi\|^2_{\mathfrak{a}_{\switch}}\;\le\;\|\varphi\|_{\mathfrak{a}_\switch,\omega}^2\;\le\;\max_{i\in\{0,\ldots,N-1\}}\omega_i\|\varphi\|^2_{\mathfrak{a}_{\switch}}.
\end{equation}
Estimating from \eqref{eq: a post estimate general all adjoint} and using \eqref{eqn:properties:c:en:norm2}, \eqref{eqn:properties:c:enrgy:norm}
together with \eqref{eqn:norm:equiv} gives 
    \begin{align*}
    \begin{aligned}
        \tfrac{1}{\lambda^2} \|\bOper_{\switch}' (\optAdStateRed-\intermediateAdStateB)\|_{\U_n}^2 \; +&\;\tfrac{1}{\lambda} \|\optOutVarRed - \intermediateOutVar\|_{L^2}^2 \; + \;\tfrac{\mu}{\lambda}\|\optOutVarRed(T_n)-\intermediateOutVar(T_n)\|_{\R^{\outVarDim}}^2\\
        \leq &  \max_{i\in\{0,\ldots,N-1\}}\tfrac{\gamma^2_{\bOper_{\switch},{a}_{\switch}}}{\tilde \omega_i\lambda^2} \|\optAdStateRed-\intermediateAdStateB\|_{a_\switch,\tilde \omega}^2  \\
         \;+&\; \max_{i\in\{0,\ldots,N-1\}}\tfrac{\gamma^2_{\cOper_{\switch},{a}_{\switch}}}{\lambda \omega_i}
        \|\optStateRed - \intermediateState\|_{a_\switch, {\omega}}^2  \;+\; \tfrac{\mu \gamma^2_{\cOper_{\sigma(T_n)},m_{\switch(T_n)}}}{\lambda}\|\optStateRed(T_n)-\intermediateState(T_n)\|_{m_{\switch(T_n)}}^2\\
        = &\; C_2\|\addErr\|_{a_\switch,\tilde \omega}^2 \; +\; C_3\big(
        \|\err\|_{a_\switch, {\omega}}^2  + \|\err(T_n)\|_{m_{\switch(T_n)}}^2\big)\\
        \leq &\; C_2 \Delta^2_{\adState}(\optOutVarRed) \; +\; C_3\Delta^2_{\state}(\optInpVarRedN),
        \end{aligned}
    \end{align*}
    where for the last inequality we used \Cref{thm:errorBoundStateEnergy,thm:errorBoundAdjointEnergy}.
\end{proof}
Now, we can state how the initial condition and the error in the optimal control propagate to the errors in the $\optInpVarRedN$ controlled \FOM or \ROM state.
\begin{lemma}[A-posteriori estimate for the optimal state]\label{lem:apost_optimal_state}
Let \Cref{ass:switchedSystem,ass:propertiesG,ass:switchedSystem:A:coerciv} be satisfied. Let the optimal \FOM state $\optState=\solper(\state_n,\optInpVarN)\in \X_n$,
 the intermediate \FOM state $\intermediateState =\solper(\intermediateInitState,\optInpVarRedN)\in\X_n$ and
the optimal \ROM state $\optStateRed=\reduce\solper(\stateRed_n,\optInpVarRedN)\in \Xred_n$ be given.
Further, let $\Delta_u\in \{\BoundControlA,\BoundControlB,  \BoundControlBtil\}$ defined in \Cref{thm:apostErrorControl} or \Cref{cor:errBoundOptimalControl}. Let us define the errors
\begin{equation*}
    \intermediateErr=\optState-\intermediateState,\quad\redErr =\optState-\optStateRed.
\end{equation*}
Then, the following error bounds hold for $t\in \closure\timeInt_n$ 
\begin{align}
    \label{eqn:optStateBound:Venergy:FOMROM}
    \|\intermediateErr(t)\|_{m_{\switch(t)}}^2 
      &\leq  \omega_{\nrSwitches_t,0} \Delta_{t_n}^2
     +\frac{C_4(t)}{ 2}\Delta^2_u(\optInpVarRedN, \initBound),\\
    \|\redErr(t)\|_{m_{\switch(t)}}^2 
    &\leq \omega_{\nrSwitches_t,0} \big(\Delta_{t_n}+\|\intermediateInitState-\projVred\intermediateInitState\|_{m_{\switch{(t_n)}}} \big)^2 +C_4(t)\Delta^2_u(\optInpVarRedN, \initBound) + \|\residual\|_{a_{\switch}',\omega_{\nrSwitches_t}, t}^2,\label{eqn:optStateBound:Venergy:ROMROM}
\end{align}
with 
\begin{equation}
    \nonumber
    C_4(t) = \max_{i\in\{0,\ldots,N_t-1\}}{\omega_{\nrSwitches_t,i}\gamma^2_{\bOper_{\switch},{a}_{\switch}}}{ },
\end{equation}
where $N_t\in\{0,\ldots,N-1\}$ is the smallest integer with $t\in [t_{n,N_t},t_{n,{N_{t}+1}})$ and $\omega_{N_t}$ is defined in \Cref{thm:errorBoundStateEnergy}. 
\end{lemma}
\begin{proof}
    We restrict ourselves to showing the result for $\redErr$, which includes the case for $\intermediateErr$. We have the error equation for all $\varphi\in V$
    \begin{align}\label{eqn:optimalstateBound:erroreq}
		\left\{\quad
    	\begin{aligned}
        	m_{\switch(t)}(\ddt \redErr(t),\varphi) + a_{\switch(t)}( \redErr(t),\varphi) &= \langle R(t), \varphi\rangle_{V',V} + b_{\switch(t)}(\optInpVarN(t)-\optInpVarRedN(t), \varphi),&& \text{a.e.~in $\timeInt_n$},\\ 
       		\redErr(t_n) &= \state_n - \stateRed_n.
    \end{aligned}\right.
	\end{align}
    Similar to the proof of \Cref{thm:errorBoundStateEnergy}, we obtain for all $i\in 0,\ldots,N-1$ from the error equation
    and using Young's weighted inequality
    \begin{align}\nonumber
    \begin{aligned}
       \|e(t_{n,i+1})\|_{m_{\switch(t^{-}_{n,i+1})}}^2  
       \leq \int_{t_{n,i}}^{t_{n,i+1}} \|R(t)\|^2_{{a_{\switch(t)}}}+\|\bOper_{\switch(t)}(\optInpVarN-\optInpVarRedN)(t)\|^2_{{a_{\switch(t)}}} \dt + \|e(t_{n,i})\|_{m_{\switch(t^{+}_{n,i})}}^2\dt.
    \end{aligned}
    \end{align}
    Since we have the estimate
    \begin{equation}\nonumber 
        \sum\limits_{i=0}^{ \nrSwitches_t-1}\omega_{ \nrSwitches_t,i}\int_{t_{n,i}}^{t_{n,i+1}}\|\bOper_{\switch(t)}(\optInpVarN-\optInpVarRedN)(t)\|^2_{{a_{\switch(t)}}} \dt\leq \max_{i\in\{0,\ldots,N_t-1\}}\omega_{\nrSwitches_t,i}\gamma_{\bOper_{\switch},{a}_{\switch}}^2 \Delta^2_u(\optInpVarRedN, \initBound)
    \end{equation}
    and due to 
    \begin{equation}\nonumber 
       \|\redErr(t_n)\|_{m_\switch(t_n)}= \|\state_n-\stateRed_n\|_{m_\switch(t_n)}\leq \initBound+\|\intermediateInitState-\stateRed_n\|_{m_\switch(t_n)}
    \end{equation}
    for the initial term (using \eqref{eqn:stateInitEst} and the definition of $\stateRed_n$ below \eqref{eqn:weakForm:reduced}), the claim \eqref{eqn:optStateBound:Venergy:ROMROM} follows using the same induction argument as in the proof of \Cref{thm:errorBoundStateEnergy}. The proof for the claim \eqref{eqn:optStateBound:Venergy:FOMROM} works analogously because in \eqref{eqn:optimalstateBound:erroreq} the residual term vanishes due to $\intermediateErr\in \X_{n,\switch}$, which improves the constant in \eqref{eqn:optStateBound:Venergy:FOMROM} by the factor $\tfrac{1}{2}$. Moreover, we obtain $ \|\intermediateErr(t_n)\|_{m_\switch(t_n)}\leq \initBound$ by \eqref{eqn:stateInitEst}.
\end{proof}
\subsection{Certification of the reduced \MPC control}
\label{subsec:applicationToMPC}
To utilize our certificates within the \MPC algorithm, we proceed as follows. To derive the \ROM-\MPC feedback law in the $n$-th \MPC iteration, we solve the \ROM-\MPC subproblem \eqref{eqn:MPCsubproblem} to find the optimal control $\optInpVarRedN$. The certification framework outlined in the previous subsection can then be applied to assess the error in the reduced \MPC schemes through the following two ways detailed here:
\begin{enumerate}
    \item 
    At each \MPC iteration $n$, we can compute an a-posteriori bound on the error between the \FOM-\MPC trajectory $\stateMPC$ in $\timeInt_n$ and its reduced counterpart. This approach ensures that the \ROM-\MPC trajectories remain within a bounded region around the \FOM-\MPC trajectory, with a radius of $\initBoundn{n+1}$, where the quality of the \ROM determines the size of this region. This concept is detailed in \Cref{cor:aposterioriBoundMPC}.
    \item The computation of $\initBound$ inherently includes the calculation of the control error estimator $\Delta(\optInpVarRedN,0)$, which measures how accurately the \ROM approximates the \FOM feedback law in the current \ROM state. Controlling this error asymptotically results in an exact approximation; see \Cref{cor:MPCasymptoticConvergence}.
\end{enumerate}
If, during the $n$-th \MPC iteration, the \ROM-\MPC trajectory deviates too far from the \FOM-\MPC trajectory or the current \FOM-\MPC feedback law is poorly approximated, in the sense that the error estimators $\initBoundn{n+1}$ or $\Delta(\optInpVarRedN,0)$ exceed user-defined thresholds $\varepsilon_n,\epsilon_n>0$, then the reduced feedback law $\optInpVarRedN$ is rejected. In such cases, a \FOM-\MPC step is performed instead, the \ROM is refined, and the procedure is restarted. This approach provides certification with respect to the last \FOM-\MPC step. The method for updating the \ROM depends on the chosen \MOR technique and is discussed in the following subsection for a \POD-\ROM (see \Cref{rem: extending the ROM POD}). 
On the other hand, if the error estimator confirms that the \ROM approximation is sufficiently accurate, the reduced optimal control $\optInpVarRedN$ is accepted. Subsequently, the \ROM control is applied either to the \FOM system \eqref{eqn:weakForm} or the \ROM system \eqref{eqn:weakForm:reduced}. These resulting approaches are referred to as \FOM-\ROM-\MPC or \ROM-\ROM-\MPC, respectively, and the details are summarized in \Cref{alg:ROM-MPC,alg:ROMROM-MPC}.
\begin{enumerate}
    \item For the \FOM-\ROM-\MPC approach (\Cref{alg:ROM-MPC}), the controller is designed using the \ROM but applied to the \FOM, resulting in the following \FOM-\ROM-\MPC iteration:
    \begin{equation}\label{eq:FOMROMiter}
        \intermediateStateMPC(0)=\theta_0 ,\quad \intermediateStateMPC(t)=\big[\solper(\intermediateStateMPC(t_{n}),\optInpVarRedN)\big](t)\quad \text{for }t\in (t_{n},t_{n+1}].
    \end{equation}
    The certification of this algorithm can be entirely based on \ROM methods that provide output bounds for the primal and dual input-output relationships (see \Cref{thm:apostErrorControl} and \eqref{eq:D_init_FOMROM}). This approach can be particularly advantageous if the Kolmogorov $n$-widths or the Hankel singular values corresponding to the input-to-state map (see \cite{Kolmogoroff1936,Unger2019}) exhibit slow decay, while those of the input-to-output map decrease rapidly.
    \item Conversely, the \ROM-\ROM-\MPC approach (\Cref{alg:ROMROM-MPC}) designs and applies the controller directly at the \ROM level, yielding the \ROM-\ROM-\MPC iteration
    \begin{equation}\label{eq:ROMROMiter}
        \stateRedMPC(0)=\projVred\theta_0 ,\quad \stateRedMPC(t)=\big[\reduce\solper(\stateRedMPC(t_{n}),\optInpVarRedN)\big](t)\quad \text{for }t\in (t_{n},t_{n+1}].
    \end{equation}
    Unlike \eqref{eq:FOMROMiter}, this algorithm is entirely online-efficient, as it remains independent of the \FOM dimension after the \ROM is constructed \ReA{(see also the complexity analysis of both algorithms in \Cref{app:sec_complexity})}. However, it necessitates a high-quality \ROM approximation of the state variable. This requirement is evident in the use of the reduced solution operator in \eqref{eq:ROMROMiter}, which introduces additional state residual terms in the error bounds (see \eqref{eq:D_init_ROMROM}).
\end{enumerate}
\begin{algorithm}[t]
	\caption{\FOM-\ROM-\MPC}\label{alg:ROM-MPC}
	\begin{algorithmic}[1]
		\Require Prediction horizon $T>0$, sampling time $0<\delta<T$, error tolerances $\varepsilon_n,\errortol_n>0$, error estimator $\Delta_u\in \{\BoundControlA,\BoundControlB, \BoundControlAtil\}$, initial value $\state_0\in H$.
        \State Set $\intermediateStateMPC(0) = \state_0$ and $\initBoundn{0}= 0$. 
		\For{$n=0,1,2,...$}
			\State Set $t_n = n\delta$ and $\intermediateState_n = \intermediateStateMPC(t_n)$ and $\stateRed_n = \projVred \intermediateState_n$.\label{line:alg:FOMROMMPC_initvalue}
            \State Solve \ROM-\MPC subproblem \eqref{eqn:MPCsubproblem:red} for the \ROM-\MPC feedback law $\optInpVarRedN$. 
            \State Compute a-posteriori error estimate $\Delta_u(\optInpVarRedN,0)$ and $\initBoundn{n+1}$ in \eqref{eq:D_init_FOMROM}.
            \If{$\Delta_u(\optInpVarRedN,0)>\varepsilon_n$ or $\initBoundn{n+1}>\epsilon_n$}\label{alg:line:ErrorEstCrit}
                 \State Solve the \FOM-\MPC subproblem \eqref{eqn:MPCsubproblemControlreduced} for the \FOM-\MPC feedback law $\optInpVarN$. 
		          \State Apply the control $\optInpVarN$ to the \FOM equation~\eqref{eqn:weakForm} on $\timeInt_n^\delta $ to obtain
                 \begin{align*} \intermediateInpVarMPC(t)\vcentcolon=\optInpVarN{(t)}, \quad\intermediateStateMPC (t)\vcentcolon=\calS(\intermediateState_n,\optInpVarRedN)(t),\quad\intermediateOutVarMPC (t)\vcentcolon=\cOper_{\switch{(t)}} \intermediateStateMPC (t)\quad\text{for }t\in \timeInt_n^\delta.
                \end{align*}
                \State Update the \ROM (c.f.~\Cref{rem: extending the ROM POD}) and reset $\initBoundn{n+1}=0$.
            \Else
                \State \label{algo:line:applyROM}Apply the control $\optInpVarRedN$ to the \FOM equation \eqref{eqn:weakForm} on $\timeInt_n^\delta$ to obtain 
                \begin{align*}
                    \intermediateInpVarMPC(t)& =\optInpVarRedN (t),\quad \intermediateStateMPC(t)=\calS(\intermediateState_n,\optInpVarRedN)(t), \quad\intermediateOutVarMPC(t)=\cOper_{\switch(t)}\intermediateStateMPC(t)\quad\text{for }t\in \timeInt_n^\delta.
                \end{align*}
            \EndIf
        \EndFor
	\end{algorithmic}
    \begin{flushleft}
    \textbf{Output:} $\intermediateOutVarMPC, \intermediateStateMPC$, $\intermediateInpVarMPC$.
    \end{flushleft}
\end{algorithm}
\begin{algorithm}[t]
	\caption{\ROM-\ROM-\MPC}\label{alg:ROMROM-MPC}
	\begin{algorithmic}[1]
		\Require Prediction horizon $T>0$, sampling time $0<\delta<T$, error tolerances $\varepsilon_n,\epsilon_n>0$, error estimator $\Delta_u\in \{\BoundControlA,\BoundControlB, \BoundControlAtil\}$, initial value $\state_0\in H$.
        \State Set $\stateRedMPC(0) = \projVred\state_0$ and $\initBoundn{0}= \|\state_0- \projVred\state_0\|_{m_{\switch(0)}}$. 
		\For{$n=0,1,2,...$}
			\State Set $t_n = n\delta$ and $\stateRed_n = \intermediateState_n = \stateRedMPC(t_n)$. \label{line:alg:ROMROMMPC_initvalue}
            \State Solve \ROM-\MPC subproblem \eqref{eqn:MPCsubproblem:red} for the \ROM-\MPC feedback control $\optInpVarRedN$. 
             \State Compute a-posteriori error estimate $\Delta_u(\optInpVarRedN,0)$ and $\initBoundn{n+1}$ in \eqref{eq:D_init_ROMROM}.
            \If{$\Delta_u(\optInpVarRedN,0)>\varepsilon_n$ or $\initBoundn{n+1}>\epsilon_n$}\label{line:alg:ROMROMupdatemechasim}
                \State Solve the \FOM-\MPC subproblem \eqref{eqn:MPCsubproblemControlreduced} for the \FOM-\MPC feedback law $\optInpVarN$. 
		          \State Apply the control $\optInpVarN$ to the \FOM equation~\eqref{eqn:weakForm} on $\timeInt_n^\delta $ to obtain
                 \begin{align*} \inpVarRedMPC(t)\vcentcolon=\optInpVarN{(t)}, \quad\stateRedMPC (t)\vcentcolon=\solper(\stateRed_n,\optInpVarRedN)(t),\quad\outVarRedMPC (t)\vcentcolon=\cOper_{\switch{(t)}} \stateRedMPC (t)\quad\text{for }t\in \timeInt_n^\delta.
                \end{align*}
                \State\label{line:alg:ROMROMupdate} Update the \ROM (c.f.~\Cref{rem: extending the ROM POD}) and reset $\initBoundn{n+1}=0$.
            \Else
                \State Apply the control $\optInpVarRedN$ to the \ROM equation \eqref{eqn:weakForm:reduced} on $\timeInt_n^\delta$ to obtain 
                \begin{align*}
                    \inpVarRedMPC(t)& =\optInpVarRedN (t),\quad \stateRedMPC(t)=\reduce\solper(\stateRed_n,\optInpVarRedN)(t),\quad\outVarRedMPC(t)=\cOper_{\switch(t)}\stateRedMPC(t)\quad\text{for }t\in \timeInt_n^\delta.
            \end{align*}
            \EndIf
        \EndFor
	\end{algorithmic}
    \begin{flushleft}
    \textbf{Output:} $\outVarRedMPC, \stateRedMPC$, $\inpVarRedMPC$.
    \end{flushleft}
\end{algorithm}
\begin{corollary}[A-posteriori error estimates for the \ROM-\MPC state and feedback law]\label{cor:aposterioriBoundMPC}
Let \Cref{ass:switchedSystem,ass:propertiesG,ass:switchedSystem:A:coerciv} be satisfied.
Let $(\inpVarMPC,\stateMPC)$, $(\intermediateInpVarMPC, \intermediateStateMPC)$, and $(\inpVarRedMPC, \stateRedMPC)$ denote the output of \Cref{alg:FOM-MPC,alg:ROM-MPC,alg:ROMROM-MPC}, respectively. Further, let $\Delta_u\in \{\BoundControlA,\BoundControlB,  \BoundControlBtil\}$ defined in \Cref{thm:apostErrorControl} or \Cref{cor:errBoundOptimalControl} and assume that the restart mechanism in \Cref{line:alg:ROMROMupdatemechasim} of {\Cref{alg:ROM-MPC,alg:ROMROM-MPC}} is not triggered.
\begin{enumerate}
    \item We define the \FOM-\ROM-\MPC state error as 
        \begin{equation*}
            \intermediateErrMPC \coloneqq \stateMPC-\intermediateStateMPC.
        \end{equation*}
        At iteration $n=0$, it holds for the \FOM-\ROM-\MPC
        \begin{align*}
             \Delta_{t_0}& =\|\intermediateErrMPC(0)\|_{m_{\switch(0)}} =0, \qquad  \|\intermediateState_0-\stateRed_0\|_{m_{\switch(0)}}=\|\intermediateState_0-\projVred\intermediateState_0\|_{m_{\switch(0)}}.
        \end{align*}
        At iteration $n\geq 1$, we have the following recursive error bound for the \FOM-\ROM-\MPC state
        \begin{align}
        \label{eqn:optStateBound:MPC:FOMROM_terminal}
            \|\intermediateErrMPC(t_n)\|_{m_{\switch(t_{n})}}^2 
            & \leq         \initBoundn{n}^2(\Delta^2_u(\inpVarRedMPC, \initBoundn{n-1}),\initBoundn{n-1})
        \end{align}
        with
        \begin{align}\label{eq:D_init_FOMROM}
        \begin{aligned}
 \initBoundn{n}^2(\Delta_u(\inpVarRedMPC, \initBoundn{n-1}),\initBoundn{n-1}) \coloneqq &\omega_{\nrSwitches_{t_{n}},0} \initBoundn{n-1}^2+\frac{C_4(t_{n})}{ 2}\Delta^2_u(\optInpVarRedN, \initBoundn{n-1}) , 
     \end{aligned}
        \end{align}
        where the corresponding constants are defined in \Cref{lem:apost_optimal_state}.
    \item We define the \ROM-\ROM-\MPC state error as 
        \begin{equation*}
            \redErrMPC \coloneqq \stateMPC-\stateRedMPC.
        \end{equation*}
        At iteration $n=0$, it holds for the \ROM-\ROM-\MPC
        \begin{align*}
             \Delta_{t_0} &= \|\redErrMPC(0)\|_{m_{\switch(0)}}  = \|\state_0-\projVred\state_0\|_{m_{\switch(0)}}, \qquad  \|\intermediateState_0-\stateRed_0\|_{m_{\switch(0)}}=0.
        \end{align*}
        At iteration $n\geq 1$, we have the following recursive error bound for the \ROM-\ROM-\MPC state
        \begin{align}\label{eqn:optStateBound:MPC:ROMROM_terminal}\|\redErrMPC(t_{n})\|_{m_{\switch(t_{n})}}^2
            \leq  \initBoundn{n}^2(\Delta^2_u(\inpVarRedMPC, \initBoundn{n-1}),\initBoundn{n-1})
        \end{align}
        with
        \begin{align}\label{eq:D_init_ROMROM}
        \begin{aligned}
        \initBoundn{n}^2(\Delta_u(\inpVarRedMPC, \initBoundn{n-1}),\initBoundn{n-1}) &\vcentcolon= 
        \omega_{\nrSwitches_{t_{n}},0} \initBoundn{n-1}^2
+\|\residual\|_{a_{\switch}',\omega_{\nrSwitches_{t_{n}}}, t_{n}}^2 
+ C_4(t_n)\Delta^2_u(\optInpVarRedN, \initBoundn{n-1}),
        \end{aligned}
        \end{align}
        where the corresponding constants are defined in \Cref{lem:apost_optimal_state}.
\end{enumerate}
\end{corollary}
\begin{proof}
    The claim for $n=0$ directly follows from the definition of the \Cref{alg:ROM-MPC,alg:ROMROM-MPC}. For $n\geq 0$, the claim is an application of \Cref{lem:apost_optimal_state}: \eqref{eqn:optStateBound:MPC:FOMROM_terminal} comes from \eqref{eqn:optStateBound:Venergy:FOMROM} for $t=t_{n+1}$. Also, \eqref{eqn:optStateBound:MPC:ROMROM_terminal} comes from \eqref{eqn:optStateBound:Venergy:ROMROM} for $t=t_{n+1}$ and since the term $\|\intermediateState_n-\projVred\intermediateState_n\|_{m_\switch(t_n)}=\|\intermediateState_n-\stateRed_n\|_{m_\switch(t_n)} = 0$ vanishes, due to \cref{line:alg:ROMROMMPC_initvalue} of \Cref{alg:ROMROM-MPC}.
\end{proof}
We now demonstrate that the output, state, and control of the \FOM-\ROM-\MPC exhibit the same asymptotic behavior as the output, state, and control of the \FOM-\MPC, provided the error sequence~$(\varepsilon_n)_{n\in \N}$ decays to zero. This holds under the condition that the \ROM converges to the \FOM as the reduced dimension approaches infinity, a condition that is satisfied by the \MOR method discussed in \Cref{subsec:POD}.
\begin{lemma}[Asymptotic convergence for the \FOM-\ROM-\MPC]\label{cor:MPCasymptoticConvergence}
In the following, we denote by $\inpVarMPC(\theta_\circ),\stateMPC(\theta_\circ),\intermediateInpVarMPC(\theta_\circ),\intermediateStateMPC(\theta_\circ)$ the output of \Cref{alg:FOM-MPC,alg:ROM-MPC}, respectively, starting from $\theta_\circ\in H$. Consider the error sequence $(\varepsilon_n)_{n\in \N}$ from \Cref{alg:ROM-MPC} and assume that $\varepsilon_n\searrow 0$ as $n\to \infty$.
    \begin{enumerate}
        \item \label{cor:MPCasymptoticConvergence:A} It holds
        \begin{align}\nonumber
            \lim_{n\to\infty} \|\inpVarMPC(\intermediateState_n)-\inpVarRedMPC(\intermediateState_n)\|_{\U_n} = 0 \qquad\text{and}\qquad
            \lim_{n\to\infty} \|\stateMPC(\intermediateState_n)-\intermediateStateMPC(\intermediateState_n)\|_{L^2(\timeInt_n,V)} = 0.
        \end{align}
        \item 
        Suppose that for given target $\outVarDes$ the \FOM-\MPC feedback law ensures $\|\outVarMPC(\theta_\circ)(t)-\outVarDes(t)\|_{\R^{\outVarDim}}\to 0 $ for $t\to \infty$ and for every $\theta_\circ\in H$. Then it holds
        \end{enumerate}
        \begin{align}\nonumber
        \lim_{n\to\infty} \|\intermediateOutVarMPC
        (\intermediateState_n)-\outVarDes\|_{L^2(\timeInt_n^\delta,\R^{\outVarDim})} = 0.
        \end{align}
\end{lemma}
\begin{proof}\, 
    \begin{enumerate}
        \item The first claim follows from the definition of \Cref{alg:ROM-MPC} and the control error estimator
        \begin{equation}\nonumber
            \|\inpVarMPC(\intermediateState_n)-\inpVarRedMPC(\intermediateState_n)\|_{\U_n} \leq \Delta_u(\inpVarRedMPC(\intermediateState_n),0)\leq \varepsilon_n \to 0 \ (n\to \infty).
        \end{equation}
        The second claim follows from a slight modification of \eqref{eqn:optStateBound:Venergy:FOMROM} in \Cref{lem:apost_optimal_state} and the fact that $\initBound$ vanishes is we start both trajectories from $\intermediateState_n$.
        \item We have
        \begin{align}\nonumber
            \|\intermediateOutVarMPC
        (\intermediateState_n)-\outVarDes\|_{L^2(\timeInt_n^\delta,\R^{\outVarDim})} \leq \|\outVarMPC
        (\intermediateState_n)-\outVarDes\|_{L^2(\timeInt_n^\delta,\R^{\outVarDim})}+\|\intermediateOutVarMPC
        (\intermediateState_n)-\outVarMPC
(\intermediateState_n)\|_{L^2(\timeInt_n^\delta,\R^{\outVarDim})}.
        \end{align}
        The first term goes to zero due to the global stabilizability assumption and the second due to \ref{cor:MPCasymptoticConvergence:A}.
    \end{enumerate}
\end{proof}
\begin{remark}
    If we additionally assume that the state residual tends to zero as $n\to\infty$, then \Cref{cor:MPCasymptoticConvergence} also holds for the output of the \ROM-\ROM-\MPC algorithm, i.e., \Cref{alg:ROMROM-MPC}.
\end{remark}
%
\section{Numerical experiments}
\label{sec:numerics}
%
In this section, we compare the performance of the \ROM-\MPC algorithms with the \FOM-\MPC when applied to the example detailed in \Cref{subsection:example,subsec:exampleContinued}. To obtain the finite-dimensional subspace $\Vred\subseteq V$ we rely on \POD, which we review in \Cref{subsec:POD}, following \cite{GubV17,BBMV24}. The computational setup for our example is given in \Cref{subsec:computationalSetup}, and our numerical experiments illustrating the error bounds for the open-loop problem and the \MPC algorithms are presented in \Cref{subsec:numerics:openloop}. Finally, we compare all the \MPC algorithms in \Cref{subsec:numerics:MPC}.\\
%
\vspace{0.2cm}
\noindent\fbox{%
    \parbox{0.98\textwidth}{%
        The code and data used to generate the subsequent results are accessible via
		\begin{center}
			\url{https://doi.org/10.5281/zenodo.14443259}
		\end{center}
		under MIT Common License.
    }%
}\\[.2em]

\subsection{Reduced model via proper orthogonal decomposition}
\label{subsec:POD}
The main idea of \POD is to extract the most relevant information of the available data in a weighted $L^2(\timeInt_n,V)$ sense. Let snapshot data $\snapshots_j\in L^2(\timeInt_n,V)$ for $j=1,\ldots,M$ be available (see \Cref{{rem: extending the ROM POD}} for possible choices).
\POD chooses basis functions $\PODbasis_i\in V$ for $i=1,\ldots,\stateDimRed$ solving the optimization problem
\begin{align}
    \label{eqn:PODminimization}
     \min_{\{\basis_i\}\subseteq V} \sum_{j=1}^{\nrSnapshots} \int_{\timeInt_n}\bigg\|\snapshots_j(t)-\sum_{i=1}^{\stateDimRed} \langle  \snapshots_j(t), \basis_i\rangle_V \,\basis_i \bigg\|^2_V \dt
    \quad \text{s.t.}\quad \langle \basis_i, \basis_j\rangle_V = \delta_{ij} \text{ for }i,j = 1,\ldots, \stateDimRed,
\end{align}
where $\delta_{ij}$ denotes the Kronecker delta.
\begin{theorem}[{\!\!\cite[Thm.~2.7 and Lem.~2.2]{GubV17}}]
    The optimization problem~\eqref{eqn:PODminimization} is solvable and a solution is given by the $\stateDimRed$ eigenvectors $\{\PODbasis_i\}_{i=1}^r\subset V$ corresponding to the largest nonzero eigenvalues $\PODeigVal_1\ge\PODeigVal_2\ge\ldots\ge\PODeigVal_r>0$ of the linear, compact, nonnegative, and self-adjoint operator
    \begin{equation*}
        \calR\colon V\to V, \qquad \basis \mapsto \sum_{j=1}^{M}  \int_{\timeInt_n}\langle \snapshots_j(t),\basis\rangle_{V} \snapshots_j(t) \dt.
    \end{equation*}
\end{theorem}
The question that remains to be answered is the choice of $\stateDimRed$. It is common, see, e.g., \cite{GubV17}, to use the information content of the basis about the whole snapshot set. In more detail, we are interested in the quantity
\begin{align}
    \label{eqn:POD:basisEnergy}
    \calE(\stateDimRed) \vcentcolon= \frac{\sum_{i=1}^\stateDimRed \PODeigVal_i}{\sum_{i=1}^\infty  \PODeigVal_i} = \frac{\sum_{i=1}^\stateDimRed \PODeigVal_i}{\sum_{j=1}^{\nrSnapshots} 
    \|\snapshots_j\|_{L^2(\timeInt_n,V)}^2},
\end{align}
where the identity is due to the Hilbert-Schmidt theorem that ensures the set of eigenfunctions forms an orthonormal basis of $V$.
Relation~\eqref{eqn:POD:basisEnergy} can be used to determine a basis size~$\stateDimRed$ containing~$e\%$ of the snapshot information by choosing $\stateDimRed$ with $\calE(\stateDimRed)\approx {e}/{100}$. The associated \POD subspace is then given as $\Vred = \spann\{\PODbasis_1,\ldots,\PODbasis_{\stateDimRed}\}$.
\begin{remark}\label{rem: extending the ROM POD}
\,
\begin{enumerate}
    \item We follow \cite{locke2021new} for the numerical solution of~\eqref{eqn:PODminimization}.
    \item The \ROM model update process in \Cref{alg:ROM-MPC,alg:ROMROM-MPC} for \POD works as follows. Starting with an empty model (assuming that the update mechanism is triggered), the \POD model update involves solving the current \FOM-\MPC subproblem \eqref{eqn:MPCsubproblem}, appending the optimal state and adjoint state to the set of snapshots, performing \POD by solving \eqref{eqn:PODminimization}, and selecting the reduced basis according to \eqref{eqn:POD:basisEnergy}. A residual-based error estimator $\BoundControlBtil$ can then be constructed in an offline-online manner by computing the $V_{a_\sigma}$-Rietz representatives of the affine components of the state and adjoint residuals for $\stateDimRed$ basis elements and each switching mode $\nrModes$ (see \cite[Sec.~4.2.5]{Hesthaven2016}). Recomputing \eqref{eqn:PODminimization} at every step may not always be necessary, as incremental \POD methods (see \cite{FAREED2019223}) offer efficient alternatives.
    \item Our approach allows for incorporating other snapshot types, such as those associated with the control target $\inpVarDes$ or derivative information between switching points, ensuring that the \ROM adapts dynamically along the \MPC trajectory. To address outdated data in the snapshot set, a practical strategy is to retain only the most recent snapshots.
\end{enumerate}
\end{remark}
\subsection{Problem data, discretization, and algorithmic setup}
\label{subsec:computationalSetup}
We consider the guiding example from \Cref{subsection:example,subsec:exampleContinued} with the domains as illustrated in \Cref{fig: geometry_tworooms}. The precise values for the domains and the parameters are given in \Cref{tab:parameters} with $\Omega_4 \vcentcolon= \Omega\setminus (\Omega_1\cup \Omega_2 \cup \Omega_3)$.
\begin{table}[ht!]
    \centering
    \caption{Domains and parameters for the numerical setup of the guiding example}
    \label{tab:parameters}
    \begin{tabular}{rcccc}
        \toprule
        \textbf{domain} & $\Omega$ & $\Omega_1$ & $\Omega_2$ & $\Omega_3$\\
        \textbf{value} & $(0,10.3)\times(0,5)$ & $(0,5)\times(0,5)$ & $(5,5.3)\times(2.3,2.7)$ & $(5.3,10.3)\times(0,5)$\\\bottomrule
    \end{tabular}\\[.5em]
    \begin{tabular}{rccccccccc}
        \toprule
        \textbf{coefficient} & $\zeta_1$ & $\zeta_2$ & $\zeta_3$ & $\kappa_1$ & $\kappa_2$ & $\kappa_3$ & $\gamma_o$ & $\advection$ & $c$ \\
        \textbf{value} & 1.00 & 0.5 & 1.00 & 0.01 & 10.0 & 0.01 & 0.15 & $[0.01, 0]^\T$ &$0.01$\\\bottomrule
    \end{tabular}
\end{table}
%
The number of controls is chosen as $\inpVarDim = 10$. Further, for the box constraints $\Uad$ we choose the constant lower and upper bounds $u_a\equiv -20$ and $u_b \equiv 20$, the $L^1$-regularization is chosen to be $\lonereg= 10^{-3}$ and the  $L^2$-regularization as $\lambda= 10^{-2}$. 
For the optimization problem, we choose the target control $\inpVarDes\equiv 0$ and compute the target output $\outVarDes$ as the output of the switched system driven with constant input $\inpVar\equiv 1$ and initial value $\state_\circ\equiv 1$. 
%
For the simulation, we fix the final time $\finalTime = 10$ and the switching points are given by $\tau_i=\nicefrac{i}{2}$ for $i=1,\ldots,20$. For the discretization in time, we apply an implicit Euler scheme using $K= 501$ time points with step size $\tau = \nicefrac{\finalTime}{(K-1)}=0.02$. For the discretization in space, we choose piecewise linear finite elements, resulting in a discrete state dimension of $N=5304$. For the solution of the non-smooth \MPC subproblems, we use a proximal gradient method with Barzilai-Borwein step size with absolute and relative tolerance $10^{-11}$; see \cite{azmi2023nonmonotone}. \ReA{Note that a Newton-type method could also be employed as the inner solver. However, since the \FOM is assumed to be high-dimensional, we restrict ourselves to first-order methods, which are known to scale more efficiently with increasing problem size.} For the \MPC algorithm, we use the sampling time $\delta = \tau $ and the prediction horizon as $T= 20\tau$ leading to $K$ \MPC steps. Furthermore, in the \MPC subproblem, we set $\mu =0$, i.e., we consider no terminal cost. The reduced dimension $\stateDimRed$ is chosen such that $\calE(\stateDimRed)\approx 1-\varrho$ for $\varrho=10^{-12}$ and the $7$ most recent states and adjoints are used to recompute the \POD basis. \ReA{To obtain unbiased results for the computation times, we repeat each experiment ten times and report the average computation times and speed-ups.}
\subsection{Investigation of the error estimator for the open-loop problem}
\label{subsec:numerics:openloop}
In the following, we examine the sharpness of the proposed error bounds for the state, adjoint state, and optimal control of the open-loop problem. Additionally, we analyze the propagated errors in the \MPC state trajectory and \MPC control within the reduced \MPC algorithms. Initially, we focus on the open-loop problem of size $20\tau$ with a randomly fixed initial value $\theta_n \in H$. The \POD basis of size $\stateDimRed$ is constructed from the state and adjoint of the optimal \FOM solution. The effectivity of an error estimator $\Delta$ for a given true error is defined as $\text{eff}(\Delta) \vcentcolon= \nicefrac{\text{true error}}{\Delta} \in (0,1]$. An effectivity close to one indicates a sharp estimator.

In \Cref{fig:ErrorEstStateAdjoint}, we present the error bounds and effectivities for the state and adjoint state, derived in \Cref{thm:errorBoundStateEnergy}~and~\ref{thm:errorBoundAdjointEnergy}, across various basis sizes $\stateDimRed$. The results reveal that the state error bound, with an effectivity near one, is significantly sharper than the adjoint error bound, whose effectivity is closer to $10^{-1}$. This difference arises primarily from the larger constants in the adjoint error bound, a consequence of employing Young's inequality and accounting for error contributions at the switching times.
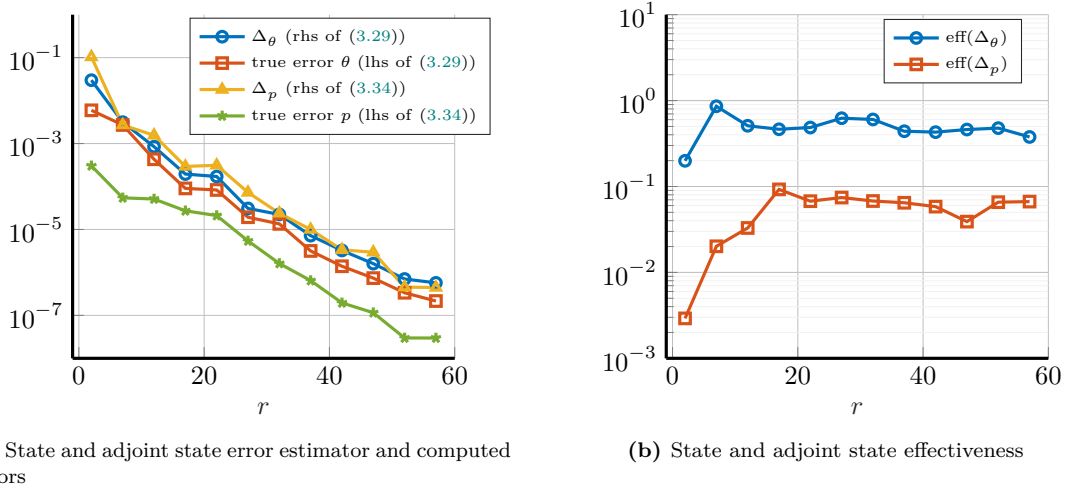
\begin{figure}[htbp]
	\centering
	\begin{subfigure}[t]{0.48\textwidth}
        \centering
	    \begin{tikzpicture}
\begin{axis}[%
	width=\imageWidth,
	height=\imageHeight,
	scale only axis,
	scaled ticks=false,
	grid=both,
	grid style={line width=.1pt, draw=gray!10},
	major grid style={line width=.2pt,draw=gray!50},
	axis lines*=left,
	axis line style={line width=\lineWidth},
xmin=-1,
xmax=60,
xlabel style={font=\color{white!15!black}},
xlabel={$r$},,
ymode=log,
ymin=1e-08,
ymax=1,
yminorticks=true,
ylabel style={font=\color{white!15!black}},
	axis background/.style={fill=white},
	legend style={%
		legend cell align=left, 
		align=left, 
		font=\tiny,
		draw=white!15!black,
		at={(0.7,0.65)},
		anchor=south,},
  ]

					\addplot [color = mycolor1, line width=\lineWidth, mark=o, mark options={ solid, fill=mycolor1} ] table [x=x1, y=y1, col sep=comma]{code/data/example1_error_est/state_estest_basis.csv};
					\addlegendentry{$\Delta_\theta$ (rhs of \eqref{eqn:stateBound:Venergy})}
					
					\addplot [color = mycolor2, line width=\lineWidth, mark=square, mark options={ fill=mycolor2}] table [x=x3, y=y3, col sep=comma] {code/data/example1_error_est/state_estest_basis.csv};
					\addlegendentry{true error $\state$ (lhs of \eqref{eqn:stateBound:Venergy})}
					
					\addplot [color = mycolor3, line width=\lineWidth, mark=triangle, mark options={  fill=mycolor3} ] table [x=x1, y=y1, col sep=comma] {code/data/example1_error_est/adjoint_estest_basis.csv};
					\addlegendentry{$\Delta_p$ (rhs of \eqref{eqn:errorBoundAdjointEnergy})}
					
					
					\addplot [color = mycolor4, line width=\lineWidth, mark=star, mark options={fill=mycolor4}] table [x=x3, y=y3, col sep=comma] 
					{code/data/example1_error_est/adjoint_estest_basis.csv};
					\addlegendentry{true error $p$ (lhs of \eqref{eqn:errorBoundAdjointEnergy})}
					
				\end{axis}
			\end{tikzpicture}
        \subcaption{State and adjoint state error estimator and computed errors}
        \label{fig:ErrorEstStateAdjoint_sub1}
	\end{subfigure}
 	\hfill
    \begin{subfigure}[t]{0.48\textwidth}
        \centering
	    \begin{tikzpicture}
\begin{axis}[%
	width=\imageWidth,
	height=\imageHeight,
	scale only axis,
	scaled ticks=false,
	grid=both,
	grid style={line width=.1pt, draw=gray!10},
	major grid style={line width=.2pt,draw=gray!50},
	axis lines*=left,
	axis line style={line width=\lineWidth},
xmin=-1,
xmax=60,
xlabel style={font=\color{white!15!black}},
xlabel={$r$},,
ymode=log,
ymin=1e-03,
ymax=10,
yminorticks=true,
ylabel style={font=\color{white!15!black}},
	axis background/.style={fill=white},
	legend style={%
		legend cell align=left, 
		align=left, 
		font=\tiny,
		draw=white!15!black,
		at={(0.75,0.80)},
		anchor=south,},
  ]

					\addplot [color = mycolor1, line width=\lineWidth, mark=o, mark options={ solid, fill=mycolor1} ] table [x=x1, y=y1, col sep=comma]{code/data/example1_error_est/state_esteffec_basis.csv};
	      \addlegendentry{\text{eff}$(\Delta_\state)$}
					
					\addplot [color = mycolor2, line width=\lineWidth, mark=square, mark options={ fill=mycolor2}] table [x=x1, y=y1, col sep=comma] {code/data/example1_error_est/adjoint_esteffec_basis.csv};
					\addlegendentry{\text{eff}$(\Delta_\adState)$}

				\end{axis}
			\end{tikzpicture}
        \subcaption{State and adjoint state effectiveness}
        \label{fig:ErrorEstStateAdjoint_sub2}
	\end{subfigure}
	\caption{Decay of the error estimator $\Delta_\theta, \Delta_p$ and effectivities against the basis size $\stateDimRed$ for the open-loop problem}
	\label{fig:ErrorEstStateAdjoint}
\end{figure}
Next, we depict in \Cref{fig:ErrorEstControl} the control error estimates $\Delta_A$, $\Delta_B$, and $\reduce{\Delta}_B$ without initial perturbation. The results show that the expensive-to-evaluate bounds $\Delta_{A}$ and $\Delta_{B}$ are very sharp with effectiveness close to one. This sharpness is because these bounds are generalizations of perturbation bounds from the literature, which are known to perform well for regularization/coercivity parameters $\lambda$ of value around $10^{-2}$; see \cite{bader2016certified, GubV17}. In contrast, the cheap-to-evaluate bound $\reduce{\Delta}_{B}$ has an effectivity around $10^{-3}$. This reduced accuracy is primarily due to the estimate of the output against the state and the estimate in \Cref{cor:errBoundOptimalControl}, where we estimate the $L^2$-energy norm against the weighted norm $\|\cdot\|_{a_\sigma, \omega}$. Despite their lower sharpness, the key advantage of these bounds is their significantly reduced evaluation time compared to $\Delta_{A}$ and $\Delta_{B}$, which we demonstrate in the forthcoming \Cref{fig:mpcExample1TimeBarplot}. 

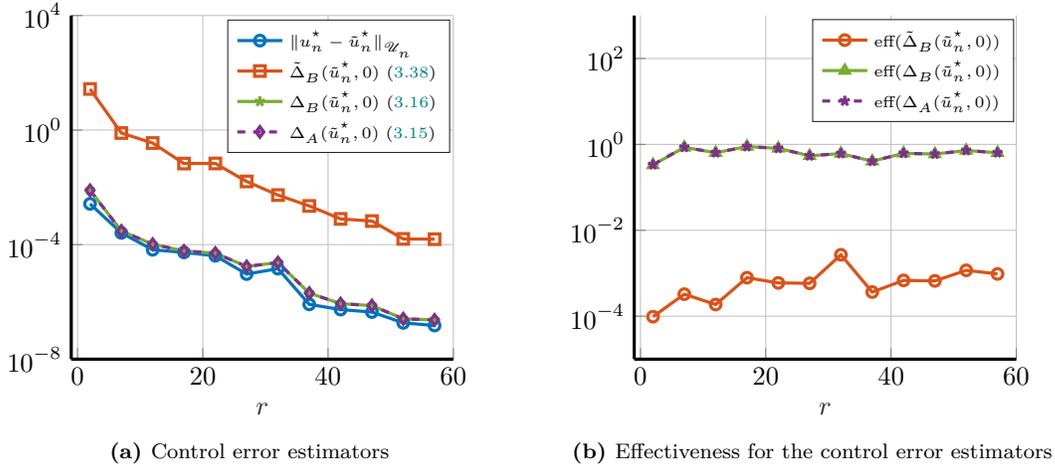
\begin{figure}[!ht]
		\centering{
	\begin{subfigure}[t]{0.48\textwidth}
	    \begin{tikzpicture}
\begin{axis}[%
	width=\imageWidth,
	height=\imageHeight,
	scale only axis,
	scaled ticks=false,
	grid=both,
	grid style={line width=.1pt, draw=gray!10},
	major grid style={line width=.2pt,draw=gray!50},
	axis lines*=left,
	axis line style={line width=\lineWidth},
xmin=-1,
xmax=60,
xlabel style={font=\color{white!15!black}},
xlabel={$r$},,
ymode=log,
ymin=1e-08,
ymax=1e4,
yminorticks=true,
ylabel style={font=\color{white!15!black}},
	axis background/.style={fill=white},
	legend style={%
		legend cell align=left, 
		align=left, 
		font=\tiny,
		draw=white!15!black,
		at={(0.7,0.60)},
		anchor=south,},
  ]

					\addplot [color = mycolor1, line width=\lineWidth, mark=o, mark options={ solid, fill=mycolor1} ] table [x=x1, y=y1, col sep=comma]{code/data/example1_error_est/basissize_control_error_est.csv};
					\addlegendentry{$\|\optInpVarN-\optInpVarRedN \|_{\U_n}$}
					
					\addplot [color = mycolor2, line width=\lineWidth, mark=square, mark options={ fill=mycolor2}] table [x=x2, y=y2, col sep=comma] {code/data/example1_error_est/basissize_control_error_est.csv};
    \addlegendentry{$\reduce \Delta_B(\optInpVarRedN,0)$ \eqref{eq: a post estimate RB realization}}
					
					
					
					\addplot [color = mycolor4, line width=\lineWidth, mark=star, mark options={fill=mycolor4}] table [x=x4, y=y4, col sep=comma] 
		{code/data/example1_error_est/basissize_control_error_est.csv};
     \addlegendentry{$\Delta_{B}(\optInpVarRedN,0)$ \eqref{control:bound:B}}

     \addplot [color = mycolor5, dashed, line width=\lineWidth, mark=diamond, mark options={solid, fill=mycolor5}] table [x=x5, y=y5, col sep=comma] {code/data/example1_error_est/basissize_control_error_est.csv};
     \addlegendentry{$\Delta_{A}(\optInpVarRedN,0)$ \eqref{control:bound:A}}
					
				\end{axis}
			\end{tikzpicture}
     \subcaption{Control error estimators}
     \label{fig:ErrorEstControl_sub1}
	\end{subfigure}
 	\hspace{0.05cm}
 \begin{subfigure}[t]{0.48\textwidth}
	    \begin{tikzpicture}
\begin{axis}[%
	width=\imageWidth,
	height=\imageHeight,
	scale only axis,
	scaled ticks=false,
	grid=both,
	grid style={line width=.1pt, draw=gray!10},
	major grid style={line width=.2pt,draw=gray!50},
	axis lines*=left,
	axis line style={line width=\lineWidth},
xmin=-1,
xmax=60,
xlabel style={font=\color{white!15!black}},
xlabel={$r$},,
ymode=log,
ymin=1e-05,
ymax=1e3,
yminorticks=true,
ylabel style={font=\color{white!15!black}},
	axis background/.style={fill=white},
	legend style={%
		legend cell align=left, 
		align=left, 
		font=\tiny,
		draw=white!15!black,
			at={(0.73,0.69)},
		anchor=south,},
  ]

					\addplot [color = mycolor2, line width=\lineWidth, mark=o, mark options={ solid, fill=mycolor2} ] table [x=x1, y=y1, col sep=comma]{code/data/example1_error_est/basissize_control_eff.csv};
        \addlegendentry{\text{eff}$(\reduce\Delta_{B}(\optInpVarRedN,0))$}
					
					
					\addplot [color = mycolor4, line width=\lineWidth, mark=triangle, mark options={  fill=mycolor4} ] table [x=x3, y=y3, col sep=comma] {code/data/example1_error_est/basissize_control_eff.csv};
        \addlegendentry{\text{eff}$(\Delta_{B}(\optInpVarRedN,0))$}
					
					
					\addplot [color = mycolor5, dashed, line width=\lineWidth, mark=star, mark options={solid, fill=mycolor5}] table [x=x4, y=y4, col sep=comma] 
			{code/data/example1_error_est/basissize_control_eff.csv};
        \addlegendentry{\text{eff}$(\Delta_{A}(\optInpVarRedN,0))$}
					
				\end{axis}
			\end{tikzpicture}
        \subcaption{Effectiveness for the control error estimators}
     \label{fig:ErrorEstControl_sub2}
	\end{subfigure}
 }
	\caption{Decay of the error estimators $\Delta_A, \Delta_B, \reduce{\Delta}_B$ (left) and corresponding effectiveness rate (right) with respect to the basis size $\stateDimRed$ for the open-loop problem}
 \label{fig:ErrorEstControl}
\end{figure}

Finally, we fix the basis size at $\stateDimRed=60$ and examine the propagated error estimation within the \MPC framework (see \Cref{cor:aposterioriBoundMPC}) for the \FOM-\ROM-\MPC and \ROM-\ROM-\MPC algorithms. In \Cref{fig:mpcExample1FOMROMMPCEST,fig:mpcExample1ROMROMMPCEST}, we illustrate the evolution of the propagated control and state error bounds $\Delta_u(\initBound)$ and $\initBound(\Delta_u(\initBoundn{n-1}), \initBoundn{n-1})$ for $\Delta_u \in \{\Delta_A, \Delta_B, \reduce{\Delta}_B\}$ over the first $n=150$ \MPC iterations for both the reduced \MPC schemes. In all cases, the control and state error estimators initially grow rapidly and then stabilize in between the switching intervals, maintaining an overestimation of approximately two or three orders of magnitude. At the switching points, we observe an increase in the error estimators, which can be attributed to the constants $\omega$ in the state error estimator. Concurrently, the true error shows an increasing trend, highlighting the need to incorporate adaptive \ROM updates throughout the \MPC iterations to mitigate error accumulation and improve the overall performance.
\begin{figure}[htbp]
		\centering{
  \begin{subfigure}[t]{0.48\textwidth}
	    \begin{tikzpicture}
\begin{axis}[%
	width=\imageWidth,
	height=\imageHeight,
	scale only axis,
	scaled ticks=false,
	grid=both,
	grid style={line width=.1pt, draw=gray!10},
	major grid style={line width=.2pt,draw=gray!50},
	axis lines*=left,
	axis line style={line width=\lineWidth},
xmin=-5,
xmax=160,
xlabel style={font=\color{white!15!black}},
xlabel={$n$},,
ymode=log,
ymin=4e-06,
ymax=5e1,
yminorticks=true,
ylabel style={font=\color{white!15!black}},
	axis background/.style={fill=white},
	legend style={%
		legend cell align=left, 
		align=left, 
		font=\tiny,
		draw=white!15!black,
		at={(0.65,0.4)},
		anchor=north,},
  ]

					\addplot [color = mycolor1, line width=\lineWidth] table [x=x1, y=y1, col sep=comma]{code/data/example1_error_est/FOMROMcontrol_pretr.csv};
    \addlegendentry{$\Delta_A(\optInpVarRedN,\initBound)$}
					
					\addplot [color = mycolor2,dashed, line width=\lineWidth] table [x=x2, y=y2, col sep=comma] {code/data/example1_error_est/FOMROMcontrol_pretr.csv};
    \addlegendentry{$\Delta_B(\optInpVarRedN,\initBound)$}
					
					
					
					\addplot [color = mycolor4, dash dot, line width=\lineWidth] table [x=x4, y=y4, col sep=comma] 
	{code/data/example1_error_est/FOMROMcontrol_pretr.csv};
        \addlegendentry{$\reduce \Delta_{B}(\optInpVarRedN,\initBound)$}

      \addplot [color = mycolor5, dash dot dot, line width=\lineWidth] table [x=x5, y=y5, col sep=comma] {code/data/example1_error_est/FOMROMcontrol_pretr.csv};
        \addlegendentry{$\|\intermediateInpVarMPC - \inpVarMPC\|_{\U_n}$}
					
				\end{axis}
			\end{tikzpicture}
     \subcaption{\FOM-\ROM control error}
     \label{fig:mpcExample1FOMROMMPCEST_sub1}
	\end{subfigure}
 	\hspace{0.05cm}
 \begin{subfigure}[t]{0.48\textwidth}
	    \begin{tikzpicture}
\begin{axis}[%
	width=\imageWidth,
	height=\imageHeight,
	scale only axis,
	scaled ticks=false,
	grid=both,
	grid style={line width=.1pt, draw=gray!10},
	major grid style={line width=.2pt,draw=gray!50},
	axis lines*=left,
	axis line style={line width=\lineWidth},
xmin=-5,
xmax=160,
xlabel style={font=\color{white!15!black}},
xlabel={$n$},,
ymode=log,
ymin=3e-04,
ymax=2e0,
yminorticks=true,
ylabel style={font=\color{white!15!black}},
	axis background/.style={fill=white},
	legend style={%
		legend cell align=left, 
		align=left, 
		font=\tiny,
		draw=white!15!black,
		at={(0.65,0.65)},
		anchor=south,},
  ]

					\addplot [color = mycolor1, line width=\lineWidth] table [x=x1, y=y1, col sep=comma]{code/data/example1_error_est/FOMROMeffect_control_pretr.csv};
    \addlegendentry{\text{eff}$(\Delta_A(\optInpVarRedN,\initBound))$}
					
					\addplot [color = mycolor2,dashed, line width=\lineWidth] table [x=x2, y=y2, col sep=comma] {code/data/example1_error_est/FOMROMeffect_control_pretr.csv};
     \addlegendentry{\text{eff}$(\Delta_B(\optInpVarRedN,\initBound))$}
					
					
					
					\addplot [color = mycolor4, dash dot, line width=\lineWidth] table [x=x4, y=y4, col sep=comma] 
	{code/data/example1_error_est/FOMROMeffect_control_pretr.csv};
        \addlegendentry{\text{eff}$(\reduce \Delta_{B}(\optInpVarRedN,\initBound))$}

				\end{axis}
			\end{tikzpicture}
        \subcaption{\FOM-\ROM control error effectiveness}
     \label{fig:mpcExample1FOMROMMPCEST_sub2}
	\end{subfigure}
	\begin{subfigure}[t]{0.48\textwidth}
	    \begin{tikzpicture}
\begin{axis}[%
	width=\imageWidth,
	height=\imageHeight,
	scale only axis,
	scaled ticks=false,
	grid=both,
	grid style={line width=.1pt, draw=gray!10},
	major grid style={line width=.2pt,draw=gray!50},
	axis lines*=left,
	axis line style={line width=\lineWidth},
xmin=-5,
xmax=160,
xlabel style={font=\color{white!15!black}},
xlabel={$n$},,
ymode=log,
ymin=3e-06,
ymax=1e2,
yminorticks=true,
ylabel style={font=\color{white!15!black}},
	axis background/.style={fill=white},
	legend style={%
		legend cell align=left, 
		align=left, 
		font=\tiny,
		draw=white!15!black,
		at={(0.6,0.4)},
		anchor=north,},
  ]

					\addplot [color = mycolor1, line width=\lineWidth] table [x=x1, y=y1, col sep=comma]{code/data/example1_error_est/FOMROMstate_pretr.csv};
     \addlegendentry{$\initBound(\Delta_A)$}
					
					\addplot [color = mycolor2,dashed, line width=\lineWidth] table [x=x2, y=y2, col sep=comma] {code/data/example1_error_est/FOMROMstate_pretr.csv};
     \addlegendentry{$\initBound(\Delta_B)$}
					
					
					
					\addplot [color = mycolor4, dash dot, line width=\lineWidth] table [x=x4, y=y4, col sep=comma] 
	{code/data/example1_error_est/FOMROMstate_pretr.csv};
         \addlegendentry{$\initBound(\reduce \Delta_{B})$}

      \addplot [color = mycolor5, dash dot dot, line width=\lineWidth] table [x=x5, y=y5, col sep=comma] {code/data/example1_error_est/FOMROMstate_pretr.csv};
        \addlegendentry{true error (lhs of \eqref{eqn:optStateBound:MPC:FOMROM_terminal})}
					
				\end{axis}
			\end{tikzpicture}
     \subcaption{\FOM-\ROM state error, see \eqref{eq:D_init_FOMROM}}
     \label{fig:mpcExample1FOMROMMPCEST_sub3}
	\end{subfigure}
 	\hspace{0.05cm}
 \begin{subfigure}[t]{0.48\textwidth}
	    \begin{tikzpicture}
\begin{axis}[%
	width=\imageWidth,
	height=\imageHeight,
	scale only axis,
	scaled ticks=false,
	grid=both,
	grid style={line width=.1pt, draw=gray!10},
	major grid style={line width=.2pt,draw=gray!50},
	axis lines*=left,
	axis line style={line width=\lineWidth},
xmin=-5,
xmax=160,
xlabel style={font=\color{white!15!black}},
xlabel={$n$},,
ymode=log,
ymin=1e-05,
ymax=2e0,
yminorticks=true,
ylabel style={font=\color{white!15!black}},
	axis background/.style={fill=white},
	legend style={%
		legend cell align=left, 
		align=left, 
		font=\tiny,
		draw=white!15!black,
		at={(0.7,0.55)},
		anchor=south,},
  ]

					\addplot [color = mycolor1, line width=\lineWidth] table [x=x1, y=y1, col sep=comma]{code/data/example1_error_est/FOMROMeffect_state_pretr.csv};
    \addlegendentry{\text{eff}$(\initBound(\Delta_A))$}
					
					\addplot [color = mycolor2,dashed, line width=\lineWidth] table [x=x2, y=y2, col sep=comma] {code/data/example1_error_est/FOMROMeffect_state_pretr.csv};
     \addlegendentry{\text{eff}$(\initBound(\Delta_B))$}
					
					
					
					\addplot [color = mycolor4, dash dot, line width=\lineWidth] table [x=x4, y=y4, col sep=comma] 
	{code/data/example1_error_est/FOMROMeffect_state_pretr.csv};
         \addlegendentry{\text{eff}$(\initBound(\reduce\Delta_{B}))$}

				\end{axis}
			\end{tikzpicture}
        \subcaption{\FOM-\ROM state error effectiveness}
     \label{fig:mpcExample1FOMROMMPCEST_sub4}
	\end{subfigure}
 }
	\caption{\FOM-\ROM-\MPC error estimators (left) and corresponding effectiveness rate (right) over the \MPC steps $n$}
	\label{fig:mpcExample1FOMROMMPCEST}
\end{figure}
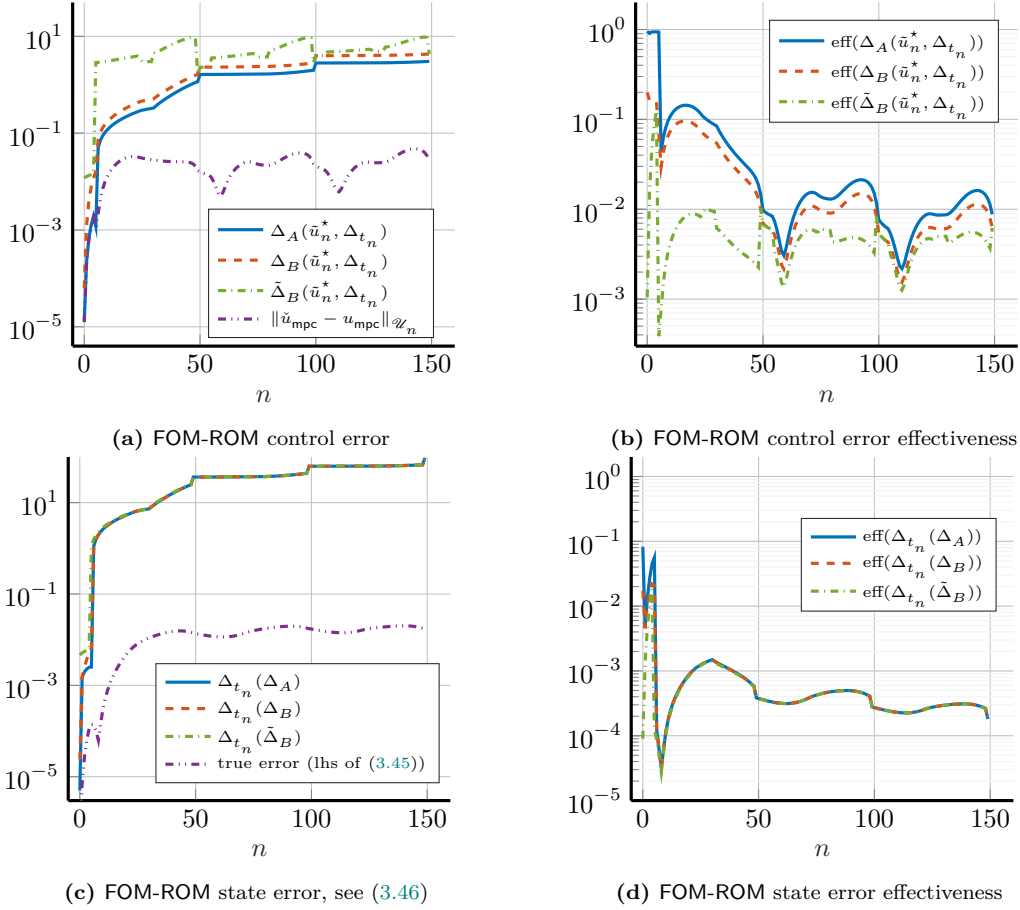
\begin{figure}[!ht]
		\centering{
  \begin{subfigure}[t]{0.48\textwidth}
	    \begin{tikzpicture}
\begin{axis}[%
	width=\imageWidth,
	height=\imageHeight,
	scale only axis,
	scaled ticks=false,
	grid=both,
	grid style={line width=.1pt, draw=gray!10},
	major grid style={line width=.2pt,draw=gray!50},
	axis lines*=left,
	axis line style={line width=\lineWidth},
xmin=-5,
xmax=160,
xlabel style={font=\color{white!15!black}},
xlabel={$n$},,
ymode=log,
ymin=1e-05,
ymax=5e1,
yminorticks=true,
ylabel style={font=\color{white!15!black}},
	axis background/.style={fill=white},
	legend style={%
		legend cell align=left, 
		align=left, 
		font=\tiny,
		draw=white!15!black,
		at={(0.69,0.05)},
		anchor=south,},
  ]

					\addplot [color = mycolor1, line width=\lineWidth] table [x=x1, y=y1, col sep=comma]{code/data/example1_error_est/ROMROMcontrol_pretr.csv};
         \addlegendentry{$\Delta_A(\optInpVarRedN,\initBound)$}   
					
					\addplot [color = mycolor2,dashed, line width=\lineWidth] table [x=x2, y=y2, col sep=comma] {code/data/example1_error_est/ROMROMcontrol_pretr.csv};
    \addlegendentry{$\Delta_B(\optInpVarRedN,\initBound)$}
					
					
					
					\addplot [color = mycolor4, dash dot, line width=\lineWidth] table [x=x4, y=y4, col sep=comma] 
	{code/data/example1_error_est/ROMROMcontrol_pretr.csv};
        \addlegendentry{$\tilde \Delta_{B}(\optInpVarRedN,\initBound)$}

      \addplot [color = mycolor5, dash dot dot, line width=\lineWidth] table [x=x5, y=y5, col sep=comma] {code/data/example1_error_est/ROMROMcontrol_pretr.csv};
        \addlegendentry{$\|\inpVarRedMPC - \inpVarMPC\|_{\U_n}$}
					
				\end{axis}
			\end{tikzpicture}
     \subcaption{\ROM-\ROM control error}
     \label{fig:mpcExample1ROMROMMPCEST_sub1}
	\end{subfigure}
 	\hspace{0.05cm}
 \begin{subfigure}[t]{0.48\textwidth}
	    \begin{tikzpicture}
\begin{axis}[%
	width=\imageWidth,
	height=\imageHeight,
	scale only axis,
	scaled ticks=false,
	grid=both,
	grid style={line width=.1pt, draw=gray!10},
	major grid style={line width=.2pt,draw=gray!50},
	axis lines*=left,
	axis line style={line width=\lineWidth},
xmin=-5,
xmax=160,
xlabel style={font=\color{white!15!black}},
xlabel={$n$},,
ymode=log,
ymin=2e-04,
ymax=2e0,
yminorticks=true,
ylabel style={font=\color{white!15!black}},
	axis background/.style={fill=white},
	legend style={%
		legend cell align=left, 
		align=left, 
		font=\tiny,
		draw=white!15!black,
		at={(0.68,0.65)},
		anchor=south,},
  ]

					\addplot [color = mycolor1, line width=\lineWidth] table [x=x1, y=y1, col sep=comma]{code/data/example1_error_est/ROMROMeffect_control_pretr.csv};
    \addlegendentry{\text{eff}$(\Delta_A(\optInpVarRedN,\initBound))$}
					
					\addplot [color = mycolor2,dashed, line width=\lineWidth] table [x=x2, y=y2, col sep=comma] {code/data/example1_error_est/ROMROMeffect_control_pretr.csv};
     \addlegendentry{\text{eff}$(\Delta_B(\optInpVarRedN,\initBound))$}
					
					
					
					\addplot [color = mycolor4, dash dot, line width=\lineWidth] table [x=x4, y=y4, col sep=comma] 
	{code/data/example1_error_est/ROMROMeffect_control_pretr.csv};
        \addlegendentry{\text{eff}$(\tilde \Delta_{B}(\optInpVarRedN,\initBound))$}

				\end{axis}
			\end{tikzpicture}
        \subcaption{\ROM-\ROM control error effectiveness}
     \label{fig:mpcExample1ROMROMMPCEST_sub2}
	\end{subfigure}
	\begin{subfigure}[t]{0.48\textwidth}
	    \begin{tikzpicture}
\begin{axis}[%
	width=\imageWidth,
	height=\imageHeight,
	scale only axis,
	scaled ticks=false,
	grid=both,
	grid style={line width=.1pt, draw=gray!10},
	major grid style={line width=.2pt,draw=gray!50},
	axis lines*=left,
	axis line style={line width=\lineWidth},
xmin=-5,
xmax=160,
xlabel style={font=\color{white!15!black}},
xlabel={$n$},,
ymode=log,
ymin=1e-03,
ymax=1e2,
yminorticks=true,
ylabel style={font=\color{white!15!black}},
	axis background/.style={fill=white},
	legend style={%
		legend cell align=left, 
		align=left, 
		font=\tiny,
		draw=white!15!black,
		at={(0.60,0.05)},
		anchor=south,},
  ]

					\addplot [color = mycolor1, line width=\lineWidth] table [x=x1, y=y1, col sep=comma]{code/data/example1_error_est/ROMROMstate_pretr.csv};
     \addlegendentry{$\initBound(\Delta_A)$}
					
					\addplot [color = mycolor2,dashed, line width=\lineWidth] table [x=x2, y=y2, col sep=comma] {code/data/example1_error_est/ROMROMstate_pretr.csv};
     \addlegendentry{$\initBound(\Delta_B)$}
					
					
					
					\addplot [color = mycolor4, dash dot, line width=\lineWidth] table [x=x4, y=y4, col sep=comma] 
	{code/data/example1_error_est/ROMROMstate_pretr.csv};
         \addlegendentry{$\initBound(\tilde \Delta_{B})$}

      \addplot [color = mycolor5, dash dot dot, line width=\lineWidth] table [x=x5, y=y5, col sep=comma] {code/data/example1_error_est/ROMROMstate_pretr.csv};
        \addlegendentry{true error (lhs of \eqref{eqn:optStateBound:MPC:ROMROM_terminal})}
					
				\end{axis}
			\end{tikzpicture}
     \subcaption{\ROM-\ROM state error, see \eqref{eq:D_init_ROMROM}}
     \label{fig:mpcExample1ROMROMMPCEST_sub3}
	\end{subfigure}
 	\hspace{0.05cm}
 \begin{subfigure}[t]{0.48\textwidth}
	    \begin{tikzpicture}
\begin{axis}[%
	width=\imageWidth,
	height=\imageHeight,
	scale only axis,
	scaled ticks=false,
	grid=both,
	grid style={line width=.1pt, draw=gray!10},
	major grid style={line width=.2pt,draw=gray!50},
	axis lines*=left,
	axis line style={line width=\lineWidth},
    xmin=-5,
    xmax=160,
    xlabel style={font=\color{white!15!black}},
    xlabel={$n$},,
    ymode=log,
    ymin=3e-03,
    ymax=2e0,
    yminorticks=true,
    ylabel style={font=\color{white!15!black}},
	axis background/.style={fill=white},
	legend style={%
		legend cell align=left, 
		align=left, 
		font=\tiny,
		draw=white!15!black,
		at={(0.7,0.50)},
		anchor=south,},
  ]

					\addplot [color = mycolor1, line width=\lineWidth] table [x=x1, y=y1, col sep=comma]{code/data/example1_error_est/ROMROMeffect_state_pretr.csv};
    \addlegendentry{\text{eff}$(\initBound(\Delta_A))$}
					
					\addplot [color = mycolor2,dashed, line width=\lineWidth] table [x=x2, y=y2, col sep=comma] {code/data/example1_error_est/ROMROMeffect_state_pretr.csv};
     \addlegendentry{\text{eff}$(\initBound(\Delta_B))$}
					
					
					
					\addplot [color = mycolor4, dash dot, line width=\lineWidth] table [x=x4, y=y4, col sep=comma] 
	{code/data/example1_error_est/ROMROMeffect_state_pretr.csv};
         \addlegendentry{\text{eff}$(\initBound(\tilde \Delta_{B}))$}

				\end{axis}
			\end{tikzpicture}
        \subcaption{\ROM-\ROM state error effectiveness}
     \label{fig:mpcExample1ROMROMMPCEST_sub4}
	\end{subfigure}
 }
	\caption{\ROM-\ROM-\MPC error estimators (left) and corresponding effectiveness rate (right) over the \MPC steps $n$}
	\label{fig:mpcExample1ROMROMMPCEST}
\end{figure}
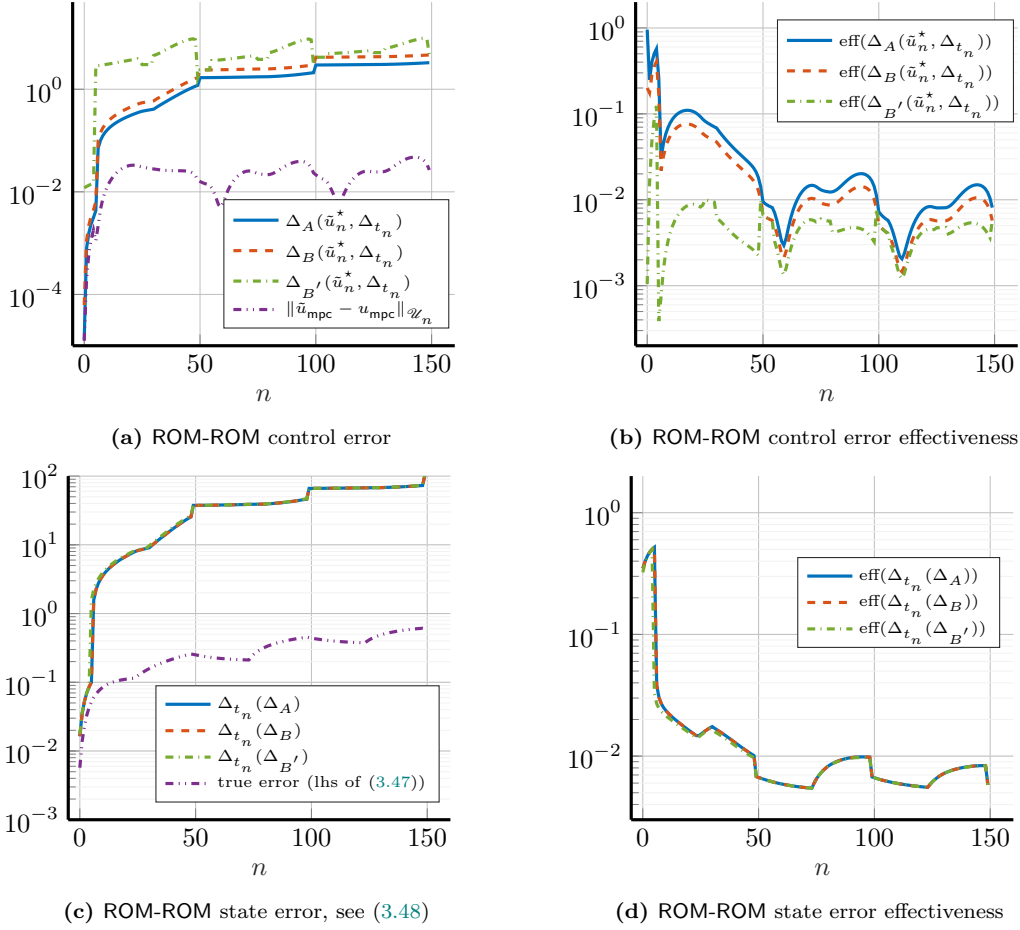
\subsection{Comparison of the MPC schemes}
\label{subsec:numerics:MPC}

%
In this section, we compare the \FOM-\MPC (\Cref{alg:FOM-MPC}), the \FOM-\ROM-\MPC (\Cref{alg:ROM-MPC}) with expensive-to-evaluate error estimator $\Delta_A$, and the \ROM-\ROM-\MPC (\Cref{alg:ROMROM-MPC}) with cheap-to-evaluate error estimator $\reduce \Delta_{B}$ in terms of computational time and approximation quality for several choices of the tolerance parameter $\varepsilon$. \ReA{We provide a computational complexity analysis of the algorithms in \Cref{app:sec_complexity}.} We observe that one could also employ the \FOM-\ROM-\MPC (\Cref{alg:ROM-MPC}) with the computationally cheaper error estimator $\reduce\Delta_{B}$. However, the results we found were similar to those of the \ROM-\ROM-\MPC. We compare the following relative errors in state and control
\begin{align*}
    e_\state \vcentcolon= &\ \frac{\|\stateMPC-{\stateRedMPC}\|_{L^2(0,\bar T, V)}}{\|\stateMPC\|_{L^2(0,\bar T, V)}}, \quad e_u \vcentcolon= \frac{\|\inpVarMPC-\inpVarRedMPC\|_{\U}}{\|\inpVarMPC\|_{\U}},
\end{align*}
and in the output, and the obtained cost functional value
\begin{align*}
    e_y \vcentcolon= &\ \frac{\|\outVarMPC-{\outVarRedMPC}\|_{L^2(0,\bar T, \R^{\outVarDim})}}{\|\outVarMPC\|_{L^2(0,\bar T, \R^{\outVarDim})}}, \quad e_{ J}\vcentcolon= \ \frac{\big|J(\outVarMPC, \inpVarMPC) -J(\outVarRedMPC, \inpVarRedMPC)\big|}{J(\outVarMPC, \inpVarMPC)},
\end{align*}
The results from the previous subsection motivate us to consider uniform tolerances with respect to $n$, specifically $\varepsilon=\epsilon\in\{10^{-2},10^{-3} \}$. The results are depicted in \Cref{fig:mpcExample1table}. For the choice $\varepsilon=10^{-2}$, we observe a speed-up factor of $11$ ($4$) for the \ROM-\ROM-\MPC (\FOM-\ROM-\MPC) scheme. The reason for the difference is mainly due to the cost of evaluating the error estimators. On the other hand, due to the different effectiveness rates of the cheap and expensive error estimators, we observe that the average reduced basis sizes of the \ROM-\ROM-\MPC (80 and 98) are significantly larger than for the \FOM-\ROM-\MPC (61 and 71). Further, we observe that for the tolerance $\varepsilon=10^{-3}$ the speedup of the \ROM-\ROM scheme reduces to 7.4, while the speedup of the \FOM-\ROM scheme remains constant. The reason for this is the high number of model updates (10 compared to 3) due to the fast increase of the error estimator in this case. Since the \ROM-\ROM-\MPC needs more \ROM updates and has a richer basis size, the errors in the output, control, and cost functional value are smaller than for the \FOM-\ROM-\MPC algorithm. Nevertheless, it is remarkable that the \FOM-\ROM-\MPC state admits a better true error with fewer basis elements and fewer basis updates than the \ROM-\ROM-\MPC algorithm.

In \Cref{fig:mpcExample1ErrorEstimation}, we display the control and propagated trajectory error estimators that trigger the \ROM updates over $n$. We note that computationally inexpensive error estimators accumulate faster than expensive ones, leading to a larger number of \ROM updates. We also observe an increase in the error and error estimators if the model switches, which is due to the switching constants $\omega$ in \Cref{thm:errorBoundStateEnergy} and \ref{thm:errorBoundAdjointEnergy}.

In conclusion, an in-depth analysis of the computational times is given in \Cref{fig:mpcExample1TimeBarplot}. In this case, we can appreciate the advantage of employing computationally inexpensive error estimators. Indeed, the evaluation of the expensive error estimators takes more than half of the time of the \FOM-\ROM-\MPC algorithm, while for the cheap error estimators, this constitutes only a small fraction of the total computational time of the entire \ROM-\ROM-\MPC procedure. We observe an increase in the elapsed time for the \ROM updates for the offline assembly of the error estimate \ReA{(see also \eqref{eq:app:costupdate})}, which, however, is very contained and does not affect the efficiency of the procedure. \ReA{This is in accordance with the complexity analysis of \FOM-\ROM-\MPC and \ROM-\ROM-\MPC provided in \Cref{app:sec_complexity}; see \eqref{eq:app:costFOMROM} and \eqref{eq:app:costROMROM}.}
\begin{table}[t!] 
    \scriptsize
	\centering 
	\begin{tabular}{lcccccccc}\toprule
		Algorithm & $e_u$ & $e_\state$ & $e_y$ & $e_J$ &\makecell{ time [s] \\ (min/avg./max)} & \makecell{speed-up\\(min/avg./max) } & avg. $\stateDimRed$ & \makecell{\#\ROM \\updates} \\ 
  \midrule
		\FOM & -  & -  &  - & - &151/154/158 & - & - & - \\
     \midrule
     $\varepsilon=\epsilon=10^{-2}$ &  &   &   &  &&  &  &  \\
      \midrule
		\FOM-\ROM   &$6.1\cdot 10^{-4}$ & $7.7\cdot 10^{-5}$ & $7.5\cdot 10^{-5}$ & $7.7\cdot 10^{-5}$ & 35/36/37 &4.0/4.1/4.4&61&2\\ 
        \ROM-\ROM   &$4.4\cdot 10^{-6}$ & $1.9\cdot 10^{-4}$ & $9.7\cdot 10^{-7}$ & $7.1\cdot 10^{-7}$ & 15/16/18 &9/11/13.7&80&5\\
        \midrule
     $\varepsilon=\epsilon=10^{-3}$ &  &   &   &  &&  &  &  \\
      \midrule
		\FOM-\ROM   &$6.3\cdot 10^{-5}$ & $7.6\cdot 10^{-6}$ & $7.5\cdot 10^{-6}$ & $3.1\cdot 10^{-7}$ &  36/38/39 &3.9/4.1/4.3&71&3\\ 
        \ROM-\ROM  &$5.3\cdot 10^{-7}$ & $8.4\cdot 10^{-6}$ & $1.3\cdot 10^{-8}$ & $1.2\cdot 10^{-9}$ & 19/20/21 &7.2/7.7/8.3&98&10\\
		\bottomrule
	\end{tabular}
     \caption{\ReA{Performance comparison of the MPC schemes for update tolerances $(\varepsilon,\epsilon)\in\{ (10^{-2},10^{-2}), (10^{-3},10^{-3})\}$}}
	\label{fig:mpcExample1table}
\end{table}
%
\begin{figure}[!ht]
		\centering{
  \begin{subfigure}[t]{0.48\textwidth}
	    \begin{tikzpicture}
\begin{axis}[%
	width=\imageWidth,
	height=\imageHeight,
	scale only axis,
	scaled ticks=false,
	grid=both,
	grid style={line width=.1pt, draw=gray!10},
	major grid style={line width=.2pt,draw=gray!50},
	axis lines*=left,
	axis line style={line width=\lineWidth},
xmin=-10,
xmax=500,
xlabel style={font=\color{white!15!black}},
xlabel={$n$},,
ymode=log,
ymin=3e-07,
ymax=5e-2,
yminorticks=true,
ylabel style={font=\color{white!15!black}},
	axis background/.style={fill=white},
	legend style={%
		legend cell align=left, 
		align=left, 
		font=\tiny,
		draw=white!15!black,
		at={(0.7,0.6)},
		anchor=south,},
  ]

					\addplot [color = mycolor1, dash dot, line width=\lineWidth] table [x=x1, y=y1, col sep=comma]{code/data/example1_mpc/FOMROM_expensive_control_est.csv};
    \addlegendentry{$\Delta_A( \optInpVarRedN,0)$}
					
					\addplot [color = mycolor2,dashed, line width=\lineWidth] table [x=x2, y=y2, col sep=comma] {code/data/example1_mpc/FOMROM_expensive_control_est.csv};
    \addlegendentry{$\Delta_A(\optInpVarRedN,\initBound)$}
					
					\addplot [color = mycolor5, line width=1.1*\lineWidth ] table [x=x3, y=y3, col sep=comma] {code/data/example1_mpc/FOMROM_expensive_control_est.csv};
    \addlegendentry{$\varepsilon$}
					
					
					 \addplot[domain=3e-07:0.01, color=black, dotted, line width=\lineWidth ] ({0}, x);
    \addplot[domain=3e-07:0.01, color= black, dotted, line width=\lineWidth ] ({13}, x);
    \addlegendentry{\ROM updates}
					
				\end{axis}
			\end{tikzpicture}
     \subcaption{\FOM-\ROM-\MPC control error estimator}
     \label{fig:mpcExample1ErrorEst_sub1}
	\end{subfigure}
 	\hspace{0.05cm}
 \begin{subfigure}[t]{0.48\textwidth}
	    \begin{tikzpicture}
\begin{axis}[%
	width=\imageWidth,
	height=\imageHeight,
	scale only axis,
	scaled ticks=false,
	grid=both,
	grid style={line width=.1pt, draw=gray!10},
	major grid style={line width=.2pt,draw=gray!50},
	axis lines*=left,
	axis line style={line width=\lineWidth},
xmin=-10,
xmax=500,
xlabel style={font=\color{white!15!black}},
xlabel={$n$},,
ymode=log,
ymin=3e-07,
ymax=5e-2,
yminorticks=true,
ylabel style={font=\color{white!15!black}},
	axis background/.style={fill=white},
	legend style={%
		legend cell align=left, 
		align=left, 
		font=\tiny,
		draw=white!15!black,
		at={(0.45,0.05)},
		anchor=south,},
  ]

					\addplot [color = mycolor1, dashed, line width=\lineWidth] table [x=x1, y=y1, col sep=comma]{code/data/example1_mpc/FOMROM_expensive_mpctraj_est.csv};
    \addlegendentry{$\initBound$}
					
					\addplot [color = mycolor5, line width=1.1*\lineWidth] table [x=x2, y=y2, col sep=comma] {code/data/example1_mpc/FOMROM_expensive_mpctraj_est.csv};
    \addlegendentry{$\varepsilon$}

					
					 \addplot[domain=3e-07:0.01, color=black, dotted, line width=\lineWidth ] ({0}, x);
    \addplot[domain=3e-07:0.01, color= black, dotted, line width=\lineWidth ] ({13}, x);
    \addlegendentry{\ROM updates}
					
				\end{axis}
			\end{tikzpicture}
        \subcaption{\FOM-\ROM-\MPC trajectory error estimator}
     \label{fig:mpcExample1ErrorEst_sub2}
	\end{subfigure}
	\begin{subfigure}[t]{0.48\textwidth}
	    \begin{tikzpicture}
\begin{axis}[%
	width=\imageWidth,
	height=\imageHeight,
	scale only axis,
	scaled ticks=false,
	grid=both,
	grid style={line width=.1pt, draw=gray!10},
	major grid style={line width=.2pt,draw=gray!50},
	axis lines*=left,
	axis line style={line width=\lineWidth},
xmin=-10,
xmax=500,
xlabel style={font=\color{white!15!black}},
xlabel={$n$},,
ymode=log,
ymin=1e-05,
ymax=5e-1,
yminorticks=true,
ylabel style={font=\color{white!15!black}},
	axis background/.style={fill=white},
	legend style={%
		legend cell align=left, 
		align=left, 
		font=\tiny,
		draw=white!15!black,
		at={(0.70,0.65)},
		anchor=south,},
  ]

					\addplot [color = mycolor1, dash dot, line width=\lineWidth] table [x=x1, y=y1, col sep=comma]{code/data/example1_mpc/ROMROM_cheap_control_est.csv};
   \addlegendentry{$\reduce \Delta_{B}(\optInpVarRedN,0)$}

					\addplot [color = mycolor2,dashed, line width=\lineWidth] table [x=x2, y=y2, col sep=comma] {code/data/example1_mpc/ROMROM_cheap_control_est.csv};
   \addlegendentry{$\reduce \Delta_{B}(\optInpVarRedN,\initBound)$}
					
					\addplot [color = mycolor5, line width=1.1*\lineWidth ] table [x=x3, y=y3, col sep=comma] {code/data/example1_mpc/ROMROM_cheap_control_est.csv};
   \addlegendentry{$\varepsilon$}
					
					
					 \addplot[domain=0.00001:0.01, color=black, dotted, line width=\lineWidth ] ({0}, x);
    \addplot[domain=0.00001:0.01, color=black, dotted, line width=\lineWidth ] ({5}, x);
    \addplot[domain=0.00001:0.01, color=black, dotted, line width=\lineWidth ] ({9}, x);
    \addplot[domain=0.00001:0.01, color=black, dotted, line width=\lineWidth ] ({199}, x);
    \addplot[domain=0.00001:0.01, color=black, dotted, line width=\lineWidth ] ({349}, x);
    \addlegendentry{\ROM updates}
					
				\end{axis}
			\end{tikzpicture}
     \subcaption{\ROM-\ROM-\MPC control error estimator}
     \label{fig:mpcExample1ErrorEst_sub3}
	\end{subfigure}
 	\hspace{0.05cm}
 \begin{subfigure}[t]{0.48\textwidth}
	    \begin{tikzpicture}
\begin{axis}[%
	width=\imageWidth,
	height=\imageHeight,
	scale only axis,
	scaled ticks=false,
	grid=both,
	grid style={line width=.1pt, draw=gray!10},
	major grid style={line width=.2pt,draw=gray!50},
	axis lines*=left,
	axis line style={line width=\lineWidth},
xmin=-10,
xmax=500,
xlabel style={font=\color{white!15!black}},
xlabel={$n$},,
ymode=log,
ymin=1e-5,
ymax=1e-1,
yminorticks=true,
ylabel style={font=\color{white!15!black}},
	axis background/.style={fill=white},
	legend style={%
		legend cell align=left, 
		align=left, 
		font=\tiny,
		draw=white!15!black,
		at={(0.70,0.78)},
		anchor=south,},
  ]

					\addplot [color = mycolor1, dashed, line width=\lineWidth] table [x=x1, y=y1, col sep=comma]{code/data/example1_mpc/ROMROM_cheap_mpctraj_est.csv};
    \addlegendentry{$\initBound$}
					
					\addplot [color = mycolor5, line width=1.1*\lineWidth] table [x=x2, y=y2, col sep=comma] {code/data/example1_mpc/ROMROM_cheap_mpctraj_est.csv};
    \addlegendentry{$\varepsilon$}


					 \addplot[domain=0.00001:0.01, color=black, dotted, line width=\lineWidth ] ({0}, x);
    \addplot[domain=0.00001:0.01, color=black, dotted, line width=\lineWidth ] ({5}, x);
    \addplot[domain=0.00001:0.01, color=black, dotted, line width=\lineWidth ] ({9}, x);
    \addplot[domain=0.00001:0.01, color=black, dotted, line width=\lineWidth ] ({199}, x);
    \addplot[domain=0.00001:0.01, color=black, dotted, line width=\lineWidth ] ({349}, x);
    \addlegendentry{\ROM updates}
    
				\end{axis}
			\end{tikzpicture}
        \subcaption{\ROM-\ROM-\MPC trajectory error estimator}
     \label{fig:mpcExample1ErrorEst_sub4}
	\end{subfigure}
 }
	\caption{Error estimators and \ROM updates over the \MPC steps $n$ for $\varepsilon=\epsilon = 10^{-2}$}
	\label{fig:mpcExample1ErrorEstimation}
\end{figure}
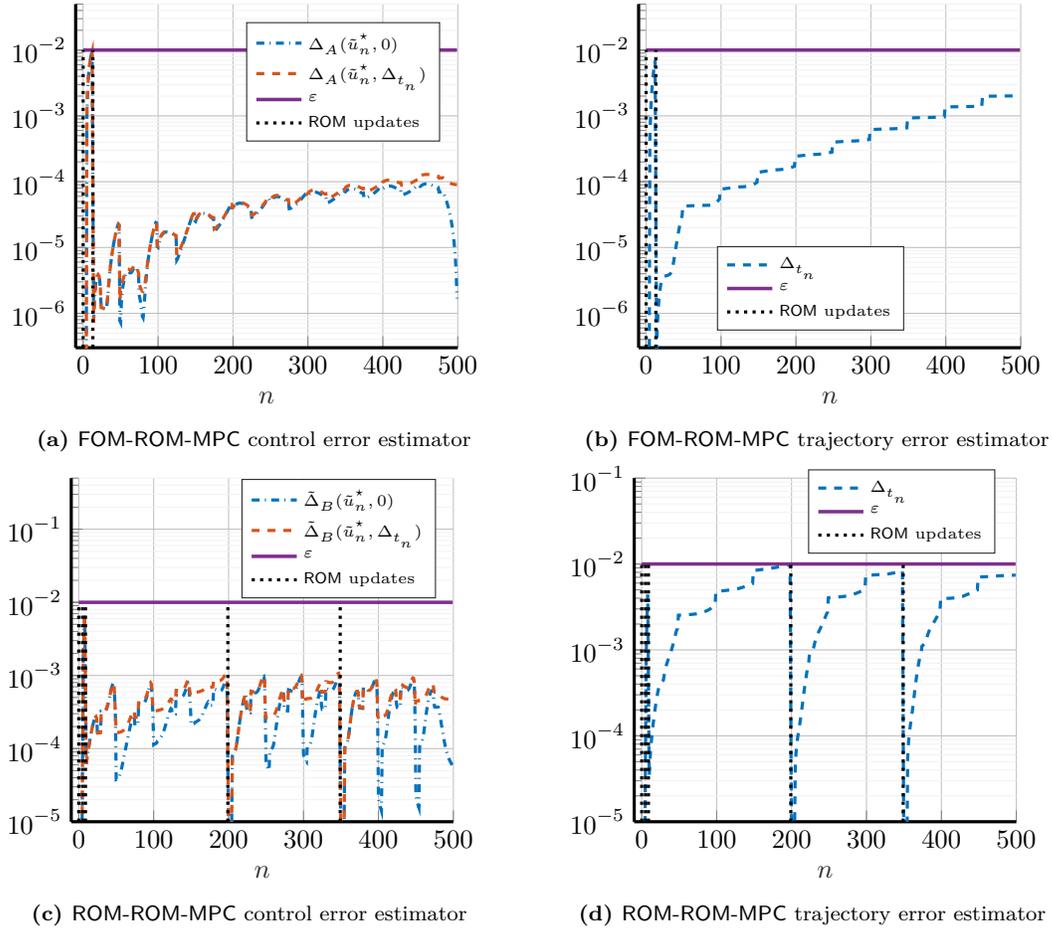
\begin{figure}[!ht]
\centering
\begin{tikzpicture}
\begin{axis} [xbar,
    bar width = 6pt,
    xmin = 0, 
    xmax = 80, 
     width=2.5*\imageWidth,
	height=1.5*\imageHeight,
    enlarge y limits={abs=0.6cm},
    enlarge x limits = {value = .25, upper},
    x tick label style={
		font=\tiny, 
        /pgf/number format/1000 sep=},
    xlabel = {time [s]},
    symbolic y coords = {total,\ROM subproblem,\FOM subproblem, error estimation,\ROM updates},
    legend style={%
		legend cell align=left,
		font=\small,
		draw=white!15!black,
  },
    ]
    \addplot [fill=mycolor1]  coordinates{
    (155,total) (0,\ROM subproblem) (155,\FOM subproblem) (0,error estimation) (0,\ROM updates)};
    
    \addplot [fill=mycolor2] coordinates{
    (35,total) (9,\ROM subproblem) (1,\FOM subproblem) (22,error estimation) (0.3,\ROM updates)};
    
    \addplot [fill=mycolor3] coordinates{
    (15,total) (10.5,\ROM subproblem) (1.2,\FOM subproblem) (1.3,error estimation) (2.4,\ROM updates)};
    \legend{\FOM-\MPC,\FOM-\ROM-\MPC,\ROM-\ROM-\MPC}
\end{axis}
\end{tikzpicture}
\caption{Overview of the elapsed time for the different \MPC schemes for $\varepsilon=\epsilon=10^{-2}$}
\label{fig:mpcExample1TimeBarplot}
\end{figure}
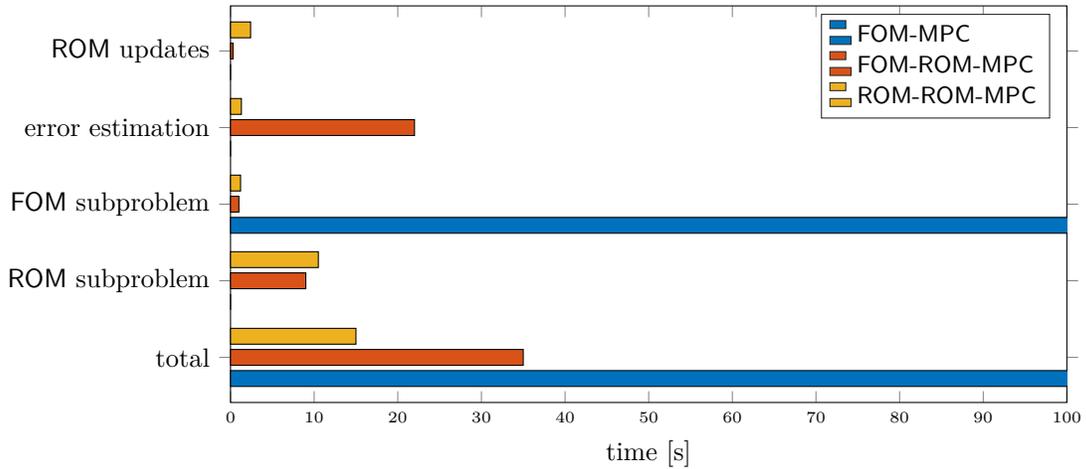
%
%
\ReA{
\subsection{Sensitivity with respect to the \POD energy parameter}
In this section, we study the sensitivity of the results with respect to the \POD energy selection parameter $\calE(r)=1-\varrho$ (see \eqref{eqn:POD:basisEnergy}) for $\varrho\in \{10^{-12}, 10^{-10}, 10^{-8} \}$ and fixed error tolerances $\varepsilon=\epsilon=10^{-2}$. The results are depicted in \Cref{fig:mpctable_PODsensi}. By increasing $\varrho$, we cut more information out of the \POD modes, resulting in a smaller average basis size in \Cref{fig:mpctable_PODsensi}. As a result, the overestimation of the error estimators becomes more apparent, and the error criterion might not be reachable. This effect is more severe for the \ROM-\ROM-\MPC that uses the cheap estimators, leading to a high number of \ROM updates and thus a reduction in speed-up for $\varrho= 10^{-10}$. Especially for $\varrho= 10^{-12}$ triggering with $433$ model updates, the \FOM-\MPC is about twice as fast as the \ROM-\ROM-\MPC. The \FOM-\ROM variant in combination with the expensive-to-evaluate error estimators is more robust against the overestimation. Thus, for bigger $\varrho$, \FOM-\ROM-\MPC obtains smaller basis sizes in combination with a similar amount of model updates, leading to increased computational performance.
\begin{table}[h!] 
    \scriptsize
	\centering 
    \caption{\ReA{Performance comparison of the MPC schemes for fixed update tolerances $(\varepsilon,\epsilon)\in\{ (10^{-2},10^{-2})$ and different \POD tolerances $\varrho$}}
	\label{fig:mpctable_PODsensi}
	\begin{tabular}{lcccccccc}\toprule
		Algorithm & $e_u$ & $e_\state$ & $e_y$ & $e_J$ &\makecell{ time [s] \\ (min/avg./max)} & \makecell{speed-up\\(min/avg./max) } & avg. $\stateDimRed$ & \makecell{\#\ROM \\updates} \\ 
  \midrule
		\FOM & -  & -  &  - & - &151/154/158 & - & - & - \\
     \midrule
     $\varrho=10^{-12}$ &  &   &   &  &&  &  &  \\
     \midrule
		\FOM-\ROM   &$6.1\cdot 10^{-4}$ & $7.7\cdot 10^{-5}$ & $7.5\cdot 10^{-5}$ & $7.7\cdot 10^{-5}$ & 35/36/37 &4.0/4.1/4.4&61&2\\ 
        \ROM-\ROM   &$4.4\cdot 10^{-6}$ & $1.9\cdot 10^{-4}$ & $9.7\cdot 10^{-7}$ & $7.1\cdot 10^{-7}$ & 15/16/18 &9/11/13.7&80&5\\
        \midrule
     $\varrho=10^{-10}$ &  &   &   &  &&  &  &  \\
      \midrule
		\FOM-\ROM   &$5.5\cdot 10^{-4}$ & $6.5\cdot 10^{-5}$ & $6.3\cdot 10^{-5}$ & $2.9\cdot 10^{-5}$ & 33/35/38 &3.9/4.1/4.3& 47 &2\\ 
        \ROM-\ROM  &$2.4\cdot 10^{-5}$ & $2.1\cdot 10^{-5}$ & $1.5\cdot 10^{-7}$ & $1.5\cdot 10^{-7}$ & 26/27/28 &5.7/6.1/6.9&57&23\\
        \midrule
         $\varrho=10^{-8}$ &  &   &   &  &&  &  &  \\
      \midrule
		\FOM-\ROM   &$7.3\cdot 10^{-4}$ & $2.0\cdot 10^{-4}$ & $2.0\cdot 10^{-4}$ & $2.2\cdot 10^{-4}$ & 33/34/36 &4.4/4.7/5.4&34&3\\ 
        \ROM-\ROM  &$1.3\cdot 10^{-5}$ & $1.8\cdot 10^{-6}$ & $9.2\cdot 10^{-7}$ & $8.3\cdot 10^{-7}$ & 213/226/245 &0.65/0.74/0.98&25&433\\
		\bottomrule
	\end{tabular}
\end{table}
}
\section{Conclusion}\label{sec:conclusion}
We presented an efficient framework for model predictive control of linear switched evolution equations based on Galerkin reduced-order modeling techniques. First, we provided optimality conditions for the \MPC subproblems that allow box constraints and $L^1$ regularization. We show that, due to switching in the operator in front of the time derivative, jumps at the switching times arise for the adjoint state. Based on these optimality conditions, we constructed a posteriori error estimates for the reduced \MPC feedback law and the closed-loop state. These estimates show that the \ROM-\MPC trajectory evolves within a neighborhood of the true \MPC trajectory, whose size is controlled by the quality of the \ROM. Further, these estimates allow us to perform the \MPC optimization procedure with a computational cost that, up to a few model updates through \FOM-\MPC subproblems resolution, is independent of the size of the \FOM. Finally, we validate the theoretical findings with numerical experiments, emphasizing the contrast between expensive-to-evaluate estimates and their inexpensive-to-evaluate upper bounds.

In future work, we plan to integrate our error estimation procedure with balanced truncation approaches by incorporating a priori output error bounds. This should allow us to have sharper, inexpensive error estimates. We also want to extend our techniques with the concepts presented in \cite{grune2021performance,DieG23} to directly assess the \MPC performance in the economic and stabilizing setting. Finally, the \ROM and \MPC frameworks are well-suited to be coupled with contour integral methods (\CIM) \cite{GugLM21,GugM23} for efficient and accurate time integration; our goal is to efficiently combine the \MPC optimization with the \CIM formulation for \ROM. 
\section*{Acknowledgements} 
MK would like to thank Martin Grepl for helpful advice about the a-posteriori error estimation and Behzad Azmi for productive discussion about the non-smooth optimality conditions.
MK and MM acknowledge funding by the BMBF (grant no.~05M22VSA). The research of BU was mainly conducted while he was affiliated with the University of Stuttgart. BU acknowledges funding by the Deutsche Forschungsgemeinschaft (DFG, German Research Foundation) under Germany's Excellence Strategy - EXC 2075 - 390740016 and project-ID 258734477 -- SFB 1173. MM and BU acknowledge the support of the Stuttgart Center for Simulation Science (SimTech).

\section*{Declarations}


\subsection*{Code availability}
The code for the numerical experiments is available at \url{https://doi.org/10.5281/zenodo.14443259}.

\subsection*{Author contribution}
M.K. developed the main analytical ideas, conducted the numerical experiments, and prepared the initial manuscript. B.U. and M.M. contributed to the formatting and conceptualization of the main manuscript as well as the development of the core ideas. S.V. supervised and proofread the manuscript. B.U. and S.V. provided funding.

\bibliographystyle{plain-doi}
\bibliography{journalabbr.bib,biblio.bib}

\appendix
\section{Properties of the switched inner product}
\begin{lemma}[Integration by parts]\label{lemma:S-IBP}
    Let $\switch\in\addSwitching$ and $\state,\adState\in \X_{n,\switch}$. Let $\calT_n$ be the set of switching times as in \Cref{ass:switchedSystem}\,\ref{ass:switchedSystem:switching} and assume that the bilinear forms $m_i$ satisfy \Cref{ass:switchedSystem}\,\ref{ass:switchedSystem:M}. Then
    \begin{align}    \label{eqn:integrationByPartsM}
    \begin{aligned}
        \int_{\timeInt_n} m_{\switch}(\dot{\state},\adState)\dt &=  m_{\switch(T_n)}(\state(T_{n}),\adState(T_{n}))-m_{\switch(t_n)}(\state(t_{n}),\adState(t_{n}))\\
        &+\sum_{t_i\in \calT} \left(m_{\switch(t^-_i)}(\state(t^-_{i}),\adState(t^-_{i}))-m_{\switch(t^+_i)}(\state(t_{i}^+),\adState(t^+_{i}))\right) -\int_{\timeInt_n} m_{\switch}(\dot{\adState},\state)\dt .
        \end{aligned}
    \end{align}
\end{lemma}

\begin{proof}
   First, observe that if $m_{\switch}$ and the adjoint state are continuous, the sum in \eqref{eqn:integrationByPartsM} would be zero, thus reducing the statement to standard integration by parts. Splitting $\timeInt_n$ into its sub-intervals with constant mode and applying integration by parts on each sub-interval, we can write
    \begin{align*}
        \begin{aligned}
            \int_{\timeInt_n} m_{\switch}(\dot{\state},\adState)\dt\;=&\;\sum_{i=0}^{N-1}\int_{t_i}^{t_{i+1}} m_{\switch}(\dot{\state},\adState)\dt\\
            =&\;\sum_{i=0}^{N-1}\big( m_{\switch(t^-_{i+1})}(\adState(t^-_{i+1}),\state(t_{i+1}^-)) - m_{\switch(t^+_i)}(\adState(t^+_{i}),\state(t_{i}^+))\big)-\int_{\timeInt_n} m_{\switch}(\dot{\adState},\state)\dt
        \end{aligned}
    \end{align*}
   The claim follows by rearranging the sum and from the fact that $t_n$ and $T_n$ have no switching points (see \Cref{ass:switchedSystem}\,\ref{ass:switchedSystem:switching}).
\end{proof}
\ReA{
\section{Complexity analysis of the \MPC-schemes}\label{app:sec_complexity}
In this section, we compare the online complexity of the proposed \MPC-schemes (\FOM-\MPC (\Cref{alg:FOM-MPC}), \FOM-\ROM-\MPC (\Cref{alg:ROM-MPC}), and \ROM-\ROM-\MPC (\Cref{alg:ROMROM-MPC}) per \MPC iteration $n\in\N$, which we denote by $\costFOM,\costFOMROM,\costROMROM$, respectively. We restrict ourselves to evaluating the complexity as a function of (stationary) \FOM linear system solves, which is denoted by $\LS$, since the number of \LS (and the computation of the \POD, which is counted separately) is the most expensive part in each algorithm. Let $K_\delta,K_T\in \N$ be the number of discrete time steps corresponding to the sampling time $\delta$ and prediction horizon $T$, respectively.

The computational cost of the \FOM-\MPC per \MPC-step $n\in \N$ consists of solving the \FOM-\OCP, which in turn consists (in the setting described in \Cref{subsec:computationalSetup}) of applying the proximal gradient method leading to one gradient evaluation $\nabla F_n(u_k)$ for the iterate $u_k$ per inner iteration $k\in \N$, consisting of the simulation of the state and the adjoint state over $K_T$ steps (see \eqref{eq: state mpc subproblem gradient}). Assuming $k^*\in \N$ proximal gradient steps until convergence, we obtain
\begin{equation}
    \costFOM= k^*2K_T\LS.
\end{equation}
Before we consider the \ROM schemes, we consider the cost of updating the \ROM in \Cref{line:alg:ROMROMupdate}, which we call $\costUpdate$. Updating the \POD model consists of computing the \POD basis in \eqref{eqn:PODminimization}, which is done using singular value decomposition (\SVD). Applying the economic SVD on a snapshots matrix $\mathbf Y \in \R^{N\times K_s}$ for some the \FOM dimension $N\in\N$ and $K_s\in \N$ has a cost of $\costPOD\approx\calO(NK_s^2)$ if $N\geq K_s$, which is satisfied, since in the \MPC setting it holds $K_s=m_s K_T$, where $K_T$ is small and $m_s\in \N$ is also a small number indicating the number of snapshots of length $K_T$ used (in the setting of \Cref{subsec:computationalSetup} we have $m_s\leq14$, $K_T=20$, and $N=5304$). If the cheap error estimators $\tilde \Delta_A,\tilde \Delta_B$ are considered, an offline-online decomposition of the error estimator (see \cite{Hesthaven2016}) is considered, which depends on \LS (for computing the Rietz representatives per affine component of the residual) that scale with the dimension $r\in \N$ of the reduced space $\Vred$, the number of switching modes $\nrModes$ and the input and output dimensions $\inpVarDim,\outVarDim\in \N$. Hence, for the offline assembly of the state error estimator, we have $\nrModes\inpVarDim \LS$ for the operator $\calB_\switch$, and $2r\nrModes \LS$ for the operators $\calM_\switch$,  $\calA_\switch$ (see \eqref{eq:OPLSS}). For the offline assembly of the adjoint error estimator, we have $\nrModes\outVarDim \LS$ for the operator $\calC_\switch$, and $r\nrModes \LS$ for $\calA'_\switch$. If the expensive error estimators $\Delta_A,\Delta_B$ are considered, no offline-online decomposition is performed, leading to
\begin{equation}\label{eq:app:costupdate}
    \costUpdate =
    \begin{cases}
        \costPOD &\text{if }\Delta_u\in \{\Delta_A,\Delta_B\},\\
       \costPOD+\nrModes(\inpVarDim+\outVarDim+3r)\LS &\text{if }\Delta_u\in \{\tilde \Delta_A,\tilde \Delta_B\}.
	\end{cases}
\end{equation}
Note that, for simplicity, we neglect the cost of constructing the \ROM via projection onto $\tilde V$. After the model is constructed, the online cost for evaluating the error estimator is given as 
\begin{equation}\label{eq:app:costest}
    \costEst =
    \begin{cases}
        2K_T\LS &\text{if }\Delta_u\in \{\Delta_A,\Delta_B\},\\
       0 &\text{if }\Delta_u\in \{\tilde \Delta_A,\tilde \Delta_B\}.
	\end{cases}
\end{equation}
Thus, \eqref{eq:app:costupdate} shows that using the cheap estimators has a higher model update cost, while ensuring a cheaper evaluation after construction of the \ROM in \eqref{eq:app:costest}.

Next, turning to the \ROM-\MPC schemes, we have to distinguish whether the update criterion in \Cref{alg:line:ErrorEstCrit} of both algorithms is triggered or not. First, we start with the \FOM-\ROM-\MPC and assume that the  \Cref{alg:line:ErrorEstCrit} is not active. Using the cheap error estimators ($\tilde \Delta_A,\tilde \Delta_B$), $\costFOMROM$ consists of applying the reduced controller onto the \FOM system in \Cref{algo:line:applyROM} (see also \eqref{eq:FOMROMiter}) for $K_\delta$ time steps, leading to $\costFOMROM= \costEst + K_\delta\LS= K_\delta\LS$ (by using \eqref{eq:app:costest}). Conversely, using the expensive error estimators ($ \Delta_A, \Delta_B$) leads with \eqref{eq:app:costest} to a cost of $\costFOMROM= \costEst= 2K_T\LS$. In this case, there is no explicit cost for applying the reduced controller onto the full system, since for the evaluation of the expensive error estimator, we already compute the new \FOM state at time step $K_\delta$ (see the discussion at the beginning of \Cref{subsec:apostStateAdjoint}). If \Cref{alg:line:ErrorEstCrit} is active, then we add up the cost of error estimation, \FOM-\OCP solving, and model update, such that we obtain
\begin{equation}\label{eq:app:costFOMROM}
\costFOMROM= 
\begin{cases}
        \costEst + \costFOM + \costUpdate &\text{if \Cref{line:alg:ROMROMupdatemechasim} is active},\\
         \begin{cases}
        2K_T\LS &\text{if }\Delta_u\in \{\Delta_A,\Delta_B\},\\
        K_\delta\LS & \text{if } \Delta_u\in \{\tilde \Delta_A,\tilde \Delta_B\}.
	\end{cases} &\text{else}.
	\end{cases}
\end{equation}
Thus, the \FOM-\ROM-\MPC is never online-efficient, i.e., independent of \LS even after model construction.

Turning to \ROM-\ROM-\MPC, note that we always use cheap estimators ($\tilde \Delta_A,\tilde \Delta_B$) in this case. If \Cref{line:alg:ROMROMupdatemechasim} is not active, then \ROM-\ROM-\MPC is online-efficient, that is, the cost of the algorithm is independent of \FOM calculations, and thus $\costROMROM=0$. If \Cref{line:alg:ROMROMupdatemechasim} is active, we obtain the \FOM-\OCP solve together with the update of the \ROM consisting of the \POD together with an assembly of the error estimator. All in all, we obtain
\begin{equation}\label{eq:app:costROMROM}
\costROMROM=
\begin{cases}
        \costFOM + \costUpdate &\text{if \Cref{line:alg:ROMROMupdatemechasim} is active},\\
        0&\text{else}.
	\end{cases}
\end{equation}
All in all, one can expect speed-ups of the \ROM algorithms only if the update mechanism is triggered not in every step and if the number of \FOM optimizer steps $ k^*$ is sufficiently large ($ k^*\geq 2$). Comparing \eqref{eq:app:costROMROM} and \eqref{eq:app:costFOMROM}, one can expect speed-ups of the \ROM-\ROM-\MPC algorithm compared to the \FOM-\ROM-\MPC, if the overestimation of the cheap error estimators does not trigger the update mechanism much more often than the expensive error estimators.
}
\end{document}